\documentclass[11pt,hyp,]{nyjm}
\usepackage{amssymb,amsmath,amsthm}

\usepackage{graphicx,epstopdf,color}
\usepackage{esint}

\usepackage{mathtools} 
\usepackage{mathrsfs}
\mathtoolsset{showonlyrefs=true}

\usepackage{dsfont}
\usepackage{verbatim}
\usepackage{enumerate}
\numberwithin{equation}{section}

\allowdisplaybreaks
\usepackage[alphabetic, initials]{amsrefs}
\usepackage[latin1]{inputenc}

\usepackage{stackengine}

\usepackage{hyperref}
\hypersetup{nesting=true,debug=true,naturalnames=true}
\usepackage{graphicx,amssymb,upref}

\newcommand{\R}{\mathbb{R}}

\newcommand{\N}{\mathbb{N}}

\usepackage{miama}
\usepackage[T1]{fontenc}

\newtheorem{theorem}{Theorem}[section]
\newtheorem{corollary}[theorem]{Corollary}
\newtheorem{lemma}[theorem]{Lemma}
\newtheorem{proposition}[theorem]{Proposition}
\newtheorem{definition}[theorem]{Definition}

\theoremstyle{remark}
\newtheorem{remark}[theorem]{Remark}
\newtheorem{example}[theorem]{Example}

\newcommand{\ml}{\mathcal{L}}

\newcommand{\Int}{\int_{\mathbb{R}^n}}
\newcommand{\am}{\mathcal{A}}

\newcommand{\bs}{\Bar{s}}

\newcommand{\into}{\int_{\Omega}}

\DeclareMathOperator*{\esssup}{ess\,sup}
\makeatletter
\newsavebox\myboxA
\newsavebox\myboxB
\newlength\mylenA

\newcommand*\xoverline[2][0.75]{%
	\sbox{\myboxA}{$\m@th#2$}%
	\setbox\myboxB\null
	\ht\myboxB=\ht\myboxA%
	\dp\myboxB=\dp\myboxA%
	\wd\myboxB=#1\wd\myboxA
	\sbox\myboxB{$\m@th\overline{\copy\myboxB}$}
	\setlength\mylenA{\the\wd\myboxA}
	\addtolength\mylenA{-\the\wd\myboxB}%
	\ifdim\wd\myboxB<\wd\myboxA%
	\rlap{\hskip 0.5\mylenA\usebox\myboxB}{\usebox\myboxA}%
	\else
	\hskip -0.5\mylenA\rlap{\usebox\myboxA}{\hskip 0.5\mylenA\usebox\myboxB}%
	\fi}
\makeatother

\let\<\langle
\let\>\rangle

\let\uml\"

\renewcommand{\leq}{\leqslant}
\renewcommand{\le}{\leqslant}
\renewcommand{\geq}{\geqslant}
\renewcommand{\ge}{\geqslant}

\renewcommand{\epsilon}{\varepsilon}

\providecommand{\keywords}[1]{\textbf{\textit{Keywords---}} #1}

\providecommand{\subjclass}[1]{\textbf{\textit{2020 Mathematics Subject Classification---}} #1}

\title[Fredholm alternative for nonlocal operators]{Fredholm alternative for a general class\\ of nonlocal operators}

\author{Francesco De Pas}  
\address{The University of Western Australia, 35 Stirling Highway,
  Crawley WA 6009, Australia} 
\email{francesco.depas@research.uwa.edu.au}  
\author{Serena Dipierro}  
\address{The University of Western Australia, 35 Stirling Highway,
  Crawley WA 6009, Australia} 
\email{serena.dipierro@uwa.edu.au}  
\author{Enrico Valdinoci}  
\address{The University of Western Australia, 35 Stirling Highway,
  Crawley WA 6009, Australia} 
\email{enrico.valdinoci@uwa.edu.au}  
\thanks{Supported by the Australian Research Council Laureate Fellowship FL190100081 and by the Australian Future Fellowship FT230100333.}


\begin{document}

\begin{abstract}
We develop a Fredholm alternative for a fractional elliptic operator~$\mathcal{L}$ of mixed order built on the notion of fractional gradient. This operator constitutes the nonlocal extension of the classical second order elliptic operators with measurable coefficients treated by Neil Trudinger in~\cite{trudinger}. We build~$\mathcal{L}$ by weighing the order~$s$ of the fractional gradient over a measure
(which can be either continuous, or discrete, or of mixed type). The coefficients of~$\mathcal{L}$ may also depend on~$s$, giving this operator a possibly non-homogeneous structure with variable exponent.
These coefficients can also be either unbounded, or discontinuous, or both.

A suitable functional analytic framework is introduced and investigated and our main results strongly rely on some custom analysis of appropriate functional spaces. 
\end{abstract}

\maketitle
\tableofcontents

\section{Introduction}\label{intro}

A classical problem in many areas of mathematics is determining whether a given equation possesses a solution and, if so, whether the solution is unique. In some cases, physical constraints introduce restrictions, known as ``resonances'', in which
specific parameter configurations either prevent solutions from existing or lead to their high multiplicity.

The case of finite-dimensional linear equations is perhaps the simplest to analyze, yet it already illustrates the potential effects of resonances caused by parameter choices. For example, consider a matrix~\(A\), a scalar~\(\lambda\), and a vector~\(b\). Elementary linear algebra examines the number of solutions \(x\) for the vector equation
\[
Ax = \lambda x + b.
\]
Two complementary scenarios arise:
\begin{itemize}
    \item\(\lambda\) is an eigenvalue of~\(A\) (resonance): The equation is solvable only if~\(b\) is orthogonal to the null space of~\((A - \lambda I)^*\), where~\(I\) is the identity matrix and the superscript~\(*\) denotes the conjugate transpose (or transpose in the real case). When this compatibility condition holds, the equation has infinitely many solutions.
    \item  \(\lambda\) is not an eigenvalue of~\(A\) (non resonance): The matrix~\(A - \lambda I\) is invertible, and the equation has a unique solution, given by~$x = (A - \lambda I)^{-1}b$.
\end{itemize}

 This simple example illustrates the interplay between resonances, compatibility conditions and solution structure, serving as a finite dimensional model to understand more complex scenarios.

The Fredholm Alternative extends these ideas to infinite-dimensional spaces, addressing linear equations involving compact operators and providing a dichotomy analogous to the finite-dimensional case. As such, the Fredholm Alternative has become a cornerstone of both abstract functional analysis and applied fields such as quantum mechanics, fluid dynamics and linear partial differential equations.

 In the context of  elliptic partial differential equations, a thorough description of this dichotomy was introduced by Neil Trudinger in~\cite{trudinger}. This framework allowed for very general linear equations, including lower-order terms and with only minimal regularity requirements on the coefficients involved.

In this article, we extend the Fredholm Alternative to nonlocal operators of fractional type. Relying on the notion of fractional gradient, we develop a framework for inhomogeneous equations and address cases involving superposition of operators of different fractional orders, including infinite sums. This extension is particularly relevant for applications, as these operators model phenomena such as biological species in which individuals exhibit diffusive patterns characterized by different L\'evy exponents. 
In this regard, the distribution of the measure weighing operators of different orders describes the proportion of a biological population adopting a certain dispersal strategy (e.g., for breeding or foraging purposes).
Moreover, this approach offers a ``unified'' treatment of both classical and nonlocal cases, with the former emerging as a specific instance of the broader theory.

The following paragraphs outline the classical setting, introduce the formalism required for nonlocal operators, and present our main results.

\subsection{The operator under consideration}\label{opISPJDL}
In~\cite{trudinger},
Neil Trudinger developed a Fredholm alternative for an elliptic operator~$\mathcal{T}$ of the form
\begin{equation}\label{trudingi}
	\mathcal{T} u:=  -\frac{\partial}{\partial x_i} \left( a^{ij}(x) u_{x_j} + a^i(x) u \right) + b^i(x) u_{x_i} + a(x) u, 
\end{equation}
whose coefficients are measurable functions on a bounded domain~$\Omega \subset \R^n$
(according to custom, here above and in the rest of the paper, the repeated index summation
convention is employed). The goal of this work is to extend this theory to a more general operator~$\mathcal{L}$ of nonlocal nature.

To this end, we recall that the \textit{fractional gradient} operator of order~$ s\in (0,1)$ can be defined as 
\begin{equation}\label{dcw}
D^s u (x) := c_{s,n} \Int \frac{(u(y)-u(x))}{|x-y|^{n+s+1}} (y-x) \, dx,
\end{equation}
where~$c_{s,n}$ is a normalizing constant vanishing as~$s$ approaches~$1$ (see Section~\ref{gitti} below). We will  denote the~$i^{th}$ component of the vector~$D^su$
by~$D_i^s u$. 

According to~\cite[Section~1]{silhavi},
the first appearance of the fractional gradient dates back to the papers~\cite{MR107788, MR500133}.
The operator in~\eqref{dcw} has been treated in the recent literature (see e.g.~\cite{MR3420498, shieh, MR3615452, MR3714833, shiehtwo, comi, silhavi, piola, mora}), though a complete understanding of its rather complex behavior is still under development, and plays an important role in the definition of nonlocal
counterparts of classical elliptic operators. 

To be consistent with the classical case,
it is customary (see e.g.~\cite{silhavi}) to extend the setting
in~\eqref{dcw} to the case~$s=1$, by taking~$D^1u$
as the classical gradient of~$u$, here denoted by~$Du$. This choice will be formally justified in Proposition~\ref{gra}
below. In general, we refer the reader to Section~\ref{gitti} for a self-contained introduction to the operator in~\eqref{dcw}.

Now, let~$\mu(s)$ be a nonnegative measure on~$[0,1]$, whose support is bounded away from~$0$. We formally define the operator~$\mathcal{L}$ as follows:
			\begin{equation}\label{rawop}
	\ml u := \int_{(0,1]} \Big( - D^s_i \left(a^{ij}(s,x) D^s_j u + a^i(s,x) u \right) + b^i(s,x) D^s_i u \Big) \, d\mu(s) + a(x) u.
\end{equation}
We stress that if~$\mu(s)$ is a Dirac delta at~$s=1$, than the operators~$\mathcal{T}$ and~$\mathcal{L}$ coincide. Furthermore, $\mu(s)$ can be either continuous, or discrete, or of mixed type, giving~$\mathcal{L}$ the nature of a fractional differential operator of mixed order (see Section~\ref{application} for some practical examples).

{W}e remark that the study of operators of mixed order is important both from the theoretical perspective and in view of concrete applications: indeed, on the one side, these operators often pose challenging theoretical questions due to their lack of scale invariance, and, on the other side, they appear naturally
in biological models, since animals of different species, and also different individuals of the same species,
in many instances exhibit different diffusive patterns of fractional type with different L\'evy exponents, see e.g.~\cite{MR4249816}.
The operator considered in~\eqref{rawop} is however structurally different than several instances already studied in the literature that dealt with general L\'evy measures, since
here we aim at capturing the salient features encoded specifically
by a ``divergence-type'' design arising from spatial inhomogeneity (but, due to the complicated 
structure of the equation under consideration, the development of the theory cannot simply rely on variations of
the classical second order theory of elliptic operators in divergence form).

{F}rom the technical standpoint, one of the main issues is
the construction of a variational framework for our Fredholm alternative.
Since the fractional gradient appears in~$\mathcal{L}$, we are led to consider suitable Bessel-type spaces~$H_0^{s,p}(\R^n)$, endowed with the norm
\begin{equation*}
 \Vert u \Vert_{H_0^{s,p}(\R^n)}:=\left( \Vert  u \Vert^p_{L^p(\R^n)}+ \Vert D^s u \Vert^p_{L^p(\R^n)}\right)^{\frac1p} .
\end{equation*}
These spaces have been extensively treated in the literature (see e.g.~\cite{shieh, comi, piola}),
nevertheless some particular features emerge regarding~$\mathcal{L}$ which will require some bespoke arguments.

Indeed, first of all, $\mathcal{L}$ weighs the fractional parameter~$s$ over a possibly continuous measure, making it difficult, for a given~$\bar{s} \in (0,1]$, to make~$D^{\bar{s}}$ appear explicitly in our operator
(notice that, in the generality that we consider, it can well be that~$\mu\{ \bar{s} \}=0$).

In addition, $\mathcal{L}$ depends on coefficients that are functions of~$s$ and~$x$. In particular, none of these coefficients are required to be bounded and it is always possible for them to be either discontinuous, or unbounded, or both, and the operator~$\mathcal{L}$ may not have a specific order of differentiation.

Also, Fredholm alternatives in their simplest formulations are often set in Hilbert spaces: however, in our case,
asking for~$\mathcal{L}$ to be well-defined in~$H_0^{s,2}(\R^n)$ would force us to impose stronger regularity assumptions on the matrix~$[a^{ij}]$, as we will show in Section~\ref{ipo}. To avoid this additional restriction, a suitable functional analytic framework will be
presented and investigated: in particular, a thorough analysis of the operator~$\mathcal{L}$ will hinge on some bespoke analysis of appropriate functional spaces.
 
\subsection{Notations}

In this article we make use of the following notations.

\begin{itemize}
\item~For any~$\Omega\subseteq\R^n$, $\mathcal{D}(\Omega)$ refers to the set of~$C^{\infty}$ functions with compact support in~$\Omega$. Also, as customary, the Schwartz space~$\mathcal{S}(\R^n)$ denotes the space of all smooth functions whose derivatives are rapidly decreasing.
\item~The symbols~$\omega_n$ and~$ S_n$ refer, respectively, to the volume and the surface of the unit ball in~$\R^n$. Moreover,~$B_R$ denotes the open ball of radius~$R>0$ centered at the origin.
\item~Given an open and bounded set~$\Omega\subset\R^n$, we suppose that~$h:\Omega\to[0,+\infty)$ is a measurable function, such that~$h^{-1} \in L^t(\Omega)$ for some~$t \in[1,+\infty]$. Then, we refer to~$L^2(h, \Omega)$ as the Hilbert space induced by the inner product
           \begin{equation*}
           	\langle u,v \rangle_{L^2(h, \Omega)} := \int_{\Omega} h(x) u (x)v(x) \, dx.
           \end{equation*}
In particular, we refer to Remark~\ref{vozzi} to show that this space is not empty.
\item~Given a Banach space $V$, we denote\footnote{ For typographical convenience, on some occasions the dual space will be denoted by~$(V)\rq{}$ instead of simply~$V\rq{}$: this occurs for instance for spaces such as~$H^{s,p}(\R^n)$ or~$ H^{s,p}_0(\Omega) $, whose dual space will be denoted, respectivly, by~$\left(H^{s,p}(\R^n)\right)\rq{}$ and~$ \left(  H^{s,p}_0(\Omega) \right) \rq{}$.} its dual space by~$V\rq{}$.
For any~$P \in V\rq{}$,~$Q \in V\rq{}$ and~$x \in V$, the notation~$\langle P, x \rangle $ \label{miwfh:302rtojgFWSa}
denotes the application of~$P$ to~$x$, while the equality~${P=Q \mbox{ in } V\rq{}}$ means
\[ \langle P, x \rangle = \langle Q, x \rangle \] for all $x\in V$.

Consistently with this setting,
and as usual in the literature, the space of tempered distributions~$ \mathcal{S}\rq{}(\R^n)$
consists of the dual of  the Schwartz space~$\mathcal{S}(\R^n)$.

\item~Given~$u \in \mathcal{S}\rq{}(\R^n)$, we denote by either~$\widehat u$ or~$\mathcal{F}(u)$ the
Fourier transform of~$u$, intended in a distributional sense\footnote{  The reason for which
we have two notations to denote the same object is merely for typographical convenience.
For instance, in the setting of the forthcoming Lemma~\ref{fr}, the notation
$\widehat u$ reads better than~$\mathcal{F}(u)$,
while~$\mathcal{F}(I_{\alpha}u)$ is clearer than~$\widehat{I_{\alpha}u}$}.
\item~Regarding the operator $\mathcal{L}$, we denote by~$\mathcal{A}=[a^{ij}]$ the coefficients matrix and by~$\mathcal{A}_S=[a_S^{ij}]$ its symmetric part.
The matrix~$\mathcal{B}=[b^{ij}]$ is the inverse of~$\mathcal{A}$.
Also,~$\mathcal{L}^*$  is the formal dual of~$\mathcal{L}$, namely 
\begin{equation*}
           \mathcal{L}^*u := \int_{(0,1]}  \Big( -D^s_i \left(a^{ji}(s,x)D^s_j u + b^i(s,x)u \right) + a^i(s,x) D_i^s u \Big) \, d\mu(s) + a(x) u.
\end{equation*}
\end{itemize}

\subsection{Hypotheses on~$\mathcal{L}$}\label{hyp_section}

This paragraph collects the assumptions we make on the operator~$\mathcal{L}$ defined in~\eqref{rawop}. These hypotheses are meant to hold throughout the whole article.

We ask~$\mu$ to be a~$\sigma$-finite and nonnegative measure on~$[0,1]$, whose support is bounded away from~$0$.

In addition, we require~$\mathcal{A}$,~$a^i(s,x)$ and~$b^i(s,x)$ to be measurable functions in~$[0,1]\times \R^n$.

Then, we ask~$\mathcal{A}$ to be positive definite for every~$s \in (0,1]$. In addition, we require that there exist~$\lambda$ , $\Lambda:\R^n\to[0,+\infty)$ such that, for any~$x$, $\xi \in \R^n$ and~$s \in (0,1]$, 
\begin{equation}\label{essa}
\lambda(x) |\xi|^2 \leq \xi^T \mathcal{A} (s,x)\,\xi \leq \Lambda(x) |\xi|^2.
\end{equation}
The quantities~$\lambda$ and~$\Lambda$ can be seen as
ellipticity bounds on~$\mathcal{A}$ (and on~$\mathcal{A}_S$). We stress that these elliptic bounds are not necessarily uniform; in fact, we only assume that there exist~$R>0$, $C>0$, $\delta>0$ and~$p<n$ such that
\begin{equation}\label{a_condi}
\begin{cases}
&\Lambda \in L^1(B_R) \quad\mbox{and}\quad \Lambda(x) \leq C |x|^{p} \quad\mbox{for any}\quad x \in \R^n \setminus B_R, \\
&\lambda^{-1}\in L^{1+\delta}_{{\rm{loc}}}(\R^n).
\end{cases}
\end{equation}  
Roughly speaking, the role of~$B_R$ in~\eqref{a_condi}
is to prescribe~$\Lambda$ to be
integrable in a neighborhood of the origin and grow at most like~$|x|^p$ at infinity. The structural
parameter~$R$ does not play a major role in this paper (in particular, we do not
need to take it sufficiently large with respect to the size of~$\Omega$).

Then, we suppose that there exists a constant~$\mathcal{K}_{\am}>0$ satisfying, for any~$x$, $\xi$, $\psi \in \R^n$ and~$s \in (0,1]$,
			\begin{equation}\label{vwm}
				\left|\xi^T \mathcal{A}(s,x) \psi\right|^2 \leq \mathcal{K}_{\am} (\xi^T \mathcal{A}(s,x)\xi)(\psi^T \mathcal{A}(s,x) \psi).
			\end{equation}
We remark that condition~\eqref{vwm} is fairly general (see e.g.
Lemma~\ref{lemon} for sufficients conditions to satisfy~\eqref{vwm}).

Concerning~$\mathcal{B}(s,x)$, $a^i(s,x)$ and~$b^i(s,x)$, we assume that there exist~$\mathcal{B}(x)=[b^{ij}(x)]$, $a^i(x)$ and~$b^i(x)$ such that, for any~$i,j=1,\ldots,n$, any~$s \in (0,1]$ and any~$x \in \R^n$, 
\begin{equation}\label{fds}
|b^{ij}(s,x)|\leq b^{ij}(x), \quad |a^i(s,x)| \leq a^i(x) \quad\mbox{and}\quad |b^i(s,x)| \leq b^i(x).
\end{equation}

\subsection{Main results}
To cope with the functional analytic difficulties mentioned in Section \ref{opISPJDL}, we set the following Hilbert spaces (see Section~\ref{ipo} for a formal introduction).

Let~$\Omega$ be a bounded domain of $\R^n$. For any~$g \in L^1(\Omega,[0,+\infty))$, we define the following scalar product in $\mathcal{D}(\Omega)$ (see Lemma~\ref{wieg})
           \begin{equation}\label{saojdcvn-435tyhnfwegbX}\begin{split}
           	\langle u, v \rangle_{H^0(\mathcal{A}, g, \Omega)} : =\;& \int_{\R^n}  \int_{(0,1]} a^{ij}_S(s,x) D_i^su (x)D^s_jv(x) \, d\mu(s) \, dx \\&\qquad+ \int_{\Omega} g(x) u(x) v(x) \, dx.\end{split}
           \end{equation}
We denote by $\|\cdot\|_{H^0(\mathcal{A}, g, \Omega)}$ the norm induced by this scalar product.

\begin{definition}\label{qeimsot}
Let~$\Omega$ be a bounded domain of $\R^n$ and~$g: \Omega \to [0, +\infty)$ be a measurable function. 
We define
\begin{equation}\label{NVLDDMDsC}
H^0(\mathcal{A}, g, \Omega) := \overline{ \mathcal{D}(\Omega) }^{\Vert \cdot \Vert_{H^0(\mathcal{A}, g, \Omega)}}.\end{equation}
\end{definition}

To ease the notation, when~$g=0$, we denote the Hilbert space in~\eqref{NVLDDMDsC}
as~$H^0(\mathcal{A}, \Omega)$.

A fundamental role will be played by the notion of compact boundedness:

\begin{definition}\label{qeimsot2}
We say that a nonnegative measurable function~$f$ is \textit{compactly bounded} on~$H^0(\am, g, \Omega )$ if, for any~$\epsilon>0$, there exists~$K_{\epsilon}\ge0$ such that, for any~$\phi \in \mathcal{D}(\Omega)$,
           \begin{equation}\label{ghhg}
           	\Vert \phi \Vert_{L^2(f, \Omega)}^2 \leq \epsilon \|\phi \|^2_{H^0(\am, g, \Omega )} + K_{\epsilon} \Vert \phi \Vert_{L^1(\Omega)}^2.
           \end{equation}
\end{definition}

\begin{remark}
The notion of compact boundedness plays a relevant role in our framework, since it allows us to retrieve a compact embedding for the space~$H^0(\am, \Omega)$ (see Theorem~\ref{trente}) which in turn is necessary in order to apply the Fredholm alternative. 
This assumption is also quite general (see e.g. Theorems~\ref{mcza} and~\ref{zega2}
for sufficient conditions in~$L^p$-class guaranteeing compact boundedness).
\end{remark}
With Definition~\ref{qeimsot2} in hand, we set
\begin{equation}\label{efe}
           f(x) := b^{ij}(x) \Big( a^i(x) a^j(x)+ b^i(x) b^j(x)\Big) + |a(x)|
\end{equation} 
and we ask~$f$ to be compactly bounded in~$H^0(\mathcal{A}, \Omega)$. We also observe that, $\mathcal{A}$ being positive definite, so it is $\mathcal{B}$, and therefore $f$ is nonnegative.

In this work we establish the Fredholm alternative in a weak framework for $\mathcal{L}$, i.e. we deal with the following bilinear form, defined on $H^0(\mathcal{A}, \Omega) \times  H^0(\mathcal{A}, \Omega)$,
\begin{equation}\label{mainop}
\begin{split}&
\left(\mathcal{L}u, v \, \right) \\&
:= \int_{\R^n} \left( \int_{(0,1]} \Big(a^{ij}(s,x) D^s_ju D^s_i v + a^i(s,x) u D^s_i v + b^i(s,x)v D^s_i u \Big)\, d\mu(s)  \right)\, dx \\
&\quad\quad+ \int_{\Omega} a(x) u v \, dx.
\end{split} 
\end{equation}
The fact that \eqref{mainop}
constitutes a meaningful variational formulation of the operator $\mathcal{L}$ in~\eqref{rawop} is justified in Theorem~\ref{hdhdhd}. In particular, we will show in Section~\ref{mainres} that this bilinear form is continuous and weakly coercive in~$H^0(\mathcal{A}, \Omega)$.

Then, for any~$u$, $v \in H^0(\mathcal{A}, \Omega)$ and~$\sigma \in \R$, we set
\begin{equation}\label{jnh}\begin{split}&
( \mathcal{I}(f) u, v ) := \langle u,v \rangle_{L^2(f, \Omega)}, \quad \mathcal{L}_{\sigma}(f) := \mathcal{L}+ \sigma  \mathcal{I}(f) \\&\mbox{and}\quad \mathcal{L}^*_{\sigma}(f) := \mathcal{L}^*+ \sigma  \mathcal{I}(f).\end{split}
\end{equation}

{W}e aim at solving the following \textit{eigenvalue problem}: given~$\sigma \in \R$ and~$T \in \left(H^0(\mathcal{A}, \Omega )\right)\rq{}$, we wish to find a function~$u \in H^0(\mathcal{A}, \Omega)$ such that
\begin{equation}\label{probbo}
 \mathcal{L}_{\sigma}(f) u = T \quad\mbox{in}\quad \left(H^0(\mathcal{A}, \Omega )\right)\rq{}.
\end{equation}

To this end, we establish the following result:

\begin{theorem}\label{provv}
Let~$f$ be as defined in~\eqref{efe} and suppose that it is compactly bounded on~$H^0(\am, \Omega)$.

Then, there exists a countable isolated set~$\Sigma$ of eigenvalues of~$\mathcal{L}$.

In particular, if~$\sigma \notin \Sigma$, then problem~\eqref{probbo} admits a unique solution~$u$.
If instead~$\sigma \in \Sigma$, then either it admits no solutions or it admits infinitely many solutions.  
\end{theorem}

Theorem~\ref{provv} is an immediate consequence of the more detailed result stated below.
          
\begin{theorem}\label{freddi}
Let~$f$ be as defined in~\eqref{efe} and suppose that it is compactly bounded on~$H^0(\am, \Omega)$.

Then, there exist~$\sigma_0\in\R$ and a countable isolated set of real numbers~$\Sigma \subset (- \infty,\sigma_0 )$ such that, for any~${\sigma} \notin \Sigma$, we have that~$ \mathcal{L}_{{\sigma}}(f)$ is a bijective mapping from~$H^0(\am, \Omega)$ to~$\left( H^0(\am, \Omega) \right)\rq{}$. 

If instead~${\sigma} \in \Sigma$, then the kernels of~$ \mathcal{L}_{{\sigma}}(f)$ and~$ \mathcal{L}^*_{{\sigma}}(f)$ coincide and are of positive, finite dimension.

In particular, problem~\eqref{probbo} admits a solution~$u \in H^0(\mathcal{A}, \Omega)$ if and only if, for any~$u^*\in H^0(\mathcal{A}, \Omega)$ which is a solution of
\begin{equation*}
\mathcal{L}^*_{\sigma}(f) u = 0 \mbox{ in } \left( H^0(\mathcal{A}, \Omega)\right)\rq{},
\end{equation*}
it holds that~$\langle T, u^* \rangle =0$.
\end{theorem}

{W}e recall that, consistently with the setting
introduced on page~\pageref{miwfh:302rtojgFWSa},
the notation~$\langle T, u^* \rangle $ denotes the application of~$T\in\left(H^0(\mathcal{A}, \Omega )\right)\rq{}$ to~$u^*\in H^0(\mathcal{A}, \Omega )$. Also,
in our framework,
the value~$\sigma_0$ mentioned in Theorem~\ref{freddi} plays the role of a coercivity constant for the operator~$\mathcal{L}_{\sigma}(f)$ (as it will be specified in Proposition~\ref{ghi}).  
\begin{remark} \label{remi}
We remark that, if~$\mu(\{1\})>0$, then Theorem~\ref{provv} and Theorem~\ref{freddi} hold true even relaxing condition~\eqref{a_condi}: in particular, 
if~$\mu(\{1\})>0$, our results remain valid even when~$\delta=0$
in~\eqref{a_condi}. This is a useful observation, since it allows us to retrieve also the
classical result for~$\mathcal{T}$ in~\cite{trudinger}, which only asks for~$ \lambda^{-1} \in L^1(\Omega)$. A formal proof for this statement is provided at the end of Section~\ref{mainres}.  
\end{remark}

{T}he rest of this article is organized as follows. We complete this introduction by providing some concrete examples of the operator~$\mathcal{L}$ to which Theorems~\ref{provv} and~\ref{freddi} apply.

Then, in Section~\ref{prelimin} we provide a self-contained introduction to the fractional gradient~$D^s $ (highlighting its connection
with the Riesz potential).

Section~\ref{asd} is devoted to some pivotal embeddings and inequalities for the Bessel-type space~$H^{s,p}_0(\Omega)$. The main focus here is on controlling
explicitly the dependence on the fractional parameter~$s$.

In Sections~\ref{ipo} and~\ref{bounded} we present the function space~$H^0(\mathcal{A},g,\Omega)$ and the concept of boundedness in this space. Finally, Section~\ref{mainres} is devoted to the proofs of Theorem~\ref{freddi} and the statement in Remark~\ref{remi}.

This paper also contains four appendices.  Appendix~\ref{sappend} contains sufficient conditions in~$L^p$-class to guarantee compact boundedness on $H^0(\mathcal{A},g,\Omega)$, while in Appendix~\ref{counterexa} we provide an example of function
that is bounded  but not compactly bounded on~$H^0(\mathcal{A},\Omega)$.
Appendix~\ref{tappend} is devoted to calculate some integrals and is used in the computation of the Fourier Transform of the fractional gradient carried out in Proposition~\ref{huy}. Finally, Appendix~\ref{fappend} contains some technical results regarding the matrix $\mathcal{A}$ and is extensively used in Section~\ref{mainres}.

\subsection{Examples and applications}\label{application}

This section presents some paradigmatic examples for the operator~$\mathcal{L}$ that we consider in this work.

\begin{example}
{I}f we take~$\mu:= \delta(1)$, then the operator~$\mathcal{L}$ in~\eqref{rawop}
boils down to the operator~$\mathcal{T}$ in~\eqref{trudingi}, that has been taken into account by Trudinger in~\cite{trudinger}. \end{example}

\begin{example}
Furthermore, we can take~$N>0$, $s_k \in (0,1]$ for any~$k=1, \ldots, N$, with~$s_N=1$, and~$\varsigma\in[0,+\infty)$
and set~$\mathcal{A}:=I$, $a^i=b^i=a=0$ and 
\[ \mu(s) := \sum_{k=1}^{N-1} \delta(s_k) + \varsigma \delta(s_N).\]
In this way,
\[ \mathcal{L}u= \sum_{k=1}^{N-1} (- \Delta)^{s_k}u + \varsigma \mathcal{T}u. \]

We stress that, when~$\varsigma\ne0$, this operator gathers both fractional and classical contributions.
This is an interesting property from the theoretical point of view, since it allows the treatment of structurally different operators via a unified method, and also in terms of concrete applications
(e.g., in the L\'evy flight foraging hypothesis it is customary to compare individuals
performing Gaussian and L\'evy dispersal strategies, in epidemic managements one may have to consider the coexistence of local and global lockdowns, etc.).\end{example}

\begin{example}
We can also consider more general operators, taking into account a measure~$\mu$ given by an infinite sum of Dirac deltas and letting~$\mathcal{A}(s,x) = \mathcal{A}(s)$
and~$a^i=b^i=a=0$. For this purpose, we consider a sequence~$(c_k)$ of nonnegative real numbers such that
\[ \sum_{k=1}^{+\infty} c_k < +\infty\]  and we define
\[\mu(s):= \sum_{k=2}^{+\infty} c_k\,\delta\left(1 - \frac{1}{k}\right) . \]
Then, we obtain the non-homogeneous operator
\[ \mathcal{L}u= -\sum_{k=2}^{+\infty} c_k \,  a^{ij} \left(1 - \frac{1}{k}\right) D_i^{1 - \frac{1}{k}} D_j^{1 - \frac{1}{k}}u \,. \]\end{example}

\begin{example}
While the previous examples deal with a discrete measure, it is also possible to consider a continuous contribution. For this, let~$\phi \in L^1(0,1)$ be positive and vanishing in a right neighbourhood of~$0$ and suppose that
\[ d\mu(s) = \phi(s)\,ds.\]
Then,
\[ \mathcal{L} u = \int_0^1  \left[ - D^s_i \left(a^{ij} D^s_j u + a^i u \right) + b^i D^s_i u \right] \phi(s) \, ds+ a u.\]

We point out that, when~$a^i=b^i=a=0$, the function~$f$ defined in~\eqref{efe} is null and thus it is compactly bounded in~$H^0(\mathcal{A}, \Omega)$. When it is not the case, instead, in order to apply Theorems~\ref{provv} and~\ref{freddi}, we must check the compact boundedness of~$f$ in~$H^0(\mathcal{A}, \Omega)$ (sufficient conditions are provided in Theorems~\ref{mcza} and~\ref{zega2}).  \end{example}

\section{A glance at the fractional gradient}\label{prelimin}

This section provides a self-contained introduction to the fractional gradient and its
connection with the Riesz potential. 

The results shown in this section are already present, at least in some form, in the literature,
but we provide here a simple and self-contained introduction to the basics of the fractional gradient and we extend the known results to the level of generality needed for our goals.
		
\subsection{The fractional gradient}\label{gitti}		
Let~$n$ denote the space dimension. For any~$s \in [-1,1)$, we set 
			\begin{equation}\label{constant}
				c_{s,n} := \frac{2^{s} \pi^{-\frac{n}{2}} \Gamma(\frac{n+s+1}{2})}{\Gamma(\frac{1-s}{2})},    
			\end{equation}
			where~$\Gamma$ is the Euler Gamma function. Since in this work the dimension~$n$ is given once and for all, from now on, we will refer to this constant simply as~$c_s$. 
			
We point out that the range~$s \in [-1,1)$ for the definition of~$c_s$ is needed when discussing
the fractional Fundamental Theorem of Calculus in~\eqref{yu} and Corollary~\ref{aar}.

			We mention here that the following results hold true for~$c_s$ (see~\cite[Lemma~2.4]{mora}):
			\begin{equation}\label{constdef}
\sup_{s \in [-1,1)} \frac{c_s}{(1-s)} < + \infty  \end{equation} and
\begin{equation}\label{gf}
 \lim_{s \nearrow 1} \frac{c_s}{(1-s)} = \frac{1}{\omega_n}. 
			\end{equation}

In line with~\cite{comi, silhavi}, we define the fractional gradient as follows.
\begin{definition}\label{curate}
Let $s \in (0,1]$ and $u \in \mathcal{D}(\R^n)$. We define the fractional gradient of $u$ as
			\begin{equation}\label{def}
				D^s u (x) := 
				\begin{cases}\displaystyle
					c_{s} \,\lim_{\epsilon \searrow 0} \displaystyle\int_{\R^n \setminus B_{\epsilon}} \frac{z u(x+z)}{|z|^{n+s+1}} \, dz & \mbox{ if } s \in(0,1), \\
					Du(x) & \mbox{ if } s=1.
				\end{cases}
			\end{equation}
\end{definition}
In Lemma~\ref{kj} here below we show that, for any $u \in\mathcal{D}(\R^n)$ with compact support, $D^s u$ is defined pointwise for any $s \in (0,1)$. Moreover, to justify this definition for~$s=1$, we will show in Proposition~\ref{gra} below
that the fractional gradient~$D^su$ converges to the classical gradient~$Du$, as~$s \nearrow 1$.

{I}n line with~\cite[Definition 2.2]{mora}, we now extend the definition in~\eqref{def} to a broader class of functions.
\begin{definition}\label{defini}
Let $s \in (0,1]$ and $p \in [1, +\infty)$.
Let $u \in L^p(\R^n)$ be such that there exists a sequence $(u_k) \subset \mathcal{D}(\R^n)$ converging to $u$ in $L^p(\R^n)$ as~$k\to+\infty$ and for which $(D^s u_k)$ is a Cauchy sequence in $L^p(\R^n, \R^n)$. 
Then, we define $D^s u$ as the limit in $L^p(\R^n, \R^n)$ of $D^s u_k$ as $k \to+ \infty$.
\end{definition} 
\begin{remark}
Definition~\ref{defini} depends neither on the sequence $(u_k)$ nor on the exponent $p$ (see \cite[Lemma 2.3]{mora}). Also, for any $u \in \mathcal{D}(\R^n)$, Definitions~\ref{curate} and~\ref{defini} coincide.
\end{remark}

We note that, by odd symmetry, for any~$\epsilon>0$,
			\begin{equation}\label{simzeromenomale09}
				\int_{\R^n \setminus B_{\epsilon}} \frac{z }{|z|^{n+s+1}} \, dz = 0.
			\end{equation}
Hence, applying the change of variable~$y:=z+x$, we obtain an equivalent definition of the fractional gradient when~$s \in (0,1)$, that is
			\begin{eqnarray*}
				D^s u (x) &=& c_{s}  \,\lim_{\epsilon \searrow 0} \int_{\R^n \setminus B_{\epsilon}} \frac{z (u(x+z)-u(x))}{|z|^{n+s+1}} \, dz  \\
				&=& c_{s}  \,\lim_{\epsilon \searrow 0} \int_{\R^n \setminus B_{\epsilon}(x)} \frac{(y-x) (u(y)-u(x))}{|x-y|^{n+s+1}} \, dy.
			\end{eqnarray*}
			
			\begin{lemma}\label{kj} Let~$s\in(0,1)$ and~$u \in \mathcal{D}(\R^n)$. Let~$\overline R >0$ be such that~$B_{\overline R}$ contains the support of~$u$.
				
				Then, for any~$x_0\in\R^N$,
				$x\in B_1(x_0)$ and~$R\geq |x_0|+ \overline R+1$, we have that
				\begin{equation}\begin{split}\label{hu43ty854b97689v6cb}
					D^su(x) &= c_s \int_{B_{R}}  \frac{z(u(x+z)-u(x))}{|z|^{n+s+1}} \, dz \\ &= c_s \int_{B_{R}(x)}  \frac{(y-x)(u(y)-u(x))}{|x-y|^{n+s+1}} \, dy.
				\end{split}\end{equation}				 
			\end{lemma}
			
			\begin{proof} We use the short notation~$R_0:=|x_0|+ \overline R+1$. Let now~$R\ge R_0$.
			We notice that, for every~$x\in B_1(x_0)$ and~$z\in\R^n\setminus B_R$,
			$$|x+z|\ge|z|-|x-x_0|-|x_0|>R-1-|x_0|\ge R_0-1-|x_0|=\overline R$$
			and accordingly~$u(x+z)=0$.
			
			Therefore, recalling also~\eqref{simzeromenomale09}, we find that, for any~$\epsilon>0$,
\begin{eqnarray*}
&&\int_{\R^n \setminus B_{\epsilon}} \frac{z (u(x+z)-u(x))}{|z|^{n+s+1}} \, dz
= \int_{\R^n \setminus B_{\epsilon}} \frac{z u(x+z)}{|z|^{n+s+1}} \, dz\\&&\qquad
= \int_{B_{R_0} \setminus B_{\epsilon}} \frac{z u(x+z)}{|z|^{n+s+1}} \, dz=
\int_{B_{R_0} \setminus B_{\epsilon}} \frac{z (u(x+z)-u(x))}{|z|^{n+s+1}} \, dz.
\end{eqnarray*}
			Consequently,
\begin{equation}\label{hu43ty854b97689v6cb2}\begin{split}
				D^s u(x) &= c_{s}  \lim_{\epsilon \searrow 0} \int_{\R^n \setminus B_{\epsilon}} \frac{z (u(x+z)-u(x))}{|z|^{n+s+1}} \, dz \\&=  c_{s}  \lim_{\epsilon \searrow 0} \int_{B_{R_0} \setminus B_{\epsilon}} \frac{z (u(x+z)-u(x))}{|z|^{n+s+1}} \, dz.\end{split}
\end{equation}
				
Now we denote by~$L_u$ the Lipschitz constant of~$u$ and by~$\mathds{1}_{A}$ the characteristic function of a set~$A\subset\R^n$. Then, we set
				\begin{eqnarray*}
					g(z)&:= &\dfrac{L_u\mathds{1}_{\{|z|\leq 1\}}}{|z|^{n-(1-s)}} + \dfrac{2 \Vert u \Vert_{L^{\infty}(\R^n)} \mathds{1}_{\{|z|>1\}}}{|z|^{n+s}}
				\\{\mbox{and }}\quad  
	f_{\epsilon}(z)&:=&\dfrac{z(u(x+z)-u(x))}{|z|^{n+s+1}} \mathds{1}_{\{|z|> \varepsilon\}}.\end{eqnarray*}
Since~$g \in L^1(\R^n)$ and~$|f_\epsilon| \leq g$ for any~$\varepsilon>0$, we can apply the Dominated Convergence Theorem to obtain that				
\begin{eqnarray*} \lim_{\epsilon \searrow 0}
\int_{B_{R_0} \setminus B_{\epsilon}}   \dfrac{z(u(x+z)-u(x))}{|z|^{n+s+1}} \, dz &=&
					\lim_{\epsilon \searrow 0}  \int_{B_{R_0}} f_\epsilon(z) \, dz   \\&=&  \int_{B_{R_0}}  \frac{z(u(x+z)-u(x))}{|z|^{n+s+1}} \, dz,    
\end{eqnarray*}
which, together with~\eqref{hu43ty854b97689v6cb2}, establishes
the first line in~\eqref{hu43ty854b97689v6cb}.
Moreover, by the change of variables~$z= y-x$, we obtain the second line in~\eqref{hu43ty854b97689v6cb}.
			\end{proof}  With the aid of Lemma \ref{kj} we now prove the following continuity result for the fractional gradient:
			
\begin{corollary}\label{ptems}
Let $s \in (0,1]$ and~$u \in \mathcal{D}(\R^n)$. Then, for any $x_0 \in \R^n$, 
\[\lim_{x\to x_0} D^s u(x) = D^s u(x_0). \]
\end{corollary}
\begin{proof}
If $s=1$ the result is trivial. Thus, from now on, we suppose that~$s \in (0,1)$.

Let $x_0 \in \R^n$
and~$\overline R >0$ be such that~$B_{\overline R}$ contains the support of~$u$.
Thus, by Lemma~\ref{kj}, for all $R\ge|x_0|+\overline R+1$ and for all~$x\in B_1(x_0)$ we have that
\begin{equation}
D^su(x) = c_s \int_{B_R}  \frac{z(u(x+z)-u(x))}{|z|^{n+s+1}} \, dz \quad\mbox{for any}\quad x \in (x_0-1, x_0+1).
\end{equation}

Now we define
\[ f(z) := \frac{\|Du\|_{L^{\infty}(\R^n)}}{|z|^{n-(1-s)}} \]and we note that $f \in L^1(B_R)$.

Moreover, for any $z \in B_R $,
\begin{equation*}\begin{split}&
\left|\frac{z(u(x+z)-u(x))}{|z|^{n+s+1}}\right|\le
\frac{1}{|z|^{n+s}}\left|\int_0^1 D u(x+tz)\cdot z\,dt\right|\\&\qquad\qquad\le \frac{\|Du\|_{L^{\infty}(\R^n)}}{|z|^{n+s-1}} =f(z).\end{split}
\end{equation*}
As a consequence, we can apply the Dominated Convergence Theorem and exploit the continuity of~$u$ to obtain that
\[ \lim_{x\to x_0} \int_{B_R}   \frac{z(u(x+z)-u(x))}{|z|^{n+s+1}} \, dz = \int_{B_R}  \frac{z(u(x_0+z)-u(x_0))}{|z|^{n+s+1}} \, dz ,\]which gives the desired result.
\end{proof}

We are now ready to provide a  justification to the definition given in~\eqref{def} for~$s=1$. The following proposition, indeed, shows that the fractional gradient converges pointwise to the classical gradient as~$s \nearrow 1$. 

\begin{proposition}\label{gra}
				Let~$u \in \mathcal{D}(\R^n)$. Then, for any~$x \in \R^n$,
				\[ \lim_{s \nearrow 1}  D^s u (x) =  Du (x) .\]
				
			\end{proposition}
			
			\begin{proof} Let~$x \in \R^n$.
			In order to establish the claim of Proposition~\ref{gra}, we prove that, for any~$i= 1, \ldots, n$, 
			\begin{equation}\label{firty34586fejhkwfgw4ketgu4i43908}
			\lim_{s \nearrow 1}  D^s_i u (x) =  D_i u (x).\end{equation}
For this, we suppose that the support of~$u$ is contained in some ball~$B_R$ and set~$R_x:=|x|+R+1$. In this way, by Lemma~\ref{kj}, we have that
				\[D^s_i u (x) = c_s \int_{B_{R_x}}\frac{z_i (u(x+z)-u(x))}{|z|^{n+s+1}}  \, dz.\]
Thus,
\begin{equation}\label{popo}	D^s_i u (x) = c_s \int_{B_{R_x}} \frac{z_i (D_ku(x) z_k+ O(|z|^2)) }{|z|^{n+s+1}} \, dz  = A_s(x) + B_s(x),
				\end{equation}
where
\begin{eqnarray*} A_s(x) &:=& c_s D_ku(x) \int_{B_{R_x}}  \frac{ z_i z_k }{|z|^{n+s+1}} \, dz \\
{\mbox{and }}\qquad
B_s(x) &:=& c_s \int_{B_{R_x}} \frac{z_i  O(|z|^2)}{|z|^{n+s+1}} \, dz. \end{eqnarray*}
We observe that
				\[\int_{B_{R_x}}  \frac{ z_i z_k }{|z|^{n+s+1}} \, dz= 0,\]
				whenever~$k \neq i$, and therefore
\begin{eqnarray*}
					A_s(x) &=& c_s D_iu(x) \int_{B_{R_x}}  \frac{ z^{2}_i }{|z|^{n+s+1}} \, dz \\&=& \frac{c_s}{n} D_iu(x)   \int_{B_{R_x}}  \frac{ dz  }{|z|^{n-(1-s)}} 
					\\&=& \frac{c_s}{(1-s)}   \frac{R_x^{1-s}S_{n-1}}{n}D_iu(x).\end{eqnarray*}
				Using~\eqref{gf} and the relation~$\omega_n= S_{n-1}/n$, we thereby obtain that
$$				\lim_{s \nearrow 1} A_s(x) = D_iu(x).$$

Accordingly, to complete the proof of~\eqref{firty34586fejhkwfgw4ketgu4i43908},
it remains to check that
\begin{equation}\label{firty34586fejhkwfgw4ketgu4i439082}
\lim_{s\nearrow 1} B_s(x) =0.
\end{equation}
For this, we notice that
$$ |B_s(x)|\le c_s \int_{B_{R_x}} \frac{O(|z|^2)}{|z|^{n+s}} \, dz 
= \frac{c_s O(R_x^{2-s})}{2-s},	$$
which vanishes as~$s\nearrow1$ (thanks to~\eqref{gf}), thus establishing~\eqref{firty34586fejhkwfgw4ketgu4i439082}
and completing the proof of~\eqref{firty34586fejhkwfgw4ketgu4i43908}, as desired.
\end{proof}

We mention that, as~$s \nearrow 1$, the convergence of the fractional gradient to the classical gradient does not hold only pointwise. For example, for any~$p \in (1, +\infty)$ and~$u \in W^{1,p}(\R^n)$, we have that (see~\cite[Theorem~3.2]{mora})
			\[ {\mbox{$D^s u \to  Du$ in~$L^p(\R^n)$ as~$s \nearrow 1$.}}\]


Now, we study the decay rate at infinity of the fractional gradient operator, that, under suitable integrability assumptions on~$u$, can be shown to be polynomial:
			
			\begin{proposition}\label{usasempre}
				Let~$s\in (0,1)$,~$u \in \mathcal{D}(\R^n)$ and  let~$ R >0$ be such that~$B_{ R}$ contains the support of~$u$.			
				Then, for any~$x\in\R^n\setminus B_{ 2 R}$,
				\[  |D^s u (x)| \leq  \frac{  2^{n+s} c_s  \Vert u \Vert_{L^1\left(B_{ R}\right)}} {|x|^{n+s}} . \]
			\end{proposition}
			
			\begin{proof}
Let~$x \in \R^n \setminus B_{2  R}$. For any~$y \in B_{ R}$, we have that
				$$
					|x-y| \ge |x|-|y|\ge |x|-\frac{|x|}2= \frac{|x|}{2}. $$
				As a consequence,  making use of Lemma~\ref{kj}, we find that
				\begin{eqnarray*}
				&&	|D^s u (x)| \leq  c_s \Int \frac{|u(x)-u(y)|}{|x-y|^{n+s}} \, dy =  c_s \int_{\R^n} \frac{|u(y)|}{|x-y|^{n+s}} \, dy \\ &&\qquad=  c_s \int_{B_{ R}} \frac{|u(y)|}{|x-y|^{n+s}} \, dy \leq \frac{  2^{n+s} c_s} {|x|^{n+s}} \Vert u \Vert_{L^1\left(B_{ R}\right)},
			\end{eqnarray*} as desired.\end{proof}
			
			Regarding the regularity properties of the fractional gradient, we recall that, for any~$s \in (0,1)$
			and~$u \in \mathcal{D}(\R^n)$, one has that (see~\cite[Lemma~3.1]{piola})
			\begin{equation}\label{fracpiola}
				{\mbox{$D^s u \in L^p (\R^n)$ for all~$p \in [1, +\infty]$.}}
			\end{equation}
			
			In particular, the next lemma ensures that the~$L^{\infty}$ norm of the fractional gradient can be bounded in terms of the~$L^{\infty}$ norm of the classical gradient, uniformly with respect to~$s \in (0,1]$.
			
			\begin{lemma}\label{drday}
				Let~$u \in \mathcal{D}(\R^n)$ and let~$R>1/3$ be such that~$B_{R}$ contains the support of~$u$. 
				
				Then, there exists a positive constant $C$, depending only on~$n$ and~$R$, such that
				\[\sup_{s \in (0,1]} \Vert D^s u \Vert_{L^{\infty}(\R^n)} \leq C \Vert Du \Vert_{L^{\infty}(\R^n)}.\]
			\end{lemma}
			
			\begin{proof}
We notice that, if~$s=1$, we can take~$C=1$ to obtain the desired estimate. Hence, from now on, we suppose that~$s \in (0,1)$.
				
				We observe that, since~$u \in \mathcal{D}(\R^n)$, we can choose~$x_0 \in B_{R}$ satisfying~$u(x_0)=0$. Furthermore, applying the Fundamental Theorem of Calculus, we obtain that, for any~$x \in \R^n$,
				\begin{equation*}\begin{split}&
				|u(x)| =|u(x)-u(x_0)|= \left|
				\int_0^1 Du(tx + (1-t)x_0) \cdot  (x-x_0) \, dt \right|\\&\qquad\qquad\leq 2R \Vert Du \Vert_{L^{\infty}(B_{R})}.	\end{split}
				\end{equation*}
				As a consequence,
$$				\Vert u \Vert_{L^{\infty}(B_{R})} \leq 2R \Vert Du \Vert_{L^{\infty}(B_{R})}.
$$
{F}rom this and Proposition~\ref{usasempre}, we have that, for any~$x \in \R^n \setminus B_{2R} $,
\begin{equation}\begin{split}\label{primd}
&|D^s u (x) | \leq \frac{2^{n+s} c_s \Vert u \Vert_{L^{1}(B_{R})}} {|x|^{n+s}} 
\leq  \frac{2^{n+s} c_s |B_{R}| \Vert u \Vert_{L^{\infty}(B_{R})} } {|x|^{n+s}}   \\
&\qquad\leq  \frac{c_s|B_{R}| \Vert u \Vert_{L^{\infty}(B_{R})} } {{R}^{n+s}} 
\leq 
\frac{2c_s|B_{R}|  } {{R}^{n+s-1}}  \Vert Du \Vert_{L^{\infty}(B_{R})}.\end{split}\end{equation} 

Furthermore, by Lemma~\ref{kj}, we have that,
for any~$x \in B_{2R}$,
\begin{equation*}\begin{split}
|D^s u (x)| &\leq c_s \int_{B_{3R}(x)} \frac{|u(y)-u(x)|}{|x-y|^{n+s}} \, dy \\&\leq c_s \Vert Du \Vert_{L^{\infty}(B_{R})} \int_{B_{3R}} \frac{dz}{|z|^{n-(1-s)}}\\&
=  \frac{c_s S_n (3R)^{1-s}}{(1-s)}   \Vert Du \Vert_{L^{\infty}(B_{R})} . \end{split}\end{equation*} 
This and~\eqref{primd} entail that
\begin{equation*}
\|D^s u \|_{L^\infty(\R^N)}\le\left( \frac{2c_s|B_{R}|  } {{R}^{n+s-1}}+ \frac{c_s S_n (3R)^{1-s}}{(1-s)}   \right) \Vert Du \Vert_{L^{\infty}(B_{R})}.
\end{equation*}
In light of~\eqref{constdef} and~\eqref{gf}, we obtain that the constant appearing
in the above estimate depends only on~$n$ and~$R$, but is independent of~$s\in(0,1)$, as desired.
\end{proof}
			
There exists a fractional Fundamental Theorem of Calculus, valid for any~$u \in \mathcal{D}(\R^n)$, stating that, for any~$x$, $y \in \R^n$,
			\begin{equation}\label{yu}
				u(y)-u(x) = c_{-s} \Int \left( \frac{(z-x)}{|z-x|^{n+1-s}} - \frac{(z-y)}{|z-y|^{n+1-s}} \right) \cdot D^{s}u(z) \, dz.
			\end{equation} 
			While we refer to~\cite[Theorem~3.12]{comi} for a detailed proof of~\eqref{yu}, we show here a useful consequence.
			
			\begin{corollary}\label{aar}
				Let~$s \in (0,1)$. Then, for any~$u \in \mathcal{D}(\R^n)$ and any~$x \in \R^n$,
				\[ u(x)= c_{-s} \Int \frac{(x-z)}{|x-z|^{n-s+1}}\cdot D^s u (z) \, dz. \]
			\end{corollary}
			
\begin{proof}
We take~$R_1>0$ such that~$B_{R_1}$ contains the support of~$u$. Then, for any~$R>2R_1$ we define
\begin{eqnarray*}A_R(x)&:=& \int_{B_{R}} \frac{|D^{s}u(z)|}{|z-x|^{n-s}} \, dz \\
{\mbox{and }}\qquad
E_R(x)&:=& \int_{\R^n \setminus B_{R}} \frac{|D^{s}u(z)|}{|z-x|^{n-s}} \, dz. \end{eqnarray*}
We claim that
\begin{equation}\label{schis}
\lim_{R \to +\infty} \lim_{|x| \to +\infty} A_R(x) = 0 
\end{equation}
and
\begin{equation}\label{schisd}
\lim_{R \to  +\infty } \lim_{|x| \to +\infty} E_R(x) = 0. 
\end{equation}

To prove~\eqref{schis}, we notice that for any~$R>0$, $z\in B_R$ and~$x\in\R^n\setminus B_{ R +1}$,
\[ \frac{| D^{s}u(z)|}{|z-x|^{n-s}}  \leq |D^{s}u(z)|. \]
As a consequence, since~$D^su \in L^1(\R^n)$ by~\eqref{fracpiola}, we can apply the Dominated Convergence Theorem and conclude that
				\[\lim_{|x| \to +\infty} A_R(x) \leq \int_{B_{R}}  \lim_{|x| \to +\infty} \left(\frac{1}{|z-x|^{n-s}} \right) |D^{s}u(z)| \, dz = 0. \]
				Taking the limit as~$R \to +\infty$, we establish~\eqref{schis}.
				
Now we prove~\eqref{schisd}. For the sake of clarity, we use the
notation~$B^c_r(p):=\R^n \setminus B_r(p)$ for every~$r>0$ and~$p\in\R^N$.
Then, for any~$R>2R_1$ and~$x \in \R^n$, using Proposition~\ref{usasempre},
we conclude that
\begin{eqnarray*}&&
E_R(x) \le 2^{n+s} c_s  \Vert u \Vert_{L^1(\R^n)} \int_{ B_{R}^c} \frac{dz}{|z-x|^{n-s}|z|^{n+s}} 
\\&&\quad\le 2^{n+s} c_s  \Vert u \Vert_{L^1(\R^n)} 
 \left(\frac{1}{ R^{n+s}}
 \int_{B^c_{R} \cap B_1(x)} \frac{dz}{|z-x|^{n-s}} + \int_{B^c_{R}\cap
 B_1^c(x)} \frac{dz}{|z|^{n+s}}\right)\\ 
&&\quad\leq 2^{n+s} c_s  \Vert u \Vert_{L^1(\R^n)} \left( \frac{1}{ R^{n+s}} \int_{B_1(x)} \frac{dz}{|z-x|^{n-s}} + \int_{B^c_{R}} \frac{dz}{|z|^{n+s}} \right) \\ 
&&\quad= \frac{2^{n+s} c_s S_n \Vert u \Vert_{L^1(\R^n)}}{s} \left( \frac{1}{R^{n+s}}+ \frac{1}{R^{s}} \right).
\end{eqnarray*}
Since this estimate does not depend on~$x$, letting~$R \to +\infty$ we obtain~\eqref{schisd}, as desired.

Now, we point out that
\begin{eqnarray*}
\left| \Int \frac{(z-x)}{|z-x|^{n+1-s}} \cdot D^{s}u(z) \, dz  \right| \le A_R(x)+E_R(x)
\end{eqnarray*}
and therefore,
from~\eqref{schis} and~\eqref{schisd}, we see that
\begin{eqnarray*}&&
\lim_{|x|\to+\infty}\left| \Int \frac{(z-x)}{|z-x|^{n+1-s}} \cdot D^{s}u(z) \, dz  \right| \\&
=& \lim_{R\to+\infty}\lim_{|x|\to+\infty}\left| \Int \frac{(z-x)}{|z-x|^{n+1-s}} \cdot D^{s}u(z) \, dz  \right| \\
&
\le& \lim_{R\to+\infty}\lim_{|x|\to+\infty}\Big(A_R(x)+E_R(x)\Big)=0.
\end{eqnarray*}

As a result, taking the limit as~$|x|\to+\infty$ in~\eqref{yu}, we obtain that 
\begin{eqnarray*}
u(y)&=&\lim_{|x|\to+\infty}\Big(u(y)-u(x)\Big)
\\&=& \lim_{|x|\to+\infty} c_{-s} \Int \left( \frac{(z-x)}{|z-x|^{n+1-s}} - \frac{(z-y)}{|z-y|^{n+1-s}} \right) \cdot D^{s}u(z) \, dz\\
&=& c_{-s} \Int \frac{(y-z)}{|z-y|^{n+1-s}} \cdot D^{s}u(z) \, dz,
\end{eqnarray*}
as desired.
\end{proof}
The next proposition deals with the differentiability properties of the fractional gradient. For any multi-index $\alpha= (\alpha_1, \ldots, \alpha_n)\in\N^n$, we set~$|\alpha|:= \sum_{i=1}^n \alpha_i$ and
\[ \partial^{\alpha} u := \frac{\partial^{|\alpha|}u}{\partial x_1^{\alpha_1} \cdots \partial x_{n}^{\alpha_n}} \, .\]
\begin{proposition}\label{ffffff}
Let~$s\in(0,1]$ and $u \in \mathcal{D}(\R^n)$. Then, we have that~$D^s u \in C^{\infty}(\R^n, \R^n)$ and, for any $\alpha\in\N^n$,
\begin{equation}\label{emscdl}
\partial^{\alpha} \left(D^s u \right) = D^s \left(\partial^{\alpha} u \right).
\end{equation}
\end{proposition}
\begin{proof}
If $s=1$, then Proposition~\ref{ffffff} is obviously true, thus we suppose from now on that~$s \in (0,1)$.
We point out that the continuity of~$D^su$ follows from Corollary~\ref{ptems}.
Throughout the proof, for the sake of clarity, we denote the partial derivative with respect to~$x_i$ by either~$\frac{\partial u}{\partial x_i}$ or $u_{x_i}$. 

Also, since $u \in \mathcal{D}(\R^n)$, we have that
\[ D^s u(x) = \frac{c_s}{2} \int_{\R^n}  \frac{z(u(x+z)-u(x-z))}{|z|^{n+s+1}} \, dz. \]

We claim that, for any $ x \in \R^n$,
\begin{equation}\label{qelucms}
\frac{\partial D^s u}{\partial x_1}(x)  = D^s u_{x_1}(x) .
\end{equation}
With the aim of proving~\eqref{qelucms}, we set
 \[ f(z):=\frac{2\|Du \|_{L^{\infty}(\R^n)}}{|z|^{n+s}}\]
 and we observe that $f \in L^1(\R^n \setminus B_1)$.
 
Moreover, for any $x \in \R^n$,
$z \in \R^n \setminus B_1$ and~$|h|\leq 1$,
\begin{eqnarray*}&&
\left|\frac{z \big[(u (x+ h e_1+z) - u (x+ h e_1-z))-(u(x+z)-u(x-z))\big]}{h |z|^{n+s+1}} \right| \\
&&\qquad\le\frac1{|h|\,|z|^{n+s}}\left| \int_0^h u_{x_1}(x+z+te_1)\,dt-\int_0^h u_{x_1}(x-z+te_1)\,dt\right|\\
&&\qquad\le\frac{2\|Du \|_{L^{\infty}(\R^n)}}{|h|\,|z|^{n+s}}= f(z).
\end{eqnarray*}
As a consequence, we can apply the Dominated Convergence Theorem to find that
 \begin{equation}\label{qelpcms}
\begin{split}
&\lim_{h \to 0} \int_{\R^n\setminus B_1} \frac{z \big[(u (x+ h e_1+z) - u (x+ h e_1-z))-(u(x+z)-u(x-z))\big]}{h |z|^{n+s+1}} \, dz \\
&\qquad\qquad=  \int_{\R^n\setminus B_1} \frac{z \left( u_{ x_1} (x+ z) - u_{ x_1 }  (x-z)\right)}{|z|^{n+s+1}} \, dz .
\end{split}
\end{equation}

Now, for any $x \in\R^n$,~$z\in B_1$ and $|h|\leq1$, 
\begin{equation*}\footnotesize{
\begin{split}
&\Bigg|  \frac{\big(u (x+ h e_1+z) - u (x+ h e_1-z)\big)-
\big(u(x+z)-u(x-z)\big)}{h} - \big( u_{x_1}(x+z) - u_{x_1}(x-z)  \big)\Bigg|\\
&= \left| \left( \frac{1}{h}\int_0^h u_{x_1}(x+z+\tau e_1) - u_{x_1}(x-z+\tau e_1) \, d\tau \right)  - \big(u_{x_1}(x+z) - u_{x_1}(x-z)  \big)\right|\\
&= \frac{1}{|h|} \left| \int_0^h \Big( \left[u_{x_1}(x+z+\tau e_1) - u_{x_1}(x-z+\tau e_1) \right]   - \big( u_{x_1}(x+z) - u_{x_1}(x-z)  \big)\Big) \,  d\tau \right|\\
&=  \frac{1}{|h|} \left|\int_0^h \left( \int_0^\tau\frac{\partial^2 u}{\partial x_1^2}(x+z+t e_1) - \frac{\partial^2 u}{\partial x_1^2}(x-z+t e_1)  \, dt\right) \,  d\tau \right|\\
&\leq |z| \,|h|\sup_{{|\alpha|=3}\atop{x \in \R^n} }\left|\partial^{\alpha}u(x)\right| .
\end{split}
}\end{equation*}

As a consequence, 
\begin{eqnarray*}
&&\lim_{h \to 0}\left| \int_{ B_1} \frac{z \big[(u (x+ h e_1+z) - u (x+ h e_1-z))-(u(x+z)-u(x-z))\big]}{h |z|^{n+s+1}} \, dz\right.\\&&\qquad\qquad\qquad\left. - \int_{ B_1} \frac{z \left( u_{ x_1} (x+ z) - u_{ x_1 }  (x-z)\right)}{|z|^{n+s+1}} \, dz\right|\\&&\qquad
\le\lim_{h\to0}\int_{B_1}|z|^{1-n-s} \,|h|\sup_{{|\alpha|=3}\atop{x \in \R^n} }\left|\partial^{\alpha}u(x)\right| \,dz=0.
\end{eqnarray*}
Combining this information with~\eqref{qelpcms} we establish~\eqref{qelucms}, as desired.

Now, since $u \in \mathcal{D}(\R^n)$, for any $\alpha\in\N^n$ we have that $\partial^{\alpha} u \in \mathcal{D}(\R^n)$ too. Consequently, we can apply the same argument as above to each derivatives of~$u$
in order to obtain~\eqref{emscdl}. 
\end{proof}

For further reference, we also recall the following ``integration by parts'' formula for the fractional gradient
(see~\cite[Lemma 2.2]{mora}):

\begin{lemma}\label{eduldp}
			Let~$s \in (0,1]$ and $p \in (1,+\infty)$. Let~$v: \R^n \to \R^n$ be such that~$v_i \in \mathcal{D}(\R^n)$ for any~$i=1,\ldots,n$.
			
			Then, for any $\phi \in \mathcal{D}(\R^n)$,
			\begin{equation}\sum_{i=1}^n
 			\Int D^s_i v_i\, \phi \, dx = - \Int v \cdot D^s \phi  \, dx.
			\end{equation}	
			\end{lemma}

\subsection{The Riesz Potential}\label{rieszsec}
			This section aims at showing some relations occurring between the fractional gradient and the Riesz potential operator. After a brief introduction on this operator, Proposition~\ref{huy} below will provide an explicit form of the Fourier transform of the fractional gradient. Then, Theorem~\ref{usacchialo} will constitute the main theorem of this section and it will be used to study embedding properties of the functional spaces under consideration in this paper (see the forthcoming
Proposition~\ref{crizzo}). For a complete discussion on this topic, we refer the reader to~\cite{stein, hedberg, gravakos}. 
			
The Riesz potential of order~$\alpha \in (0,1)$ is formally defined as
			\begin{equation}\label{riesz}
				I_{\alpha}u (x) := K_{\alpha} \ast u (x) = \frac{1}{\gamma_{\alpha,n}}\Int \frac{u(y)}{|x-y|^{n-\alpha}} \, dy,   
			\end{equation}
			where~$K_{\alpha}$ is defined by 
	\begin{equation}\label{defgamma09}
			K_{\alpha}(x)= \frac{1}{\gamma_{\alpha,n}} |x|^{-n+\alpha}
		\qquad {\mbox{ with }}	\qquad
				\gamma_{\alpha,n}= \frac{2^{\alpha} \pi^{\frac{n}{2}} \Gamma(\frac{\alpha}{2})}{\Gamma(\frac{n-\alpha}{2})}.    \end{equation}
							Since the space dimension~$n$ constitues a fixed parameter in this work, we will refer to~$\gamma_{\alpha,n}$ simply as~$\gamma_{\alpha}$. We stress that, for any~$s \in (0,1)$, the relation between~$\gamma_{s}$ and~$c_s$ is the following
			\begin{equation}\label{relation}
				c_s= \frac{n+s-1}{\gamma_{1-s}}.
			\end{equation}
It is known (see e.g.~\cite[I, Theorem~1]{stein}) that if
$$ p \in \left(1,\frac{ n}{\alpha}\right)\qquad {\mbox{and}}\qquad
q:=\frac{np}{n-\alpha p},
$$ 		
then, for some~$C>0$, it holds that
\begin{equation}\label{riz}
\Vert I_{\alpha}u \Vert_{L^q(\R^n)} \leq C \Vert u \Vert_{L^p(\R^n)}.
\end{equation}

Now we proceed by computing the Fourier transform of the Riesz potential.
To this end, we establish the following preliminary result:

			\begin{lemma}\label{fr}
				Let~$\alpha \in (0,1)$ and~$u \in \mathcal{S}(\R^n)$. Then,
				\[
				\mathcal{F}(I_{\alpha}u) (\xi) = |2 \pi \xi|^{-\alpha}\, \widehat{u}(\xi),
				\]
				where the Fourier transform is intended in the sense of distribution.
			\end{lemma}
			
			\begin{proof}
				We claim that
				\begin{equation}\label{shov}
				K_{\alpha} \in \mathcal{S}'(\R^n).
				\end{equation}
				To this end, for every~$\phi \in \mathcal{S}(\R^n)$, we define
	$$\langle  K_{\alpha}, \phi \rangle: 
= \Int \frac{\phi(x)}{\gamma_{\alpha}|x|^{n-\alpha}} \, dx.$$
In this way, $K_\alpha$ is identified with a linear
functional from~$\mathcal{S}(\R^n)$ to~$\R$.
				Moreover,
				\begin{equation*}\begin{split}
					&\langle \gamma_{\alpha} \, K_{\alpha}, \phi \rangle 
					= \Int \frac{\phi(x)}{|x|^{n-\alpha}} \, dx = \int_{\{|x|<1\}} \frac{\phi(x)}{|x|^{n-\alpha}} \, dx + \int_{\{|x|>1\}} \frac{\phi(x)}{|x|^{n-\alpha}} \, dx  \\ 
					&\qquad\leq \Vert \phi \Vert_{L^{\infty}(\R^n)} \int_{\{|x|<1\}} \frac{dx}{|x|^{n-\alpha}}+ \Vert x \, \phi \Vert_{L^{\infty}(\R^n)} \int_{\{|x|>1\}} \frac{dx}{|x|^{n+1-\alpha}},\end{split}\end{equation*} 
		for some~$C>0$, which proves~\eqref{shov}. 
				
				As a consequence of~\eqref{shov}, since~$u \in \mathcal{S}(\R^n)$, we can write that
				\begin{equation}\label{huk}
				\mathcal{F}(I_{\alpha}u)(\xi) = \mathcal{F}(K_{\alpha} \ast u)(\xi) = \widehat{K_{\alpha}}(\xi) \widehat{u}(\xi).
				\end{equation}
				Now, for all~$z \in \mathbb{C}$, with~$\mbox{Re} z >-n$, we define the function
				\[
				u_z(x) := \frac{\pi^{\frac{z+n}{2}}|x|^z }{\Gamma(\frac{z+n}{2})}.
				\]
We observe that, since~$u_z \in L^1_{\rm loc}(\R^n)$, we can compute its Fourier transform, which is given by (see~\cite[Theorem~2.4.6.]{gravakos})
				\begin{equation}\label{grr}
					\widehat{u}_z(\xi)=u_{-n-z}(\xi).
				\end{equation}
				
				Thus, taking~$z:= -n + \alpha$,
				\[
				u_{-n+\alpha} (x)= \frac{\pi^{\frac{\alpha}{2}}|x|^{-n + \alpha} }{\Gamma(\frac{\alpha}{2})} 
				\]
				and, by~\eqref{grr},
				\begin{equation*}
	\mathcal{F}	\left( \frac{\pi^{\frac{\alpha}{2}}|x|^{-n + \alpha} }{\Gamma(\frac{\alpha}{2})}\right) =
	\widehat{u}_{-n+\alpha}(\xi)=u_{-\alpha}(\xi)=
	 \frac{\pi^{\frac{n-\alpha}{2}}|\xi|^{-\alpha}}{\Gamma(\frac{n-\alpha}{2})}.   
				\end{equation*}
As a result, recalling the definitions of~$K_\alpha$ and~$\gamma_\alpha$ in~\eqref{defgamma09}, we obtain that
				\begin{equation}\label{weneedit}
		\widehat{K_{\alpha}}(\xi)= \mathcal{F}\left(\frac1{\gamma_\alpha}
		|x|^{-n+\alpha}\right)(\xi)=  |2\pi \xi|^{-\alpha}. \end{equation}
Together with~\eqref{huk}, this entails the desired result.		
\end{proof}			

\begin{remark}
We stress that, if~$u \in \mathcal{S}(\R^n)$, then~$\mathcal{F}(I_{\alpha}u) \in L^p(\R^n)$ for all~$p \in [1, n/\alpha)$. Indeed, in light of Lemma~\ref{fr},
\begin{eqnarray*}&&
\Vert \mathcal{F}(I_{\alpha}u)\Vert^p_{L^p(\R^n)} \leq C \left( \int_{B_1} |\xi|^{-\alpha p} | \widehat{u}(\xi)|^p \, d\xi + \int_{\R^n \setminus B_1} |\xi|^{-\alpha p} | \widehat{u}(\xi)|^p \, d\xi \right)\\ &&\qquad
\leq C   \left( \int_{B_1} |\xi|^{-\alpha p} \, d\xi + \int_{\R^n \setminus B_1} |\xi|^{-(n+\alpha p)} \, d\xi \right)\leq C
,
\end{eqnarray*}
up to renaming~$C>0$ line after line.
\end{remark}				
			
With the aid of Lemma~\ref{fr} and a density argument, we now show the following:
			
			\begin{proposition}\label{prdue}
				Let~$\alpha \in (0,1)$ and~$u \in L^p(\R^n)$ for some~$p \in (1, n/\alpha)$. Then,
				\[
				\mathcal{F}(I_{\alpha} u) (\xi)= | 2 \pi \xi|^{-\alpha}\widehat{u}(\xi),
				\]
				where the Fourier transform is intended in the sense of distribution.
			\end{proposition}
			
\begin{proof}
To start with, we claim that 
\begin{equation}\label{pooo}
I_{\alpha}u \in \mathcal{S}'(\R^n).
\end{equation}
To this end, for any~$\phi\in{\mathcal{S}}(\R^n)$, we define
$$\langle I_{\alpha}u, \phi\rangle= \Int I_{\alpha}u (x)\, \phi(x) \, dx.$$
As a result, $I_{\alpha}u$ is identified with a linear functional from~${\mathcal{S}}(\R^n)$ to~$\R$. Furthermore,
we set	
\[q:= \frac{np}{n-\alpha p}, \]
and we use the H\"older inequality and~\eqref{riz} to see that
				\begin{eqnarray*}&&\Int I_{\alpha}u(x) \, \phi(x) \, dx \leq \Vert I_{\alpha}u \Vert_{L^q(\R^n)} \Vert \phi \Vert_{L^{\frac{q}{q-1}}(\R^n)} \\&&\qquad\qquad\leq C \Vert u \Vert_{L^p(\R^n)}\Vert \phi \Vert_{L^{\frac{q}{q-1}}(\R^n)} < +\infty,\end{eqnarray*}
				which completes the proof of~\eqref{pooo}.

Now, by density, we take a sequence~$(u_k)_k \subset \mathcal{D}(\R^n)$ that converges to~$u$ in~$L^p(\R^n)$ as~$k\to+\infty$. Thanks to Lemma~\ref{fr}, we know that, for any~$k\in\N$, in the sense of distribution,
				\begin{equation}\label{klim}
		\mathcal{F}(I_{\alpha} u_k) (\xi) = | 2 \pi \xi|^{-\alpha} \widehat{u_k} (\xi).
				\end{equation}
				
We point out that
\begin{equation}\label{wq4r87365gfsekjh}
\mathcal{F}(I_{\alpha} u)= \lim_{k \to +\infty} \mathcal{F}(I_{\alpha} u_k).
\end{equation}
Indeed, by the H\"older inequality and~\eqref{riz}, for all~$\phi\in{\mathcal{S}}(\R^n)$,
\begin{eqnarray*}
\left|\Int I_{\alpha}(u_k-u) (x)\, \phi(x) \, dx \right|& \leq& \Vert \phi \Vert_{L^{\frac{q}{q-1}}(\R^n)} \Vert I_{\alpha}(u_k-u) \Vert_{L^q(\R^n)} \\&\leq& C \Vert \phi \Vert_{L^{\frac{q}{q-1}}(\R^n)} \Vert u_k-u\Vert_{L^p(\R^n)}
,\end{eqnarray*}
which gives that
$$ \lim_{k\to+\infty} I_{\alpha}u_k= I_{\alpha}u.$$
This implies~\eqref{wq4r87365gfsekjh}, as desired.

Moreover, we have that, in the sense of distribution,
\begin{equation}\label{jeoiwt5734230h56yg}
\lim_{k\to+\infty} \widehat{u_k} = \widehat{u}.\end{equation}
To check this, we notice that, for all~$\phi\in{\mathcal{S}}(\R^n)$,
\begin{eqnarray*}
\left|\int_{\R^n} \Big(\widehat{u_k}(\xi) -\widehat{u}(\xi)\Big)\phi(\xi)\,d\xi\right|
&=&\left| \int_{\R^n} \Big( {u_k}(\xi) - {u}(\xi)\Big)\widehat\phi(\xi)\,d\xi\right|\\&\le&
\Vert {u_k} - {u} \Vert_{L^p(\R^n)} \Vert\widehat\phi \Vert_{L^{\frac{p}{p-1}}(\R^n)} ,
\end{eqnarray*}
from which~\eqref{jeoiwt5734230h56yg} follows.

Gathering together~\eqref{klim}, \eqref{wq4r87365gfsekjh} and~\eqref{jeoiwt5734230h56yg}, we conclude that, in the sense of distribution,
\begin{equation*}
\mathcal{F}(I_{\alpha} u)= \lim_{k \to +\infty} \mathcal{F}(I_{\alpha} u_k) 
=  \lim_{k \to +\infty} | 2 \pi \xi|^{-\alpha} \widehat{u_k} (\xi)
= | 2 \pi \xi|^{-\alpha} \widehat{u} (\xi),
\end{equation*}
which is the desired result.
\end{proof}
			
We now compute the Fourier transform of the fractional gradient.

			\begin{proposition}\label{huy}
				Let~$s \in (0,1)$ and~$u \in \mathcal{D}(\R^n)$. Then,
				\begin{equation}\label{four}
					\mathcal{F}(D^s u)(\xi)= i (2 \pi)^s \xi |\xi|^{s-1} \widehat{u}(\xi).
				\end{equation}
			\end{proposition}
			
			\begin{proof}
We let~$R>0$ such that the support of~$u$ is contained in~$B_R$.			
			
Also, we point out that 
$$
D^s u(x) =  \frac{c_s}{2}  \Int \frac{[u(x+z)-u(x-z)]z}{|z|^{n+s+1}} \, dz.
$$
Thus, for all~$i\in\{1,\ldots,n\}$, we have that
\begin{equation}\label{mnbvcxz123456780}
{\mathcal{F}}(D^s_iu)(\xi)=\frac{c_s}2 \int_{\R^n}
\left(\Int \frac{[u(x+z)-u(x-z)]z_i}{|z|^{n+s+1}} \, dz\right)e^{2\pi i\xi\cdot x}\,dx.
\end{equation}

We claim that
\begin{equation}\label{mnbvcxz12345678}
{\mbox{the function~$\frac{[u(x+z)-u(x-z)]z_i}{|z|^{n+s+1}} e^{2\pi i\xi\cdot x}$
belongs to~$L^1(\R^{2n})$.}}
\end{equation}
To check this, we observe that, when~$z\in B_1$,
\begin{eqnarray*}&&
\left| \frac{[u(x+z)-u(x-z)]z_i}{|z|^{n+s+1}} e^{2\pi i\xi\cdot x}\right|\le
\frac{|u(x+z)-u(x-z)|}{|z|^{n+s}}\\ &&\qquad
=\frac{1}{|z|^{n+s}}\left|\int_{-1}^1 \nabla u(x+tz)\cdot z\,dt\right|\le
\frac{1}{|z|^{n+s-1}}\sup_{t\in(-1,1)}|\nabla u(x+tz)|\\
&&\qquad \le \frac{\mathds{1}_{B_{R+1}}(x)\|\nabla u\|_{L^\infty(\R^n)}  }{|z|^{n+s-1}}
=: f_1(x,z).
\end{eqnarray*}
Moreover, when~$z\in\R^n\setminus B_1$,
\begin{eqnarray*}
&&
\left| \frac{[u(x+z)-u(x-z)]z_i}{|z|^{n+s+1}} e^{2\pi i\xi\cdot x}\right|\le
\frac{|u(x+z)|+|u(x-z)|}{|z|^{n+s}}\\&&\qquad \le
\frac{\mathds{1}_{B_R}(x+z)+\mathds{1}_{B_R}(x-z)}{|z|^{n+s}}\,\|u\|_{L^\infty(\R^n)}
=: f_2(x,z).
\end{eqnarray*}
{F}rom these observations, we deduce that
$$ \left| \frac{[u(x+z)-u(x-z)]z_i}{|z|^{n+s+1}} e^{2\pi i\xi\cdot x}\right|\le f_1(x,z)
\mathds{1}_{B_1}(z)+ f_2(x,z)\mathds{1}_{\R^n\setminus B_1}(z),$$
which belongs to~$L^1(\R^{2n})$, thus proving~\eqref{mnbvcxz12345678}.

In light of~\eqref{mnbvcxz12345678}, we can exploit Fubini-Tonelli Theorem and obtain from~\eqref{fracpiola} and~\eqref{mnbvcxz123456780} that
\begin{align*}
\mathcal{F}\left(D^s_i u\right)(\xi) &=  c_s i \left( \Int \frac{\left( e^{i 2 \pi z \cdot \xi} - e^{- i 2 \pi z \cdot \xi} \right)}{2i} \frac{z_i}{|z|^{n+s+1}} \, dz \right) \widehat{u} (\xi) \\ &= c_s i \left( \int_{\R^n} \sin(2 \pi z \cdot \xi) \frac{z_i}{|z|^{n+s+1}} \, dz \right) \widehat{u} (\xi)\\ &= c_s i (2\pi)^s \left( \int_{\R^n} \sin(\zeta \cdot \xi) \frac{\zeta_i}{|\zeta|^{n+s+1}} \, d\zeta \right) \widehat{u} (\xi).
\end{align*}
Then, by~\eqref{constant} and by Proposition~\ref{fbp}, we see that
\begin{equation*}\begin{split}
\mathcal{F}\left(D^s_i u\right)(\xi) &= c_si (2\pi)^s 2^{-s}\pi^{\frac{n}{2}} \frac{\Gamma(\frac{1-s}{2})}{\Gamma(\frac{n+s+1}{2})}\xi_i |\xi|^{s-1} \widehat{u} (\xi) 
\\&= i (2\pi)^s \xi_i |\xi|^{s-1} \widehat{u} (\xi). \qedhere\end{split}
\end{equation*}
\end{proof}
			
With the work done so far, we can now relate the fractional gradient with the Riesz potential via the following result. 
			
			\begin{theorem}\label{usacchialo}
				Let~$\Bar{s} \in (0,1)$ and~$u \in \mathcal{D}(\R^n)$. Then, for any~$s \in (\bs,1]$,
				\begin{equation}\label{imp}
					D^{\bs} u= I_{s-\bs} D^s u.
				\end{equation}
			\end{theorem}
			
			\begin{proof}
We point out that~$D^\sigma u \in \mathcal{S}'(\R^n)$, thanks to~\eqref{fracpiola},
for all~$\sigma\in(0,1]$. Moreover,	 by~\eqref{shov} we have that~$K_{s-\bar s} \in \mathcal{S}'(\R^n)$, and therefore~$I_{s-\bar s} D^su =K_{s-\bar s} \ast D^su \in \mathcal{S}'(\R^n)$. 

Also, by Propositions~\ref{prdue} and~\ref{huy},
$$
\mathcal{F}(I_{s-\bar s}D^s u) (\xi)= | 2 \pi \xi|^{\bar s-s}\widehat{D^su}(\xi)
= i( 2 \pi)^{\bar s} \xi|\xi|^{\bar s-1}\widehat u(\xi) = \mathcal{F}(D^ {\bs} u)(\xi),
$$
from which the desired result follows.
			\end{proof}

\section{The space $H_0^{s,p}(\Omega)$: embeddings and inequalities}\label{asd}

In this section we present the function spaces that naturally arise when dealing with fractional gradients and partial differential equations. A detailed account on these function spaces can be found in~\cite{shieh, comi, piola, mora}. 

In line with~\cite{shieh, comi, mora}, we define the following spaces.
\begin{definition}
Let~$s \in (0,1]$ and~$p \in [1, +\infty)$. We define the norm
\begin{equation}\label{odjschv nk:SWDFP}
\Vert \phi \Vert_{H^{s,p}(\R^n)} := \left(\Vert \phi \Vert^{p}_{L^p(\R^n)} + \Vert D^s \phi \Vert^{p}_{L^p(\R^n)}\right)^{\frac1p}
\end{equation}
and the space
\begin{equation}\label{sf}
H^{s,p} (\R^n):= \overline{ \mathcal{D}(\R^n) }^{\Vert \cdot \Vert_{H^{s,p}(\R^n)}}.
\end{equation}
\end{definition}

It is worth observing that the definition above
is well posed, since if~$u\in \mathcal{D}(\R^n) $, then~$u$, $ D^s \phi \in L^p(\R^n)$.

\begin{remark}
We stress that, for any function $u \in H^{s,p}(\R^n)$ with $s \in (0,1]$ and $p \in [1, +\infty)$, we have that $D^s u$ is well-defined according to Definition~\ref{defini}. 
\end{remark}
Furthermore, in a bounded domain $\Omega \subset \R^n$ we have the following setting
(see~\cite[Section 2.3]{mora}):
\begin{definition}
Let $\Omega \subset \R^n$ be a bounded domain. Let~$s \in (0,1)$ and~$p \in [1, +\infty)$. We define
\begin{equation}
H^{s,p}_0 (\Omega):= \overline{\mathcal{D}(\Omega)}^{\Vert \cdot \Vert_{H^{s,p}(\R^n)}}.
\end{equation}
\end{definition}
We remark that $H^{s,p}_0 (\Omega)$ is a subspace of $H^{s,p} (\R^n)$.

\begin{proposition}\label{ascmtm}
Let $s \in (0,1]$ and $p \in(1,+\infty)$. Also, let $\Omega$ be a bounded domain in $\R^n$.

Then, the spaces $H^{s,p}(\R^n)$ and $H^{s,p}_0(\Omega)$ are reflexive.
\end{proposition}

\begin{proof}
We point out that~$H^{1,p}(\R^n)$ is reflexive, thanks to~ \cite[Proposition 8.1]{brezis}.

Hence, we now focus on the case~$s \in (0,1)$.
To this end,
we notice that the product space~$L^{p}(\R^n) \times L^p(\R^n,\R^n)$ is reflexive. Also, we define the operator~$T: H^{s,p}(\R^n) \to L^{p}(\R^n) \times L^p(\R^n,\R^n)$ as~$T u := (u, D^su)$ and observe that~$T$
is an isometry from $H^{s,p}(\R^n)$ to $L^{p}(\R^n) \times L^p(\R^n,\R^n)$.
As a consequence, $T(H^{s,p}(\R^n))$ is a closed subspace of~$L^{p}(\R^n) \times L^p(\R^n,\R^n)$ and
therefore~$H^{s,p}(\R^n)$ is reflexive
(see e.g.~\cite[Theorem~15 on page~82]{MR1892228}). 

Since $H^{s,p}_0(\Omega)$ is a closed subspace of $H^{s,p}(\R^n)$, it is reflexive as well.
\end{proof}

\begin{remark}
{I}n \cite[Theorem 1.7]{shieh}, it was established, for any $s \in (0,1)$ and $p \in (1,+\infty)$, the identification of the spaces $H^{s,p}(\R^n)$ with the classical Bessel spaces (see~\cite[Definition 2.1]{shieh} for a formal definition of Bessel spaces). We refer the interested reader to~\cite[Theorem 2.2]{shieh} for a detailed characterization of these spaces.
\end{remark}  	
			
{N}ow, we present an embedding result for the spaces $H^{s,p}(\R^n)$ and $H^{s,p}_0(\Omega)$.

\begin{theorem}\label{qtppd}
Let~$s \in (0,1)$, $p \in (1, +\infty)$ and~$\Omega$ be a bounded domain of~$\R^n$.

Let also
\begin{equation}\label{pranbef}
			\begin{cases}
				q \in [1,p^*_s] & \mbox{ if } sp<n ,\\
				q \in [1,+\infty) & \mbox{ if } sp=n ,\\
				q \in [1,+\infty] & \mbox{ if } sp>n,
			\end{cases}
		\end{equation}
		where~$p^*_s := np / (n-sp)$ is the so-called critical exponent. 
		
Then,
\begin{equation}\label{polopo}
\mbox{$H_0^{s,p}(\Omega)$ continuously embeds into $L^q(\Omega)$.}
\end{equation}
		
Furthermore, if~$q$ satisfies 
		\begin{equation} \label{pran}
			\begin{cases}
				q \in [1,p^*_s) & \mbox{ if } sp<n ,\\
				q \in [1,+\infty) & \mbox{ if } sp=n, \\
				q \in [1,+\infty] & \mbox{ if } sp>n,
			\end{cases}
		\end{equation}
		then, for any sequence~$(u_k) \subset H^{s,p}_0(\Omega)$ such that~$u_k \rightharpoonup u$ in~$H^{s,p}(\R^n)$ as~$k\to+\infty$, for some~$u \in H^{s,p}(\R^n)$, we have that~$u \in H^{s,p}_0(\Omega)$ and 
		\begin{equation}\label{orcio}
		\mbox{$u_k \to u $ in $ L^q(\R^n)$ as~$k\to+\infty$.}
		\end{equation}
		\end{theorem}
		\begin{proof}
	When $sp<n$, the claim in~\eqref{polopo} is immediate from~\cite[Theorem~1.8]{shieh}.
	When~$sp=n$, the claim in~\eqref{polopo} comes from~\cite[Theorem~1.10]{shieh}. When~$sp>n$,
		\eqref{polopo} plainly follows from~\cite[Theorem~2.2, \textit{(e)}]{shieh}.

		The second part of Theorem~\ref{qtppd} is a straightforward consequence of \cite[Theorem 2.8]{mora}.
		\end{proof}
		\begin{corollary} \label{kas}
			Let~$s \in (0,1)$ and~$p \in (1, +\infty)$. Let~$\Omega$ be a bounded domain of~$\R^n$
			and~$q$ satisfy~\eqref{pran}. 
			Then, the embedding of~$H^{s,p}_0(\Omega)$ into~$L^q(\Omega)$ is compact.
		\end{corollary}
		\begin{proof}
		By Proposition~\ref{ascmtm}, we know that $H^{s,p}_0(\Omega)$ is reflexive. Also, by
		Theorem~\ref{qtppd}, we have that the embeddings of~$H_0^{s,p}(\Omega)$ into~$ L^q(\Omega)$ are continuous if~$q$ satisfies~\eqref{pran}. Then, by the theory of compact operators, we obtain that the embeddings are compact if they  map sequences converging in the weak topology to sequences that converge in the strong sense (i.e., in the norm topology); namely if for any~$(u_k)\subset H_0^{s,p}(\Omega)$ and~$u\in H_0^{s,p}(\Omega) $ such that~$u_k \rightharpoonup u$ in~$H^{s,p}_0(\R^n)$ as~$k\to+\infty$, we have that
\begin{equation}\label{plkj}\lim_{k\to+\infty}
\Vert u_k - u \Vert_{L^q(\Omega)}= 0.
\end{equation}		 
	Aiming at proving~\eqref{plkj}, we observe that, since~$H^{s,p}_0(\Omega) \subset H^{s,p}(\R^n)$, the reverse inclusion for the dual spaces is valid, i.e.~$\left(H^{s,p}(\R^n)\right)^*  \subset \left(  H^{s,p}_0(\Omega) \right) ^*$. Accordingly, any sequence~$(u_k)$ weakly converging to~$u$ in~$ H^{s,p}_0(\Omega)$, weakly converges to~$u$ in~$H^{s,p}(\R^n)$. Then,
	we are in the position of using the second statement in Theorem~\ref{qtppd} and therefore~\eqref{plkj} is a consequence of~\eqref{orcio}.
		\end{proof}
	
		We stress that, if~$sp<n$, Corollary~\ref{kas} constitutes a fractional counterpart of the classical Sobolev embedding.

		Now, let~$s \in (0,1)$, $p \in (1, +\infty)$ and~$\Omega$ be a bounded domain in~$\R^n$. According to~\cite[Theorem~2.9]{mora}, we know that there exists a positive
		constant~$C$, depending only on~$n$ and~$ \Omega$, such that, for any~$u \in H_0^{s,p}(\Omega)$,
		\begin{equation}\label{simn}
			\Vert u \Vert_{L^p(\Omega)} \leq \frac{C}{s} \Vert D^s u \Vert_{L^p(\R^n)}.
		\end{equation}
	Our aim is now to extend this result to the case $p=1$, which seems not to be covered in the available literature.		
		\begin{proposition}\label{costa}
			Let~$s \in (0,1)$ and~$\Omega$ be a bounded domain of~$\R^n$. Let~$\rho>0$ be such that~${\Omega \subset [-\rho,\rho]^n}$.
			
			Then, there exists~$C>0$, depending only on~$n$, such that, for any~$u \in H_0^{s,1}(\Omega)$,
			\begin{equation*}
				\Vert u \Vert_{L^1(\Omega)} \leq \frac{C\,\rho^s }{s} \Vert D^s u \Vert_{L^1(\R^n)}.
			\end{equation*}
		\end{proposition}
		\begin{proof}
		
		Let us set
			\begin{equation*}
				\Tilde{c}:= \sup_{s \in(-1,1)} c_{s},
			\end{equation*}		
			which is finite, thanks to~\eqref{constdef}, and 
		\begin{equation}\label{zizo}
R:= \left(  2^{\frac{n+1}{n}} \left(\frac{{\Tilde{c}}^2 S_{n-1}}{n}\right)^{\frac{1}{	n}} +1 \right)  {\rho},
\end{equation}	
			so that~$\Omega \subset B_{R}$.
			
			We establish the desired result for~$u \in \mathcal{D}(\Omega)$, then we will apply a density argument to complete the proof of the claim in its full generality. 		
			
We observe that, since~$u \in \mathcal{D}(\Omega)$, we infer from Corollary~\ref{aar} that, for any~$x \in \R^n$,
			\[ |u(x)| \leq c_{-s} \Int \frac{|D^s u(y)|}{|x-y|^{n-s}} \, dy. \] 
As a result, 
\begin{equation}\label{col}
\Vert u \Vert_{L^1(\Omega)} \leq A + B,
\end{equation}
		where
		\begin{eqnarray*}A&:=& c_{-s} \int_{B_{2R}}\left(\int_{\Omega}\frac{dx}{|x-y|^{n-s}}  \right)|D^s u(y)| \, dy 		\\
{\mbox{and }}\qquad 
	B&:=& c_{-s} \int_{\R^n \setminus B_{2R}}  \left( \int_{\Omega}\frac{dx}{|x-y|^{n-s}} \right) |D^s u(y)| \, dy.	\end{eqnarray*}
We first estimate~$A$. To this end, we change variable~$z:=x-y$ and obtain that
			\begin{equation}\label{098}\begin{split}&
				A \leq \Tilde{c} \left(\int_{B_{2R}}|D^s u(y)| \, dy\right) \left(\int_{B_{3R}}\frac{dz}{|z|^{n-s}} \right)\\&\qquad \leq \frac{3\Tilde{c}\, S_{n-1}\,{R}^s}{s} \Vert D^s u \Vert_{L^1(\R^n)} \leq \frac{C\, \rho^s}{s} \Vert D^s u \Vert_{L^1(\R^n)},\end{split}
			\end{equation} for some~$C>0$ depending only on~$n$.
			
Now we estimate~$B$. For this, we observe that if~$y \in \R^n \setminus B_{2R}$ and~$x \in \Omega$, then
\begin{equation*}
\frac{1}{|x-y|^{n-s} } \leq \frac{2^{n-s}}{|y|^{n-s}}.
\end{equation*}
Thus, by Proposition~\ref{usasempre} and the definition of~$R$ in~\eqref{zizo}, we obtain that
\begin{align*}&
B \leq c_{-s} 2^{n-s}|\Omega| \int_{\R^n \setminus	B_{2R}} \frac{|D^su(y)|}{|y|^{n-s}} \, dy \leq  {\Tilde{c}}^2 2^{2n} \rho^n \left( \int_{\R^n \setminus B_{2R}}\frac{dy}{|y|^{2n}}  \right) \Vert u \Vert_{L^1(\Omega)} \\&\qquad\qquad= \frac{{\Tilde{c}}^2 S_{n-1} 2^{n} \rho^n}{n\,R^n}  \Vert u \Vert_{L^1(\Omega)} \leq \frac{1}{2} \Vert u \Vert_{L^1(\Omega)}.
\end{align*}

From this, \eqref{col} and~\eqref{098}, we conclude that
		\begin{equation}\label{sggs}
			\Vert u \Vert_{L^1(\Omega)} \leq \frac{C\, \rho^s}{s} \Vert D^s u \Vert_{L^1(\R^n)} + \frac{1}{2} \Vert u \Vert_{L^1(\Omega)},
		\end{equation}
		which proves the thesis for~$u \in \mathcal{D}(\Omega)$. 
		
{N}ow we take~$u \in H^{s,1}_0(\Omega)$. Then, by definition there exists  a sequence $(\phi_k) \subset \mathcal{D}(\Omega)$ converging to~$u$ in~$H^{s,1}(\R^n)$ as~$k\to+\infty$. Hence, for all~$k\in\N$,
\begin{equation}
\Vert \phi_k \Vert_{L^1(\Omega)} \leq \frac{C\,\rho^s}{s} \Vert D^s \phi_k \Vert_{L^1(\R^n)} .
\end{equation}
Passing to the limit in $k$ concludes the proof.
\end{proof}

We observe that Proposition~\ref{costa} in the case $s=1$ reduces to the classical Poincar\'{e} inequality.
Also, relying on~\eqref{simn} and Proposition~\ref{costa}, we obtain the following properties for the space~$H_0^{s,p}(\Omega)$.

	\begin{corollary}\label{equivalentno} Let $s \in (0,1]$, $p \in [1, +\infty)$
		and~$\Omega$ be a bounded domain of~$\R^n$. 
		
		Then, there exists a positive constant~${\overline C}$, depending only
		on~$n$, $s$, $p$ and~$\Omega$, and independent of~$u$,
		such that, for any~$u \in H_0^{s,p}(\Omega)$,
\begin{equation}\label{gdg}
 \Vert u \Vert_{H^{s,p}(\R^n)} \leq \overline C\Vert D^s u \Vert_{L^p(\R^n)} .
\end{equation}

Moreover, if $s \in (0,1)$, the constant $\overline C$ can be made independent of~$p$ and takes the form
$$ {\overline C}=\left( \frac{C}{s}+1 \right),$$
with~$C>0$ depending only on~$n$ and~$\Omega$.
\end{corollary}

\begin{proof}
{W}hen $s=1$,~\eqref{gdg} is an straightforward consequence of the Poincar\'{e} inequality. 

When~$s \in (0,1)$, we exploit
either Proposition~\ref{costa} if~$p=1$ or~\eqref{simn} if~$p \in (1, +\infty) $, obtaining that
\begin{equation*}\begin{split}&
\Vert u \Vert _{H^{s,p}(\R^n)}= \left(\Vert u \Vert _{L^p(\Omega)}^p  +\Vert D^s u \Vert _{L^p(\R^n)}^p \right)^{\frac1p}\\&\qquad\leq \left(  \left(\frac{C}{s}\right)^p+1 \right)^{\frac1p} \Vert D^s u \Vert _{L^p(\R^n)}
\le \left( \frac{C}{s}+1 \right)\Vert D^s u \Vert _{L^p(\R^n)},\end{split}
\end{equation*}
for some~$C>0$ depending only on~$n$ and~$\Omega$.
\end{proof}

\begin{proposition}\label{doverf}
Let~$p \in [1, +\infty)$ and~$\Omega$ be a bounded domain of~$\R^n$. 

Then, there exists~${C}>0$, depending only on~$n$, $p$ and~$\Omega$, such that, for all~$s\in(0,1]$
and~$R>0$ with
\begin{equation}\label{pcd}
s^{2}R^s> C ,
\end{equation} 
we have that, for any $u \in H^{s,p}_0(\Omega)$,
\begin{equation}\label{hdowe}
 \Vert D^s u \Vert_{L^p(\R^n)} \leq 2 \Vert D^s u \Vert_{L^p(B_R)} .
\end{equation}
\end{proposition}

\begin{proof}
By density, it is enough to prove the inequality for $u\in \mathcal{D}(\Omega)$.

When $s=1$ the thesis plainly follows taking~$C:={\sup_{x\in\Omega} |x|}$.

When~$s\in(0,1)$, we take
\begin{equation}\label{feutyjkrt8456}
C:=\max\left\{ 1+2\sup_{x\in\Omega}|x|, \frac{2^p\,\widetilde C}{2^p-1}\right\},\end{equation}
for a suitable~$\widetilde C$ that will be specified later on (depending only on~$n$, $p$ and~$\Omega$).

With this choice, if~$R>0$ satisfies~\eqref{pcd}, we have that~$\Omega\subset B_{R/2}$, and therefore
we can exploit Proposition~\ref{usasempre} to deduce that
\begin{eqnarray*} 
\Vert D^s u \Vert^p_{L^p(\R^n \setminus B_{R})} &\le& 2^{(n+s)p} c_s^p\|u\|^p_{L^1(\Omega)}
\int_{\R^n \setminus B_{R}}\frac{dx}{|x|^{(n+s)p} }\\&=&\frac{
2^{(n+s)p} \,c_s^p \,S_{n-1}\|u\|^p_{L^1(\Omega)}}{(n(p-1)+sp) R^{n(p-1)+sp}}\\ &
\le& \frac{
2^{(n+s)p} \,c_s^p \,S_{n-1}|\Omega|^{p-1}\|u\|^p_{L^p(\Omega)}}{(n(p-1)+sp) R^{n(p-1)+sp}}.
\end{eqnarray*}
Now we use either~\eqref{simn} if~$p\in(1,+\infty)$ or
Proposition~\ref{costa} if~$p=1$ and we see that
\begin{eqnarray*} &&
\Vert D^s u \Vert^p_{L^p(\R^n \setminus B_{R})} 
\le \frac{\widetilde C \|D^su\|^p_{L^p(\R^n)}}{s\,(n(p-1)+sp) R^{n(p-1)+sp}},
\end{eqnarray*}
for some~$\widetilde C>0$ depending only on~$n$, $p$ and~$\Omega$.

We notice that, in light of the definition of~$C$ in~\eqref{feutyjkrt8456} and the condition in~\eqref{pcd},
we have that~$R>1$, and consequently the map~${\color{red}[0,+\infty)}\ni\tau\mapsto \tau R^\tau$ is increasing.
Thus, since~$n(p-1)+sp\ge s$, we have that~$(n(p-1)+sp) R^{n(p-1)+sp}\ge s R^s$. Accordingly,
\begin{eqnarray*} &&
\Vert D^s u \Vert^p_{L^p(\R^n \setminus B_{R})} 
\le \frac{\widetilde C \|D^su\|^p_{L^p(\R^n)}}{s^2 R^{s}}\le\left(1- \frac1{2^p}\right) \|D^su\|^p_{L^p(\R^n)}.
\end{eqnarray*}

As a result,
\begin{eqnarray*}
\|D^su\|^p_{L^p(\R^n)}&=& \|D^su\|^p_{L^p(B_R)}+\Vert D^s u \Vert^p_{L^p(\R^n \setminus B_{R})}\\& \le&
\|D^su\|^p_{L^p(B_R)}+\left(1-\frac1{2^p}\right)\Vert D^s u \Vert^p_{L^p(\R^n )},
\end{eqnarray*}
which gives the desired result.
\end{proof}

	Now, let~$0< \bs < s_0 <1$ and~$\Omega$ be a bounded domain of~$\R^n$. Then, in light~\cite[Proposition~4.1]{mora}, there exists a constant~$C>0$, depending only on~$n$, $s_0$, $\bs$ and~$\Omega$, such that,
	for every~$s \in [s_0,1)$, $p\in(1, +\infty)$ and~$u \in H_0^{s,p}(\Omega)$, we have that
	\begin{equation}\label{NF}
		 \Vert D^{\bs} u \Vert_{L^p(\R^n)} \leq C \Vert D^s u \Vert_{L^p(\R^n)}.
	\end{equation}
	In the following proposition, we provide a new proof of~\eqref{NF}, 
	and in fact we extend the result to the case~$p=1$, which seemed not
	to be covered in the available literature.
	Moreover, we show that the dependence of~$C$ on~$s_0$ can be dropped, which seems to be also new.
	 
		\begin{proposition}\label{crizzo}
		Let~$\bs \in (0,1)$ and~$\Omega$ be a bounded domain of~$\R^n$.
			
			 Then, there exists~$C>0$,
			 depending only on~$n$, $\bs$ and~$\Omega$,
			 such that, for any~$s \in [\bs,1)$, $p \in [1,+\infty)$ and~$u \in H_0^{s,p}(\Omega)$, 
			\begin{equation}\label{re}
				\Vert D^{\bs}u \Vert_{L^p(\R^n)} \leq C\Vert D^s u \Vert_{L^p(\R^n)}. 
			\end{equation}
		\end{proposition}
		
		\begin{proof}
If~$s = \bar{s}$, the desired result holds true by taking~$C=1$, thus from now on
we suppose that~$s\in (\bar{s},1)$. 
In this case, we first establish Proposition~\ref{crizzo} for~$u \in \mathcal{D}(\Omega)$ and then we apply a density argument to complete the proof. 
						
			We take~$R>0$ such that~$\Omega \subset B_{R}$. Also, recalling~\eqref{relation}, we see that
			\[ \gamma_{s-\bs} = \frac{n-(s- \bs)}{c_{1-(s-\bs)}},\]
and therefore, in light of~\eqref{constdef},			
			\begin{equation}\label{fgh} \widetilde\gamma:= \sup_{ s \in (\bs,1)} \frac{n-(s- \bs)}{\gamma_{s-\bs}(s-\bs)} =
				\sup_{ s \in (\bs,1)} \frac{c_{1-(s-\bs)}}{s-\bs} < +\infty. 
				\end{equation}
					
	Now we prove the desired result for~$p=1$.
To this end, we define
	\begin{align*}
	 &A_I := \frac{1}{\gamma_{s-\bs}} \int_{B_{2R}} \left( \int_{B_{4R}} \frac{|D^su(y)|}{|x-y|^{n-(s-\bs)}} \, dy \right) \, dx, \\
	  &A_{II} := \frac{1}{\gamma_{s-\bs}} \int_{B_{2R}} \left( \int_{\R^n \setminus B_{4R}} \frac{|D^su(y)|}{|x-y|^{n-(s-\bs)}} \, dy \right) \, dx \\ {\mbox{and }}\qquad
	 &B := \int_{\R^n \setminus B_{2R}} |D^{\bar{s}} u (x)| \, dx.
	\end{align*}
	We first estimate~$A_I$: making use of~\eqref{fgh}, we find that 			
\begin{equation}\label{jeiwoytgrej768905}\begin{split}
A_I &\leq   \frac{1}{\gamma_{s-\bs}}  \left( \int_{B_{6R}} \frac{dz}{|z|^{n-(s-\bs)}} \right)\Vert D^s u \Vert_{L^1(\R^n)}\\&
= \frac{S_{n-1}(6R)^{s- \bar{s}}}{\gamma_{s-\bs}(s-\bar{s})} \Vert D^s u \Vert_{L^1(\R^n)}\\
&\le  \frac{6S_{n-1}\widetilde\gamma\max\{1,R\} }{n-(s-\bs)} \Vert D^s u \Vert_{L^1(\R^n)}\\&
\le  \frac{C_1 }{\bs} \Vert D^s u \Vert_{L^1(\R^n)},
\end{split}\end{equation} for some~$C_1>0$, depending only on~$n$ and~$\Omega$.

Now we estimate~$A_{II}$. 
To this end, we set~$\Tilde{c}:= \sup_{s \in(-1,1)} c_{s}$, which is finite, thanks to~\eqref{constdef}.
As a result, 
exploiting Propositions~\ref{usasempre} and~\ref{costa},
\begin{equation}\label{jeiwoytgrej7689052}\begin{split}
			A_{II}& \leq \frac{ 2 ^{n+s}\,\Tilde{c}}{\gamma_{s-\bar{s}}} \left[ \int_{\R^n \setminus B_{4R}} \frac{1}{|y|^{n+s}} \left( \int_{B_{2R}} \frac{dx}{|x-y|^{n-(s-\bar{s})}} \right) \, dy \right] \Vert u \Vert_{L^1(\Omega)}	 \\ 
& \leq \frac{ 2^{3n+\bs}\,\Tilde{c} \,S_{n-1} R^n}{\gamma_{s-\bar{s}}} \int_{\R^n \setminus B_{4R}} \frac{dy}{|y|^{2n + \bar{s}}}\,\Vert u \Vert_{L^1(\Omega)}  \\&
\le\frac{ 2^{n}\,\Tilde{c} \,S^2_{n-1}}{\gamma_{s-\bar{s}} \, R^{\bs} } \,\Vert u \Vert_{L^1(\Omega)} \\
& \le 
\frac{ 2^{n}\,\Tilde{c} \,\widetilde\gamma\,S^2_{n-1}(s-\bs)}{(n-(s-\bs)) R^{\bs} } \,\Vert u \Vert_{L^1(\Omega)}\\& \le
\frac{ 2^{n}\,\Tilde{c} \,\widetilde\gamma\,S^2_{n-1}}{\bs\,\min\{1, R\} } \,\Vert u \Vert_{L^1(\Omega)}\\&
\le\frac{ C_2}{\bar{s}^2} \Vert D^s u \Vert_{L^1(\R^n)},
			\end{split}\end{equation} for some~$ C_2>0$, depending
			only on~$n$ and~$\Omega$.

In order to estimate~$B$, we use Proposition~\ref{usasempre} and~\ref{costa}
to obtain that
\begin{equation}\label{jeiwoytgrej7689053}\begin{split}&
B \leq 2^{n+\bar{s}} \Tilde{c} \left( \int_{\R^n \setminus B_{2R}}\frac{dx}{|x|^{n+\bs}}\right) \Vert u \Vert_{L^1(\Omega)} 
\\&\qquad= \frac{ S_{n-1} 2^{n} \Tilde{c}}{\bar{s}\,R^{\bar{s}}}  \Vert u \Vert_{L^1(\Omega)}
\leq \frac{C_3}{\bar{s}^2} \Vert  D^s u \Vert_{L^1(\R^n)},\end{split}
\end{equation} for some~$C_3>0$, depending only on~$n$ and~$\Omega$.
			
Now we recall that, in light of Theorem~\ref{usacchialo}, $D^{\bar{s}}u =I_{s-\bar{s}} D^s u$, and thus, recalling the definition of~$I_{s-\bar{s}}$ 
in~\eqref{riesz}
\begin{eqnarray*}&& \Vert D^{\bar{s}}u \Vert_{L^1(\R^n)}=\int_{\R^n}| D^{\bar{s}}u(x)|\,dx
=\int_{\R^n}| K_{s-\bar s} \ast D^su(x)|\,dx\\&&\qquad\qquad
=\frac{1}{\gamma_{s-\bar s}}\int_{\R^n}\left|
\Int \frac{D^su(y)}{|x-y|^{n-(s-\bar s)}} \, dy\right|\,dx.\end{eqnarray*}
As a consequence of this and the estimates provided in~\eqref{jeiwoytgrej768905},
\eqref{jeiwoytgrej7689052} and~\eqref{jeiwoytgrej7689053}, we conclude that
$$
\Vert D^{\bar{s}}u \Vert_{L^1(\R^n)} \leq A_{I} + A_{II} + B\le\frac{C_4}{\bs}\left(1+\frac{1}{\bs}\right) \Vert  D^s u \Vert_{L^1(\R^n)},
$$ for some~$C_4>0$, depending only on~$n$ and~$\Omega$.
Hence, the proof of the desired claim
is complete for~$p=1$ and~$u \in \mathcal{D}(\Omega)$.
			
			We now establish the desired result for~$p \in (1, +\infty)$.
			For this, we define
			\begin{align*}
			&A_{I}(p):= \frac{2^p}{\gamma^p_{s-\bs}} \int_{B_{2R}} \left( \int_{B_{4R}} \frac{|D^su(y)|}{|x-y|^{n-(s-\bs)}} \, dy\right)^p \, dx,  \\
			&A_{II}(p):= \frac{2^p}{\gamma^p_{s-\bs}} \int_{B_{2R}} \left (\int_{\R^n \setminus B_{4R}} \frac{|D^su(y)|}{|x-y|^{n-(s-\bs)}} \, dy\right)^p \, dx \\ {\mbox{and }} \qquad
			&B(p):= \int_{\R^n \setminus B_{2R}} |D^{\bs}u(x)|^p \, dx.
			\end{align*}
			We first estimate~$A_I(p)$. For this purpose, we observe that, if~$x \in B_{2R}$, setting~$p':=\frac{p}{p-1}$ and employing the H\"older inequality,
			\begin{eqnarray*}&&
				\int_{B_{4R}} \frac{|D^su(y)|}{|x-y|^{n-(s-\bs)}} \, dy\\& =&
				 \int_{B_{4R}} \frac{|D^su(y)|}{|x-y|^{\frac{n-(s-\bs)}{p}}} \frac{1}{|x-y|^{\frac{n-(s-\bs)}{p'}}} \, dy  \\ 
				 &\leq&  \left(\int_{B_{6R}} \frac{dz}{|z|^{n-(s-\bs)}} \right)^{\frac{1}{p'}} \left(\int_{B_{4R}} \frac{|D^s u(y)|^p}{|x-y|^{n-(s-\bs)}} \, dy \right)^{\frac{1}{p}}
				  \\ &= &\left( \frac{ S_{n-1} (6R)^{s-\bar{s}}}{s-\bs}\right)^{\frac{1}{p'}} \left(\int_{B_{4R}} \frac{|D^s u(y)|^p}{|x-y|^{n-(s-\bs)}} \, dy \right)^{\frac{1}{p}}  
\\ &\le &
\frac{ 6S_{n-1} \max\{1,R\} }{(s - \bar{s})^{\frac{1}{p'}}} \left(\int_{B_{4R}} \frac{|D^s u(y)|^p}{|x-y|^{n-(s-\bs)}} \, dy \right)^{\frac{1}{p}}. 
\end{eqnarray*}
Consequently, 
\begin{equation*}\begin{split}
A_I(p)& \leq \frac{(12 S_{n-1} \max\{1,R\})^p}{\gamma^p_{s-\bs}(s-\bs)^{\frac{p}{p'}}} \int_{B_{4R}} \left(\int_{B_{R_2}} \frac{dx}{|x-y|^{n-(s-\bs)}}  \right) |D^s u(y)|^p\, dy \\ &
=\frac{12^{p+s-\bar{s}} S_{n-1}^{p+2} \max\{1,R\}^p\,R^{s-\bar s}}{\gamma^p_{s-\bs}(s-\bs)^{\frac{p}{p'}+1}}  \Vert D^s u \Vert^p_{L^p(\R^n)}
\\&=\frac{12^{p+s-\bar{s}} S_{n-1}^{p+2} \max\{1,R\}^p\,R^{s-\bar s}}{\big(\gamma_{s-\bs}(s-\bs)\big)^p} \Vert D^s u \Vert^p_{L^p(\R^n)}\\&
\le\frac{12^{p+1} S_{n-1}^{p+2} \max\{1,R\}^{p+1}\,\widetilde\gamma^p}{\big(n-(s-\bar s)\big)^p} \Vert D^s u \Vert^p_{L^p(\R^n)}\\&\le
\frac{12^{p+1} S_{n-1}^{p+2} \max\{1,R\}^{p+1}\,\widetilde\gamma^p}{\bar{s}^p} \Vert D^s u \Vert^p_{L^p(\R^n)}
.
			\end{split}\end{equation*}
			
			Now we estimate~$A_{II}(p)$. By Proposition~\ref{usasempre} and equation~\eqref{simn}, we see that
\begin{equation*}\begin{split}
A_{II}(p) &\leq \frac{2^{p+n+s} \Tilde{c}}{\gamma^p_{s-\bar{s}}}  \left[ \int_{B_{2R}} \left( \int_{\R^n \setminus B_{4R}} \frac{dy}{|y|^{n+s} |x-y|^{n-(s-\bs)}}\right)^p \, dx \right] \Vert u \Vert_{L^1(\Omega)}^p \\ & 
\leq  \frac{2^{p+n+s+p(n-s+\bar s)} \Tilde{c}}{\gamma^p_{s-\bar{s}}}
\left(\int_{\R^n \setminus B_{4R}} \frac{dy}{|y|^{2n+ \bs}}\right)^p \Vert u \Vert_{L^1(\Omega)}^p \\&\le
\frac{2^{p+n+s+p(n-s+\bar s)} \Tilde{c}\,S_{n-1}^p|\Omega|^{p-1}}{(4R)^{p(n+\bar s)}\gamma^p_{s-\bar{s}}} \Vert u \Vert_{L^p(\Omega)}^p\\&  \leq
\frac{2^{p+n+s+p(n-s+\bar s)} \Tilde{c}\,S_{n-1}^p|\Omega|^{p-1}\widetilde\gamma^p}{(4R)^{p(n+\bar s)}\bar{s}^p} \Vert u \Vert_{L^p(\Omega)}^p\\&\le 
\frac{C\,2^{p+n+s+p(n-s+\bar s)} \Tilde{c}\,S_{n-1}^p|\Omega|^{p-1}\widetilde\gamma^p}{(4R)^{p(n+\bar s)}\bar{s}^{2p}} \Vert D^su \Vert_{L^p(\Omega)}^p
,\end{split}\end{equation*}
where~$C$ depends only on~$n$ and~$\Omega$.
			
Now we estimate~$B(p)$. We will proceed in a similar way as for~$p=1$,
using Proposition~\ref{usasempre} and equation~\eqref{simn}. The details go as follows:
\begin{equation*}\begin{split}
B(p)& \leq (2^{n+ \bar{s}} \Tilde{c})^p  \left( \int_{\R^n \setminus B_{2R}} \frac{dx}{|x|^{p(n+ \bs)}}\right) \Vert u \Vert^p_{L^1(\Omega)} \\
& =  \frac{(2^{n+ \bar{s}} \Tilde{c})^p S_{n-1} }{(n(p-1)+ \bar{s}p) R^{n(p-1)+ \bar{s}p}} \Vert u \Vert^p_{L^1(\Omega)}\\&\leq 
\frac{(2^{n+ \bar{s}} \Tilde{c})^p S_{n-1}\,|\Omega|^{p-1} }{(n(p-1)+ \bar{s}p) R^{n(p-1)+ \bar{s}p}} \Vert u \Vert^p_{L^p(\Omega)}
\\& \le \frac{C(2^{n+ \bar{s}} \Tilde{c})^p S_{n-1}\,|\Omega|^{p-1} }{(n(p-1)+ \bar{s}p) R^{n(p-1)+ \bar{s}p}\bar{s}^p}  \Vert D^s u \Vert^p_{L^p(\R)}.
\end{split}\end{equation*}

Gathering these estimates and recalling
Theorem~\ref{usacchialo}, we conclude that
\begin{equation}\label{eiwoty437683869wejfo}\begin{split}&
\Vert D^{\bs} u \Vert^p_{L^p(\R^n)} \le A_{I}(p) + A_{II}(p) + B(p)\\
&\qquad \le\left(
\frac{12^{p+1} S_{n-1}^{p+2} \max\{1,R\}^{p+1}\,\widetilde\gamma^p}{\bar{s}^p}\right.\\&\qquad\qquad\qquad
+\frac{C\,2^{p+n+s+p(n-s+\bar s)} \Tilde{c}\,S_{n-1}^p|\Omega|^{p-1}\widetilde\gamma^p}{(4R)^{p(n+\bar s)}\bar{s}^{2p}}\\ &\qquad\qquad\qquad \left. 
+\frac{C(2^{n+ \bar{s}} \Tilde{c})^p S_{n-1}\,|\Omega|^{p-1} }{(n(p-1)+ \bar{s}p) R^{n(p-1)+ \bar{s}p}\bar{s}^p} \right) \Vert D^s u \Vert^p_{L^p(\R)}.
\end{split}\end{equation}

Now we observe that, for all~$a$, $b\ge0$, we have that~$a^p+b^p\le(a+b)^p$, therefore,
for every~$j\in\N$ and~$a_1,\dots, a_j\ge0$, we have that
$$ \sum_{i=1}^j a_i^p\le \left( \sum_{i=1}^j a_i\right)^p.$$
Using this into~\eqref{eiwoty437683869wejfo}, we thereby obtain that
\begin{equation}\label{dewiosdvjt473kipipoiupuoy}\begin{split}
\Vert D^{\bs} u \Vert_{L^p(\R^n)}& \le\left(
\frac{12^{p+1} S_{n-1}^{p+2} \max\{1,R\}^{p+1}\,\widetilde\gamma^p}{\bar{s}^p}\right.\\ &\qquad \qquad
+\frac{C\,2^{p+n+s+p(n-s+\bar s)} \Tilde{c}\,S_{n-1}^p|\Omega|^{p-1}\widetilde\gamma^p}{(4R)^{p(n+\bar s)}\bar{s}^{2p}}\\ &\qquad\qquad \left. 
+\frac{C(2^{n+ \bar{s}} \Tilde{c})^p S_{n-1}\,|\Omega|^{p-1} }{(n(p-1)+ \bar{s}p) R^{n(p-1)+ \bar{s}p}\bar{s}^p} \right)^{\frac1p} \Vert D^s u \Vert_{L^p(\R)}\\
&\le\left(
\frac{12^{1+\frac1p} S_{n-1}^{1+\frac2p} \max\{1,R\}^{1+\frac1p}\,\widetilde\gamma}{\bar{s}}\right.\\ &\qquad \qquad
+\frac{C^{\frac1p}\,2^{1+n-s+\bar s+\frac{n+s}p} \Tilde{c}^{\frac1p}\,S_{n-1}|\Omega|^{\frac{p-1}p}\widetilde\gamma}{(4R)^{n+\bar s}\bar{s}^{2}}\\ &\qquad \qquad \left. 
+\frac{C^{\frac1p} 2^{n+ \bar{s}} \Tilde{c}\, S_{n-1}^{\frac1p}\,|\Omega|^{\frac{p-1}p} }{(n(p-1)+ \bar{s}p)^{\frac1p} R^{\frac{n(p-1)+ \bar{s}p}p}\bar{s}} \right) \Vert D^s u \Vert_{L^p(\R)}
.\end{split}\end{equation}

Moreover, we point out that, for all~$M\ge0$, we have that
$$\min\{1,M\}\le M^{\frac1p}\le \max\{1,M\}.$$
{F}rom this and~\eqref{dewiosdvjt473kipipoiupuoy}, we infer that
$$ \Vert D^{\bs} u \Vert_{L^p(\R^n)}\le \widetilde C\Vert D^s u \Vert_{L^p(\R)},$$
for some~$\widetilde C>0$, depending only on~$n$, $\bs$ and~$\Omega$, as desired.
					
{N}ow we complete the proof Proposition~\ref{crizzo} by
employing a density argument. More precisely, let~$u \in H^{s,p}_0(\Omega)$ and~$(\phi_k)\in\mathcal{D}(\Omega)$ converging to~$u$ in~$u \in H^{s,p}_0(\Omega)$ as~$k\to+\infty$. Then, for all~$k\in\N$,
\begin{equation}\label{nevren}
				\Vert D^{\bs}\phi_k \Vert_{L^p(\R^n)} \leq C \Vert D^s \phi_k \Vert_{L^p(\R^n)}. 
\end{equation}
Furthermore, since~$(D^s \phi_k)$ is a Cauchy sequence in $L^p(\R^n, \R^n)$, we see that $(D^{\bar{s}}\phi_k)$ is a Cauchy sequence in $L^p(\R^n, \R^n)$ too,
thanks to~\eqref{nevren}.
As a result, the thesis plainly follows passing to the limit as $k \to +\infty$ in~\eqref{nevren}.
\end{proof}
		
In addition to Proposition~\ref{crizzo}, we have the following statement, that takes
into account the~$L^p$-norm of the classical gradient. 
		
		\begin{proposition} \label{capire}
		Let~$s \in (0,1]$, $p \in [1,+\infty)$ and~$\Omega$ be a bounded domain of~$\R^n$. 
			
			Then, there exists~$C>0$, depending only on~$n$, $p$ and~$\Omega$, such that, for all~$u \in  H_0^{1,p}(\Omega) $,
			\begin{equation*}
\Vert D^s u \Vert_{L^p(\R^n)} \leq \frac{C}{s} \Vert Du \Vert_{L^p(\Omega)}.
\end{equation*}
		\end{proposition}
		
		\begin{proof}
		If $s=1$, then the claim is obviously true.
		If instead~$s \in (0,1)$, we recall~\cite[Proposition~2.7]{mora} to conclude that $$ \Vert D^s u \Vert_{L^p(\R^n)} \leq \frac{C}{s} \Vert u \Vert_{H^{1,p}(\R^n)}=
\frac{C}{s} \Vert u \Vert_{H^{1,p}(\Omega)},$$
for some~$C>0$ depending only on~$n$ and~$p$.
			
This and the Poincar\'{e} inequality yield the desired result.
\end{proof}

\section{The function spaces $L^p(h, \Omega)$ and $H^{0}(\am, g, \Omega)$.}\label{ipo} 

In this section we introduce the function spaces related to the operator~$\mathcal{L}$ in~\eqref{rawop}.

Let $\Omega$ be a bounded domain of~$\R^n$, $t \in[ 1,+\infty]$
and~$h:\Omega\to[0,+\infty]$ be a measurable function such that~$h^{-1} \in L^t(\Omega)$.  For any~$p \in [1, +\infty]$,
we define the norm
	\begin{align*}
				&\Vert u \Vert_{L^p(h, \Omega)} = \begin{cases}
					\left( \displaystyle\int_{\Omega} h (x)|u(x)|^p \ dx \right)^{\frac{1}{p}}  & \mbox{ if } p \in [1, +\infty) ,\\  \displaystyle
					\esssup_{x \in \Omega} |u(x)|  & \mbox{ if } p = +\infty,
				\end{cases} 
			\end{align*}
			where the~$\esssup$ is intended with respect to the measure~$\nu(A):= \int_{A} h(x) \, dx$, for any set~$A \subset \Omega$.

Then, we write~$L^p(h, \Omega)$ for the Banach space of functions satisfying $\| u \|_{L^p(h, \Omega)}<+\infty$.

\begin{remark}\label{vozzi}
We point out that the space~$L^p(h, \Omega)$ is not empty. Indeed, for any~$p \in [1, +\infty)$, let
\begin{equation}
A_{h}:= \left\{ x\in \Omega: h^{-1}(x) \geq 1 \right\} \qquad\mbox{and}\qquad B_{h}:= \left\{ x\in\Omega: h^{-1}(x) < 1 \right\}.
\end{equation}
Then, if~$u:\Omega\to\R$ is a measurable function satisfying
\begin{equation}
|u(x)| \leq C \begin{cases}
h^{- \frac{1+t}{p}}(x) &\mbox{if } x\in A_h, \\
h^{- \frac{1}{p}}(x) &\mbox{if } x\in B_h,
\end{cases}
\end{equation}
for some constant~$C>0$,  we have that
\begin{equation*}
\int_{\Omega} h(x) |u(x)|^{p} 
\leq C \left( \int_{A_h} h^{-t}(x) \, dx + \int_{B_h} \, dx\right) < +\infty.
\end{equation*} 
Thus, $u$ belongs to~$L^p(h, \Omega)$. 

Moreover, we notice that any measurable function~$u:\Omega\to\R$ such that~$ |u(x)| \leq C h^{-1}(x)$ for all~$x\in\Omega$, for some constant~$C>0$, lies in~$L^{\infty}(h, \Omega)$.

Furthermore, if~$h \in L^1_{\rm loc}(\Omega)$, then~$\mathcal{D}(\Omega) \subset L^p(h, \Omega)$ for any~$p \in [1, +\infty)$.
\end{remark} 

In this setting, we have the following result: 
			
\begin{lemma}\label{gt}
Let~$\Omega$ be a bounded domain of~$\R^n$, $t\in [ 1,+\infty]$ and~$h^{-1} \in L^t(\Omega)$. Let also
\begin{equation}\label{ATVT0}
p \in \left[\frac{t+1}t, +\infty\right)\end{equation}
				
				Then, for all measurable functions~$u:\Omega\to\R$,
				\[ \Vert u \Vert_{L^{\frac{pt}{t+1}}(\Omega)} \leq  \Vert h^{-1} \Vert^{\frac{1}{p}}_{L^t(\Omega)}  \Vert u \Vert_{L^p(h,\Omega)}.\]
			\end{lemma}
			
			\begin{proof}
			We set	\begin{equation}\label{ATVT}q :=
				\frac{pt}{t+1}
\end{equation} and notice that~$q\ge1$, thanks to~\eqref{ATVT0}.
			
 If~$t\in[1,+\infty)$, we apply the H\"older inequality with exponents~$\alpha:=\frac{t+1}t$ and~$\beta:=t+1$
			and find that
				\begin{eqnarray*}
\into |u(x)|^q \, dx &=& \into |u(x)|^q h^{\frac{1}{\alpha}}(x) h^{-\frac{1}{\alpha}}(x) \, dx \\&\leq& \left(\into |u(x)|^{q \alpha} h(x) \, dx \right)^{\frac{1}{\alpha}} \left( \into h^{-\frac{\beta}{\alpha}}(x) \, dx \right)^{\frac{1}{\beta}}  \\&
					=& \left( \into |u(x)|^{\frac{q(t+1)}{t}} h(x) \, dx \right)^{\frac{t}{t+1}}  \left( \into  \left|h(x) \right|^{-t} \, dx \right)^{\frac{1}{t+1}}.
				\end{eqnarray*}
				The desired result follows from~\eqref{ATVT}.
				
When~$t=+\infty$, we have that~$q=p$, and therefore
\begin{eqnarray*}&&\into |u(x)|^q \, dx =\into |u(x)|^p \, dx = \into h^{-1}(x)h(x)|u(x)|^p \, dx\\&&\qquad\qquad\le
\|h^{-1}\|_{L^\infty(\Omega)}  \Vert u \Vert_{L^p(h,\Omega)}^p,\end{eqnarray*}
which completes the proof.
			\end{proof}

Now we deal with the definition of the space~$H^0(\mathcal{A}, g, \Omega)$.
For this, we let~$\delta>0$ be given by~\eqref{a_condi} and set
				\begin{equation}\label{piz}
					p(\delta):=  \frac{1+\delta}{1+\frac{\delta}{2}}.
				\end{equation}
We stress that~$p(\delta)\in(1,2)$.

We recall the assumptions made on~$\mathcal{A}$ and on~$\mu(s)$ in Section~\ref{hyp_section}.
\begin{lemma}\label{wieg}
Let~$\Omega$ be a bounded domain of $\R^n$ and~$g \in L^1(\Omega, [0,+\infty))$. 

Then, the bilinear form
				\begin{equation}\label{rc}\begin{split}&
					\langle u, v \rangle_{H^{0}(\mathcal{A}, g, \Omega)}\\& := \int_{\R^n} \int_{(0,1]} a^{ij}_S(s,x) D_i^su(x) D_j^sv(x) \, d\mu(s) \, dx + \int_{\Omega} g (x)u(x)v(x) \, dx \end{split}
				\end{equation}
				defines a real scalar product on~$\mathcal{D}(\Omega)$.
\end{lemma}

\begin{proof}
Let~$s_0 \in (0,1]$ denote the minimum of the support of~$\mu$
(we stress that~$s_0>0$ since the support of~$\mu$ is assumed to be bounded away from~$0$).
Let~$R$ be the constant given by condition~\eqref{a_condi} and~$R_1>0$ be such that~$\Omega\subset B_{{R_1}/2}$.
Let~$C$ be the constant appearing in Proposition~\ref{doverf} and
\begin{equation}\label{noti354237} R_2:= 2\sup_{s\in[s_0,1]}\left(\frac{C}{s^2}\right)^{\frac1s}.\end{equation}
Let also~$\overline R:=\max\{R,R_1, R_2\}$.

We first claim that, for any~$u\in \mathcal{D}(\Omega)$,
\begin{equation}\label{940wjhfskr984hgfehgoiuerp09546}
\langle u, u \rangle_{H^{0}(\mathcal{A}, g, \Omega)}<+\infty.
\end{equation}
For this purpose, recalling~\eqref{a_condi} and Proposition~\ref{usasempre}, we estimate
\begin{equation}\begin{split}\label{pk}&
\int_{\R^n \setminus B_{\overline R}} \int_{(0,1]} \Lambda(x) |D^s u(x)|^2  \, d\mu(s) \, dx \\&\qquad\qquad\leq  \Vert{\phi}\Vert^2_{L^1(\Omega)} \int_{\R^n \setminus B_{\overline R}} \int_{(0,1]} \frac{4^{n+s}c_s^2\Lambda(x)}{|x|^{2n+2s}}  \, d\mu(s) \, dx \\ &\qquad\qquad
\leq 4^{n+1}\sup_{s\in(0,1]}c_s^2\,\mu([s_0,1]) \int_{\R^n \setminus B_{\overline R}} \frac{\Lambda(x)}{|x|^{2n}}  \, dx \\ &\qquad\qquad \leq  C  4^{n+1}\sup_{s\in(0,1]}c_s^2\,\mu([s_0,1])\int_{\R^n \setminus B_{\overline R}} \frac{dx}{|x|^{2n-p}} \\&\qquad\qquad < +\infty.
\end{split}\end{equation}

Moreover, 
we observe that, in light of~\eqref{a_condi},
\begin{eqnarray*}&&
\Vert \Lambda \Vert_{L^1(\overline B_R)}= \int_{B_{\overline R}}|\Lambda(x)|\,dx=
\int_{B_{R}}|\Lambda(x)|\,dx+\int_{B_{\overline R}\setminus B_R}|\Lambda(x)|\,dx\\&&\qquad\qquad\le \Vert \Lambda \Vert_{L^1(B_R)}+C\int_{B_{\overline R}\setminus B_R}|x|^p\,dx,
\end{eqnarray*}
which is a finite quantity.

Accordingly,
exploiting Lemma~\ref{drday},
\begin{equation}\begin{split}\label{dbw}&
\int_{B_{\overline R}} \int_{(0,1]} \Lambda(x) |D^s u(x)|^2  \, d\mu \, dx \leq \Vert \Lambda \Vert_{L^1(B_{\overline R})} \int_{[s_0,1]} \Vert D^s u \Vert^2_{L^{\infty}(\overline B_R)} \, d\mu \\ &\qquad
\leq \Vert \Lambda \Vert_{L^1(B_{\overline R})}\mu([s_0,1]) \sup_{s \in [s_0,1]}\Vert D^s u \Vert^2_{L^{\infty}(B_{\overline R})}\\&\qquad \leq C \Vert \Lambda \Vert_{L^1(B_{\overline R})} \mu([s_0,1])
\Vert Du \Vert^2_{L^{\infty}(\Omega)} < + \infty.\end{split}\end{equation}

As a consequence, from~\eqref{essa}, \eqref{pk} and~\eqref{dbw} we have that
\begin{eqnarray*}
\langle u, u \rangle_{H^{0}(\mathcal{A}, g, \Omega)} &\leq&  \int_{ \R^n  } \int_{[0,1]} \Lambda(x) |D^s u(x)|^2  \, d\mu (s)\, dx + \int_{\Omega} g (x)|u(x)|^2 \, dx \\&< &+ \infty ,
\end{eqnarray*}
which establishes~\eqref{940wjhfskr984hgfehgoiuerp09546}.

{N}ow, we show that the bilinear form defined in~\eqref{rc} is a scalar product. It is clearly linear and symmetric, thus we only need to show that it is positive definite. To this end, we use condition~\eqref{essa} to see that
\begin{eqnarray*} \langle u, u \rangle_{H^{0}(\mathcal{A}, g, \Omega)}&\ge&
\int_{\R^n}\int_{[s_0,1]}  \lambda(x) |D^su(x)|^2\, d\mu(s) \, dx\\&
\ge&  \int_{[s_0,1]}\int_{B_{\overline R}}  \lambda(x) |D^su(x)|^2 \, dx\, d\mu(s).\end{eqnarray*}

Moreover, taking~$t:= 1+\delta$, $p:=2$, $\Omega:=B_R$ and~$h:=\lambda$ 
in Lemma~\ref{gt} (used here with~$D^s u $ in place of~$u$), we see that
\begin{equation}\label{pot}
 \Vert D^s u \Vert^2_{L^{p(\delta)}(B_R)}\le \Vert \lambda^{-1}\Vert^{-1}_{L^{1+\delta}(B_R)} \Vert D^s u \Vert_{L^2(\lambda, B_R)}^2.
\end{equation}
Accordingly, we obtain that
$$\langle u, u \rangle_{H^{0}(\mathcal{A}, g, \Omega)}\ge
\Vert \lambda^{-1} \Vert_{L^{1+\delta}(B_{\overline R})}^{-1} \int_{[s_0,1]} \Vert D^s u \Vert^2_{L^{p(\delta)}(B_{\overline R})} \, d\mu(s).$$
Therefore, using Proposition~\ref{doverf}
(notice that the assumption in~\eqref{pcd}
is satisfied thanks to~\eqref{noti354237}), we obtain that
$$\langle u, u \rangle_{H^{0}(\mathcal{A}, g, \Omega)}\ge\frac{
\Vert \lambda^{-1} \Vert_{L^{1+\delta}(B_{\overline R})}^{-1}}2 \int_{[s_0,1]} \Vert D^s u \Vert^2_{L^{p(\delta)}(\R^n)} \, d\mu(s).$$
This and Propositions~\ref{crizzo} (used here with~$\bar s:=s_0$) and~\ref{capire} (used with~$s:=s_0$) give that
$$\langle u, u \rangle_{H^{0}(\mathcal{A}, g, \Omega)}\ge \frac{C \mu([s_0,1])
\Vert \lambda^{-1} \Vert_{L^{1+\delta}(B_{\overline R})}^{-1}}2 \Vert D^{s_0} u \Vert^2_{L^{p(\delta)}(\R^n)}.$$
Therefore, using Proposition~\ref{doverf}
(notice that the assumption in~\eqref{pcd}
is satisfied thanks to~\eqref{noti354237}), we obtain that
\begin{eqnarray*}&&
\langle u, u \rangle_{H^{0}(\mathcal{A}, g, \Omega)}\geq 
C \mu([s_0,1])
\Vert \lambda^{-1} \Vert_{L^{1+\delta}(B_{\overline R})}^{-1} 
  \Vert D^{s_0} u \Vert^2_{L^{p(\delta)}(\R^n)}  
					\end{eqnarray*}
up to renaming~$C>0$.					
					
This and Corollary~\ref{equivalentno} give that			
$$	\langle u, u \rangle_{H^{0}(\mathcal{A}, g, \Omega)}\geq C \Vert u \Vert^2_{H^{s_0, p(\delta)}(\R^n)}.	$$
Accordingly, if we have that~$\langle u, u \rangle_{H^{0}(\mathcal{A}, g, \Omega)}=0$,
recalling~\eqref{odjschv nk:SWDFP}
we find that~$\Vert u \Vert_{L^{p(\delta)}(\R^n)} =0$ and consequently~$u$ vanishes identically.
This shows
that the bilinear form in~\eqref{rc}
is positive definite.
\end{proof}

We call~$H^{0}(\mathcal{A}, g,\Omega)$
the Hilbert space associated with the scalar product given in Lemma~\ref{wieg}
and we point out that~$H^{0}(\mathcal{A}, g,\Omega)$ is endowed with the norm
$$ \| u\|_{H^{0}(\mathcal{A}, g,\Omega)} :=\sqrt{   \langle u, u \rangle_{H^{0}(\mathcal{A}, g, \Omega)}}.$$
In particular, taking~$g \equiv 0$
in Lemma~\ref{wieg} we have that the bilinear form
$$ \int_{\R^n} \int_{(0,1]} a^{ij}_S(s,x) D_i^su(x) D_j^sv(x) \, d\mu(s) \, dx$$
defines a scalar product on~$\mathcal{D}(\Omega)$ and we name the associated Hilbert space $H^{0}(\mathcal{A}, \Omega)$ and the associated norm by~$ \| u\|_{H^{0}(\mathcal{A}, \Omega)}$. 

\begin{remark}\label{eba}
We stress that, if the measure~$\mu$ coincides with the Dirac delta at~$s=1$,
the space~$H^{0}(\mathcal{A}, g, \Omega)$ coincides with the one defined by Trudinger in~\cite{trudinger}.
\end{remark}

{N}ow, we provide some embedding results for the space~$H^0(\mathcal{A}, g, \Omega)$. 
\begin{proposition}\label{ytr}
Let~$S_0 \in (0,1]$ denote the maximum of the support of~$\mu$.

Then, the space~$H^{0}(\mathcal{A}, g, \Omega)$ continuously embeds into~$H^{\bar{s},p(\delta)}_0(\Omega)$ for any $\bar{s} \in (0,S_0)$.

Moreover, if~$\mu\left(\{S_0\}\right)>0$, 
then the embedding holds for any $\bar{s}\in(0,S_0]$.
\end{proposition}

\begin{proof} Let~$s_0$ denote the minimum of the support of~$\mu$.
Let~$C$ be the constant appearing in Proposition~\ref{doverf} and take
$$ R\ge \sup_{s\in[s_0,S_0]}\left(\frac{C}{s^2}\right)^{\frac1s}.$$
In this way, $R$ satisfies the assumption
in~\eqref{pcd} for any~$s$ in the support of~$\mu$, and therefore Proposition~\ref{doverf} can be used in this setting. 

Moreover, taking~$t:= 1+\delta$, $p:=2$, $\Omega:=B_R$ and~$h:=\lambda$ 
in Lemma~\ref{gt} (used here with~$D^s u $ in place of~$u$), we see that
\begin{equation}\label{pot}
 \Vert D^s u \Vert^2_{L^{p(\delta)}(B_R)}\le \Vert \lambda^{-1}\Vert^{-1}_{L^{1+\delta}(B_R)} \Vert D^s u \Vert_{L^2(\lambda, B_R)}^2.
\end{equation}

Now, we let~$\bar{s} \in (0,S_0)$. We exploit~\eqref{essa}, \eqref{pot}
and Proposition~\ref{doverf} to find that
\begin{eqnarray*}
\Vert u \Vert^2_{H^{0}(\mathcal{A}, g, \Omega)}&\ge&
\int_{B_R}\int_{(0,S_0]} \lambda(x) |D^su(x)|^2 \, d\mu(s) \, dx \\&
\geq& \Vert \lambda^{-1} \Vert_{L^{1+\delta}(B_R)}^{-1} \int_{(0,S_0]} \Vert D^s u \Vert^2_{L^{p(\delta)}(B_R)} \, d\mu (s)\\
&   \ge &
\frac{\Vert \lambda^{-1} \Vert_{L^{1+\delta}(B_R)}^{-1}}2 \int_{(0,S_0]} \Vert D^s u \Vert^2_{L^{p(\delta)}(\R^n)} \, d\mu (s)
 \\&\ge&
\frac{\Vert \lambda^{-1} \Vert_{L^{1+\delta}(B_R)}^{-1}}2 \int_{[\bar s,S_0]} \Vert D^s u \Vert^2_{L^{p(\delta)}(\R^n)} \, d\mu (s)
.\end{eqnarray*}
{F}rom this and Propositions~\ref{crizzo} and~\ref{capire}, we conclude that
\begin{equation}\label{fhweiuthwo4twogfeqe098765}
\Vert u \Vert^2_{H^{0}(\mathcal{A}, g, \Omega)}\ge \frac{C\mu([s_0,S_0])\Vert \lambda^{-1} \Vert_{L^{1+\delta}(B_R)}^{-1}}2
\Vert D^{\bar{s}} u \Vert^2_{L^{p(\delta)}(\R^n)}.
\end{equation} 
Using Corollary~\ref{equivalentno} we thereby conclude that
$$ 
\Vert u \Vert^2_{H^{0}(\mathcal{A}, g, \Omega)}
  \geq C \Vert u \Vert^2_{H^{\bar{s}, p(\delta)}(\R^n)},
$$ as desired.

Moreover,  if~$\mu\left(\{S_0\}\right)>0$, then in particular~\eqref{fhweiuthwo4twogfeqe098765} holds true with~$\bar s:=S_0$, and
this gives the desired result.
\end{proof}

\begin{proposition}\label{qer}
Let~$S_0 \in (0,1]$ denote the maximum of the support of~$\mu$. Let~$p(\delta)$ be as defined in~\eqref{piz}.

Then, for any~$s \in (0,S_0)$, setting
$$ p(\delta)^*:=\frac{np(\delta)}{n-sp(\delta)},$$ we have that
\begin{itemize}
\item[(i)] $H^{0}(\mathcal{A}, g, \Omega)$ continuously embeds into~$L^q (\Omega)$ for any~$q$ in the range
		\begin{equation}
			\begin{cases}
				q \in [1,p(\delta)^*] & \mbox{ if } sp(\delta)<n ,\\
				q \in [1,+\infty) & \mbox{ if } sp(\delta)=n, \\
				q \in [1,+\infty] & \mbox{ if } sp(\delta)>n.
			\end{cases}
		\end{equation} 
\item[(ii)] $H^{0}(\mathcal{A}, g, \Omega)$ compactly embeds into~$L^q(\Omega)$ for any~$q$ in the range
		\begin{equation}
			\begin{cases}
				q \in [1,p(\delta)^*) & \mbox{ if } sp(\delta)<n, \\
				q \in [1,+\infty) & \mbox{ if } sp(\delta)=n, \\
				q \in [1,+\infty] & \mbox{ if } sp(\delta)>n.
			\end{cases}
\end{equation}
\end{itemize}

Moreover, if~$\mu(\{S_0\})>0$, then the above embeddings hold for any~$s \in (0,S_0]$.
\end{proposition}

\begin{proof}
Proposition~\ref{ytr} establishes a continuous embedding of~$H^{0}(\mathcal{A}, g, \Omega)$ into~$H_0^{s,p(\delta)}(\Omega)$ for any~$s\in(0,S_0)$ (and also for~$S_0$
if~$\mu(\{S_0\})>0$).

Thus, the thesis follows from
Theorem~\ref{qtppd} if~$s\in(0,1)$
(notice indeed that the assumptions in~\eqref{pranbef} and~\eqref{pran} are guaranteed to hold in cases~$(i)$ and~$(ii)$, respectively) and from the
classical Sobolev embeddings if~$s=1$.
\end{proof}


\section{Boundedness in $H^{0}(\mathcal{A}, g, \Omega)$ }\label{bounded}
				
In this section, we introduce the concepts of boundedness and compactness in~$H^{0}(\mathcal{A}, g, \Omega)$. The compactness result stated in Theorem~\ref{trente} here below provides crucial information needed to develop the Fredholm alternative for~$\mathcal{L}$.
		  
{L}et~$f$ be a nonnegative measurable function on~$\Omega$. Then, $f$ is said to be \textit{bounded} in~$H^0(\am, g, \Omega )$, if there exists a constant~$C>0$ such that, for any~$\phi \in \mathcal{D}(\Omega)$,
				\begin{equation}\label{wres}
					\int_{\Omega} f(x) \phi^2(x) \, dx \leq C \| \phi\|^2_{H^0(\am, g, \Omega )}.
				\end{equation}
				
				\begin{corollary}\label{coroinut99}
				The following statements are equivalent:
				\begin{enumerate}[(i)]
				\item $f$ is bounded in~$H^0(\am, g, \Omega )$.
				\item $H^0(\am, g, \Omega )$ continuously embeds into~$L^2(f, \Omega)$.
				\item  There exist two positive constants~$C_1$ and~$C_2$ such that, for any~$\phi \in H^0(\am, g, \Omega )$,
				\begin{equation*}
C_1 \| \phi\|_{H^{0}(\mathcal{A}, g+f, \Omega)} \leq \| \phi\|_{H^{0}(\mathcal{A}, g, \Omega)} 
\leq C_2 \|\phi\|_{H^{0}(\mathcal{A}, g+f, \Omega)}.
\end{equation*}
				\end{enumerate}
				\end{corollary}
				
\begin{proof}
We first show that \textit{(i)} implies \textit{(ii)} (the reverse implication is trivial). 
In order to do this, we show that
\begin{equation}\label{EWSTery4u509876}
{\mbox{\eqref{wres} is valid for any~$u \in H^0(\am, g, \Omega )$.}}\end{equation}
For this, we take~$u \in H^0(\am, g, \Omega )$ and, in light of
Definition~\ref{qeimsot},
we consider~$(\phi_k) \in \mathcal{D}(\Omega)$ that converges to~$u$ in~$H^0(\am, g, \Omega )$ as~$k\to+\infty$.
Thus, from~\eqref{wres} we have that, for any~$k\in\N$,
$$
\int_{\Omega} f(x) \phi_k^2(x) \, dx \leq C \| \phi_k\|^2_{H^0(\am, g, \Omega )}.
$$

Now, thanks to Proposition~\ref{qer}, we have that~$(\phi_k)$ converges a.e. in~$\Omega$ and therefore, by Fatou's Lemma,
\begin{eqnarray*}&& \int_{\Omega} f(x) u^2(x) \, dx\le
\lim_{k\to+\infty}\int_{\Omega} f(x) \phi_k^2(x) \, dx \\&&\qquad\leq C \lim_{k\to+\infty}\| \phi_k\|^2_{H^0(\am, g, \Omega )} =C\|u\|^2_{H^0(\am, g, \Omega )} ,
\end{eqnarray*} which establishes~\eqref{EWSTery4u509876}.

We now show that \textit{(ii)} implies \textit{(iii)}. We notice that, being~$f$ and~$g$
nonnegative functions, for any~$\phi  \in H^0(\am, g, \Omega )$,
\begin{equation*}
\|\phi\|_{H^{0}(\mathcal{A}, g, \Omega)} \leq \| \phi\|_{H^{0}(\mathcal{A}, g+f, \Omega)}, 
\end{equation*}				
Also, we deduce from \textit{(ii)} that, for any~$\phi  \in H^0(\am, g, \Omega )$,
			\begin{equation*}\begin{split}
\| \phi\|^2_{H^{0}(\mathcal{A}, g+f, \Omega)} &= \| \phi\|^2_{H^{0}(\mathcal{A}, g, \Omega)} + \int_{\Omega} f(x) |\phi(x)|^2 \,dx\\&\leq 
 \| \phi\|^2_{H^{0}(\mathcal{A}, g, \Omega)} +C\|u\phi\|_{H^{0}(\mathcal{A}, g, \Omega)}^2.\end{split}
\end{equation*}
Combining these observations, we obtain that \textit{(iii)} holds true.

To complete the proof of Corollary~\ref{coroinut99}, it remains to establish
that~\textit{(iii)} implies~\textit{(ii)}.
For this, we observe that, for any~$\phi \in H^0(\am, g, \Omega )$,
\begin{eqnarray*}
\int_{\Omega} f(x) \phi^2(x) \, dx \le \| \phi\|^2_{H^{0}(\mathcal{A}, g+f, \Omega)}\le
\frac1{C_1^2}\| \phi\|^2_{H^{0}(\mathcal{A}, g, \Omega)},
\end{eqnarray*}
which entails \textit{(ii)}, as desired.
\end{proof}

We now recall the definition of compact boundedness as given in Definition~\eqref{qeimsot2} and we point out that this is a stronger property than boundedness, as the next corollary points out.
				
				\begin{corollary}\label{crol}
					Let~$f$ be compactly bounded on~$H^0(\am, g, \Omega )$. Then, $f$ is bounded in~$H^0(\am, g, \Omega )$.
				\end{corollary}
				
				\begin{proof}
Taking~$\epsilon:=1$ in~\eqref{ghhg}
	and exploiting Proposition~\ref{qer}, we obtain that, for any~$\phi \in \mathcal{D}(\Omega)$
\begin{equation*}
\int_{\Omega} f(x) \phi^2(x) \, dx \leq \| \phi\|^2_{H^0(\am, g, \Omega )} + K_{1} \Vert \phi \Vert_{L^1(\Omega)}^2 \leq C \| \phi\|^2_{H^0(\am, g, \Omega )}. \quad\qedhere
\end{equation*}
\end{proof}

We stress that the converse does not hold true.
We refer the reader to Appendix~\ref{counterexa} for an example of function~$f$ that is bounded in~$H^0(\mathcal{A}, g, \Omega )$ but not compactly bounded in~$H^0(\mathcal{A}, g,\Omega )$.

Now we present a compactness result.

\begin{theorem}\label{trente}
Let~$f$ be compactly bounded in~$H^0(\am, g, \Omega )$. Then, the embedding of~$H^0(\am, g, \Omega )$ into~$L^2(f, \Omega)$ is compact.
				\end{theorem}
	
				\begin{proof}
We first check that
\begin{equation}\label{fhreuity4bv3fjsdkfgasdfghj}
{\mbox{the inequality in~\eqref{ghhg} holds true for every~$\phi\in H^0(\am, g, \Omega)$.}}
\end{equation}				
To this aim, we let~$\phi\in H^0(\am, g, \Omega)$ and~$(\phi_k)$ be a sequence
of functions in~${\mathcal{D}}(\Omega)$ such that~$\phi_k$ converges
to~$\phi$ in~$H^0(\am, g, \Omega)$ as~$k\to+\infty$. Thus, for all~$k$, we deduce from~\eqref{ghhg} that
$$ \Vert \phi_k \Vert_{L^2(f, \Omega)}^2 \leq \epsilon \Vert \phi_k \Vert^2_{H^0(\am, g, \Omega )} + K_{\epsilon} \Vert \phi_k \Vert_{L^1(\Omega)}^2 .$$
Moreover, by Proposition~\ref{qer}, we know that~$H^0(\am, g, \Omega)$ compactly embeds into~$L^1(\Omega)$,
and therefore~$\phi_k$ converges
to~$\phi$ in~$L^1(\Omega)$ as~$k\to+\infty$. These considerations and the Fatou's Lemma give that
\begin{eqnarray*}&&
\Vert \phi \Vert_{L^2(f, \Omega)}^2\le \lim_{k\to+\infty} \Vert \phi_k \Vert_{L^2(f, \Omega)}^2 \le  \lim_{k\to+\infty}\left(\epsilon \Vert \phi_k \Vert^2_{H^0(\am, g, \Omega )} + K_{\epsilon} \Vert \phi_k \Vert_{L^1(\Omega)}^2\right)\\
&&\qquad\qquad =\epsilon \Vert \phi\Vert^2_{H^0(\am, g, \Omega )} + K_{\epsilon} \Vert \phi \Vert_{L^1(\Omega)}^2,
\end{eqnarray*}
thus establishing~\eqref{fhreuity4bv3fjsdkfgasdfghj}.

Now, let~$(u_k)$ be a sequence of functions in~$H^0(\am, g, \Omega )$
that converges weakly to some~$u$ in~$H^0(\am, g, \Omega )$. 
In light of~\eqref{fhreuity4bv3fjsdkfgasdfghj}, we have that, for all~$k\in\N$,
\begin{eqnarray*} \Vert u_k - u \Vert_{L^2(f, \Omega)}^2 &\leq& \epsilon \Vert u_k - u \Vert^2_{H^0(\am, g, \Omega )} + K_{\epsilon} \Vert u_k- u \Vert_{L^1(\Omega)}^2 \\
&\leq&2 \epsilon\left( \Vert u_k \Vert^2_{H^0(\am, g, \Omega )}+\Vert u \Vert^2_{H^0(\am, g, \Omega )}\right) + K_{\epsilon} \Vert u_k- u \Vert_{L^1(\Omega)}^2\\
&\le& M\varepsilon+ K_{\epsilon} \Vert u_k- u \Vert_{L^1(\Omega)}^2,
.\end{eqnarray*}
for some~$M>0$.

This and the compact embedding of~$H^0(\am, g, \Omega)$ into~$L^1(\Omega)$ entail that
					\begin{align*}
						\lim_{k \to +\infty} \Vert \phi_k - \phi \Vert_{L^2(f, \Omega)}^2 \leq \lim_{k \to +\infty} \left( M \epsilon + K_{\epsilon} \Vert \phi_k - \phi \Vert_{L^1(\Omega)}^2 \right) = M \epsilon.
					\end{align*} 
					Letting~$\epsilon \to 0$, we conclude that
$$ \lim_{k \to +\infty} \Vert \phi_k - \phi \Vert_{L^2(f, \Omega)}^2=0.$$
This says that the embedding of~$H^0(\am, g, \Omega )$ into~$L^2(f, \Omega)$ is compact.
Since $H^0(\am, g, \Omega )$ is a Hilbert space and thus reflexive, we obtain the desired result.
\end{proof}
				
\section{The Fredholm alternative and proofs of Theorem~\ref{freddi} and Remark~\ref{remi}}\label{mainres}

This section presents the proofs of the main results stated in Theorem~\ref{freddi} and Remark~\ref{remi}.

Let us recall here that the function~$f$ defined in~\eqref{efe} is assumed to be compactly bounded in~$H^0(\mathcal{A}, \Omega)$. Also, we recall the notation for the quantities~$a^i(x)$,~$b^i(x)$ in~\eqref{fds} and the bilinear form~$(\mathcal{L}u,v)$ in~\eqref{mainop}.

We present the following observation, that provides a justification to the definition of the operator in~\eqref{mainop} as the variational formulation of the operator~$\mathcal{L}$ in~\eqref{rawop}.

\begin{theorem}\label{hdhdhd}
 Let~$\Omega$ be a bounded domain in~$\R^n$. 
Suppose that there exists a constant~$C>0$ such that
\begin{equation}\label{bimbm}
\begin{split}&
 \tilde{\mathcal{A}} :=
\displaystyle \sup_{{s\in (0,1]}\atop{ x \in \R^n}} | \mathcal{A}(s,x)| \leq C ,\qquad
 \tilde{a}:= \displaystyle\sup_{{x \in \Omega}\atop{i=1,\ldots,n} } |a^i(x) | \leq C ,\\
{\mbox{and }}\quad &\tilde{b}:= \sup_{{x \in \Omega}\atop{i=1,\ldots,n} } |b^i(x) | \leq C .
\end{split}
\end{equation}
Also, suppose that~$a^{ij}(s,\cdot)\in\mathcal{D}(\Omega)$ for all~$s\in(0,1]$, for all~$i$, $j=1,\dots,n$.

Moreover let~$a^i(s,\cdot) \in C^{\infty}(\Omega)$ for all~$s\in(0,1]$, for all~$i=1,\dots,n$. 

In addition, let~$a \in L^{1}_{\rm loc}(\Omega)$. 
 
Then, for any $u $, $\phi\in \mathcal{D}(\Omega)$,
\begin{equation}\label{hucdmdt}
\begin{split}&
\int_{\Omega}  \mathcal{L}u(x) \phi(x) \, dx \\&= \int_{\Omega} \left( \int_{(0,1]} \Big(a^{ij}(s,x) D^s_ju D^s_i \phi + a^i(s,x) u D^s_i \phi + b^i(s,x)\phi D^s_i u \Big)\, d\mu (s) \right)\, dx \\
&\quad\quad+ \int_{\Omega} a(x) u(x) \phi(x) \, dx.
\end{split} 
\end{equation}
In particular, for any~$u \in \mathcal{D}(\Omega)$, the map
\begin{equation}\label{sssss}
\mathcal{D}(\Omega) \ni  \phi \mapsto \int_{\Omega}  \mathcal{L}u \phi  \, dx.
\end{equation}
defines a distribution, namely is linear and continuous.
\end{theorem}

\begin{proof}
Let us consider~$u \in \mathcal{D}(\Omega)$. Then, by Proposition~\ref{ffffff}, we gather that, for any $i=1,\ldots,n$
and~$s\in(0,1]$,
\[ a^{ij}(s,x) D^s_j u + a^i(s,x) u \in \mathcal{D}(\Omega). \]
As a consequence, exploiting Lemma~\ref{eduldp},
\begin{eqnarray*}
&& \int_{(0,1]} \int_{\Omega}\Big( D^s_i \left(a^{ij}(s,x) D^s_j u + a^i(s,x) u \right) + b^i(s,x) D^s_i u  \Big) \phi(x)\, dx \, d\mu \\
& =& -\int_{(0,1]} \int_{\Omega}  \Big(a^{ij}(s,x) D^s_j u D^s_i \phi + a^i(s,x) u D^s_i \phi + b^i(s,x) D^s_i u \phi\Big) \, dx \, d\mu (s).
\end{eqnarray*}
Therefore, by~\eqref{bimbm} and Lemma~\ref{drday},
\begin{equation*}
\begin{split}
&\int_{(0,1]} \int_{\R^n} \left|\Big(D^s_i \left(a^{ij}(s,x) D^s_j u + a^i(s,x) u \right) + b^i(s,x) D^s_i u \Big)\phi(x)\right|\, dx \, d\mu (s)  \\
&\quad\leq C \Big(\tilde{\mathcal{A}} \|Du\|_{L^{\infty}(\Omega)} \|D\phi\|_{L^{\infty}(\Omega)} + \tilde{a} \| u \|_{L^{\infty}(\Omega)} \|D\phi\|_{L^{\infty}(\Omega)}  \\ &\qquad\qquad\qquad+ \tilde{b} \|Du\|_{L^{\infty}(\Omega)} \|\phi\|_{L^{\infty}(\Omega)}\Big),
\end{split}
\end{equation*}
which is a finite quantity.

Accordingly, we can employ the Fubini-Tonelli Theorem and see that
\begin{equation*}
\begin{split}&
 \int_{\R^n}  \phi \mathcal{L}u  \, dx\\ &= \int_{(0,1]} \int_{\R^n} \Big( -D^s_i \left(a^{ij}(s,x) D^s_j u + a^i(s,x) u \right) + b^i(s,x) D^s_i u \Big)\phi(x) \, dx \, d\mu (s)\\&\qquad\qquad+ \int_{\Omega} a(x) u \phi \, dx\\
&= \int_{(0,1]} \int_{\R^n} a^{ij}(s,x) D^s_j u D^s_i \phi + a^i(s,x) u D^s_i \phi + b^i(s,x) D^s_i u \phi \, dx \, d\mu (s)\\&\qquad\qquad+ \int_{\Omega} a(x) u \phi \, dx\\
&= \int_{\Omega} \Bigg[ \int_{(0,1]} \Big(a^{ij}(s,x) D^s_ju D^s_i \phi + a^i(s,x) u D^s_i \phi + b^i(s,x)\phi D^s_i u \Big)\, d\mu(s)\\&\qquad\qquad\qquad+ a(x) u \phi   \Bigg] \, dx,
\end{split}
\end{equation*}
this proving~\eqref{hucdmdt}.

Moreover, the map in~\eqref{sssss} is clearly linear, and its continuity follows from~\eqref{hucdmdt}.
\end{proof}

			\begin{proposition}\label{ghighi}
				Let~$K_{\am}$ be given by~\eqref{vwm}. Then, for any~$u$, $v \in H^0(\mathcal{A}, f, \Omega)$,
				\[|(\ml u, v ) | \leq \big(3 \sqrt{K_{\am}}+1\big) \Vert u \Vert_{H^0(\mathcal{A}, f, \Omega)} \Vert v \Vert_{H^0(\mathcal{A}, f, \Omega)}.\]        
			\end{proposition}

			\begin{proof}
Exploiting~\eqref{vwm},~\eqref{fds} and Lemma~\ref{spexx}
(used here with~$\xi_i:=D^s_i v$ and~$\psi_i:=a^i$, and also with~$\xi_i:=D^s_i u$ and~$\psi_i:=a^i$), we have that
				\begin{align*}&
					|(\ml u, v)| \\&\leq \sqrt{K_{\am}} \int_{\R^n} \int_{(0,1]} (D^su^T \mathcal{A}_S D^su)^{\frac{1}{2}} (D^sv^T \mathcal{A}_S D^sv)^{\frac{1}{2}} \, d\mu(s) \, dx \\ &\quad + \sqrt{K_{\am} }\int_{\R^n} \int_{(0,1]}\Big(
|u(x)| |f(x)|^{\frac{1}{2}} (D^sv^T \mathcal{A}_S D^sv)^{\frac{1}{2}} \\&\qquad\qquad\qquad+ |v(x)| |f(x)|^{\frac{1}{2}} (D^su^T \mathcal{A}_S D^su)^{\frac{1}{2}} \Big)\, d\mu(s) \, dx \\ &\quad +  \int_{\Omega}f(x) |u(x) v(x)| \, dx \\ &=A + B + C,
\end{align*}	
where
\begin{eqnarray*}
A&:=& \sqrt{K_{\am}}  \displaystyle\int_{\R^n}
\int_{(0,1]} (D^su^T \mathcal{A}_S D^su)^{\frac{1}{2}} (D^sv^T \mathcal{A}_S D^sv)^{\frac{1}{2}} \, d\mu(s) \, dx \\&&\qquad+ \int_{\Omega}f |u v| \, dx  ,\\
B&:=&\sqrt{ K_{\am} }\displaystyle\int_{\Omega} \int_{(0,1]} |u(x)| |f(x)|^{\frac{1}{2}} (D^sv^T \mathcal{A}_S D^sv)^{\frac{1}{2}} \, d\mu(s) \, dx \\ 
{\mbox{and }}\qquad C&:=& \sqrt{K_{\am}} \displaystyle\int_{\Omega} \int_{(0,1]} |v| |f|^{\frac{1}{2}} (D^su^T \mathcal{A}_S D^su)^{\frac{1}{2}} \, d\mu (s)\, dx.
\end{eqnarray*}

We first estimate~$A$. For this, we observe that, by the H\"{o}lder inequality,
\begin{align*}&
\int_{\R^n} \int_{(0,1]} (D^su^T \mathcal{A}_S D^su)^{\frac{1}{2}} (D^sv^T \mathcal{A}_S D^sv)^{\frac{1}{2}} \, d\mu(s) \, dx  \\ &\leq \left( \int_{\R^n} \int_{(0,1]} D^su^T \mathcal{A}_S D^su  \, d\mu(s) \, dx \right)^{\frac{1}{2}}  \left( \int_{\R^n} \int_{(0,1]} D^sv^T \mathcal{A}_S D^sv  \, d\mu \, dx \right)^{\frac{1}{2}}
\\ &\leq \Vert u \Vert_{H^0(\mathcal{A}, f, \Omega)} \Vert v \Vert_{H^0(\mathcal{A}, f, \Omega)}
\end{align*}
and
\begin{equation*}\begin{split}
\int_{\Omega}f |u v| \, dx &\leq  \left( \int_{\Omega} f |u|^2  \, dx \right)^{\frac{1}{2}} \left( \int_{\Omega} f |v|^2  \, dx \right)^{\frac{1}{2}} \\&\leq \Vert u \Vert_{H^0(\mathcal{A}, f, \Omega)} \Vert v \Vert_{H^0(\mathcal{A}, f, \Omega)} .\end{split}
\end{equation*}
As a consequence, 
\begin{equation}\label{jkdd}
A \leq \big(\sqrt{K_{\mathcal{A}}}+1\big) \Vert u \Vert_{H^0(\mathcal{A}, f, \Omega)} \Vert v \Vert_{H^0(\mathcal{A}, f, \Omega)}.
\end{equation}

In order to estimate~$B$, we exploit the H\"older inequality to the integral in~$\Omega$ and then the Jensen inequality to the integral in~$(0,1]$ and we have that
\begin{equation}\begin{split} \label{judch}
B &\leq \sqrt{K_{\am}} \int_{\Omega} \left(
|u| |f|^{\frac{1}{2}}  \int_{(0,1]} (D^sv^T \mathcal{A}_S D^sv)^{\frac{1}{2}} \, d\mu(s) \right) \, dx  \\ &\leq \sqrt{K_{\am}} \left(\int_{\Omega} |u|^2 |f| \, dx\right)^{\frac{1}{2}}\left(\int_{\Omega}\left( \int_{(0,1]}(D^sv^T \mathcal{A}_S D^sv )^{\frac12} \, d\mu(s)\right)^2 \, dx \right)^{\frac{1}{2}}\\ &\leq \sqrt{K_{\am}} \left(\int_{\Omega} |u|^2 |f| \, dx\right)^{\frac{1}{2}}\left(\int_{\Omega} \int_{(0,1]}D^sv^T \mathcal{A}_S D^sv  \, d\mu (s)\, dx \right)^{\frac{1}{2}}
\\ & \leq \sqrt{K_{\am}} \Vert u \Vert_{H^0(\mathcal{A}, f, \Omega)} \Vert v \Vert_{H^0(\mathcal{A}, f, \Omega)}.
\end{split}\end{equation}

By swapping the role of~$u$ and~$v$, we also estimate~$C$ as
\begin{equation*}
C \leq K_{\am} \Vert u \Vert_{H^0(\mathcal{A}, f, \Omega)} \Vert v \Vert_{H^0(\mathcal{A}, f, \Omega)}.
\end{equation*}
This, \eqref{jkdd} and~\eqref{judch} entail the desired result. 
			\end{proof}
			
			{F}rom Proposition~\ref{ghighi}, we infer that~$\mathcal{L}$ is a bounded bilinear form on $H^0(\am, f, \Omega)$, whose norm depends on~$K_{\am}$. In order to use the Lax Milgram Theorem and to develop a Fredholm alternative, we now study its coercivity.
			
\begin{proposition}\label{ghi}
There exists~$\sigma_0>0$, depending on~$\mu$ and~$K_{\am}$, such that, for any~$u \in H^0(\mathcal{A}, f, \Omega)$,
\begin{equation}\label{rds}
(\ml u,u) \geq \frac{1}{2} \Vert u \Vert^2_{H^0(\mathcal{A},\Omega)} - \sigma_0 \int_{\Omega} f |u|^2 .   
\end{equation}        
\end{proposition}
			
\begin{proof}  
We observe that, exploiting Lemma~\ref{spexx} (with~$\xi_i:=D^s_iu$ and~$\psi_i:=(a^i+b^i)\,u$),
\begin{equation*}
(a^i+b^i) \,u\, D^s_iu  \geq - |(a^i+b^i) \,u\,D^s_i u| \geq -2 \sqrt{K_{\am}}        
|f|^{\frac{1}{2}} |u| (D^su^T \mathcal{A}_S D^su)^{\frac{1}{2}}.       
\end{equation*}
Thus, we have that
\begin{equation}\label{mbfweu354278}\begin{split}
(\ml u, u) &\geq \int_{\R^n}  \int_{(0,1]}\left(  D^su^T \mathcal{A}_S D^su 
+ (a^i + b^i) \,u\,D^s_i u\right) \, d\mu(s) \,dx +\int_\Omega a |u|^2\, dx
\\&\geq\Vert u \Vert^2_{H^0(\mathcal{A},\Omega)} -2 \sqrt{K_{\am}} \int_{\Omega} \left( \int_{(0,1]} (D^su^T \mathcal{A}_S D^su)^{\frac{1}{2}} \, d\mu (s)\right) |f|^{\frac{1}{2}} |u| \, dx\\&\qquad-\int_\Omega |f| |u|^2\,dx.
				\end{split}\end{equation}
	
Now we use the Young inequality and we gather that
\[ 2 \sqrt{K_{\am}} |f|^{\frac{1}{2}} |u| (D^su^T \mathcal{A}_S D^su)^{\frac{1}{2}} \leq \frac{(D^su^T \mathcal{A}_S D^su)}{2}+ \frac{\big(
2 \sqrt{K_{\am}} |f|^{\frac{1}{2}} |u|\big)^2}{2}.\]
Plugging this information into~\eqref{mbfweu354278}, we deduce that
\begin{eqnarray*}
(\ml u, u) 
&\ge &\Vert u \Vert^2_{H^0(\mathcal{A},\Omega)} -\frac12
\int_{\Omega} \int_{(0,1]} (D^su^T \mathcal{A}_S D^su) \, d\mu (s)\, dx\\&&\qquad
-\big(2K_{\mathcal{A}}\,\mu((0,1])+1\big)\int_\Omega |f| |u|^2\,dx
\\
&\geq &
\frac{1}{2} \Vert u \Vert^2_{H^0(\mathcal{A},\Omega)} -\big(2K_{\mathcal{A}}\,\mu((0,1])+1\big) \int_{\Omega} |f| |u|^2 \, dx.
				\end{eqnarray*}
Accordingly, the desired result holds true with~$ \sigma_0:=2K_{\mathcal{A}}\,\mu((0,1])+1$.
			\end{proof}
			
With this preliminary work, we have that if~$\sigma$ is big enough, the Lax Milgram Theorem applies to the operator~$\mathcal{L}_{\sigma}(f)$ (defined in~\eqref{jnh}), as the next proposition points out. 

\begin{proposition}\label{laximi}
Let~$\sigma_0$ be given by Proposition~\ref{ghi}.

Then, for any~$\sigma > \sigma_0$,
the operator~$ \ml_{\sigma}(f)$ is a bijection from~$H^0(\am, f, \Omega)$ to its dual space.

If in addition~$f$ is bounded in~$H^0(\am, \Omega)$, then, for any~$\sigma \geq \sigma_0$, the operator~$ \ml_{\sigma}(f)$ is a bijection from~$H^0(\am, \Omega)$ to its dual space.
\end{proposition}

\begin{proof}
We exploit Proposition~\ref{ghighi} and the H\"{o}lder inequality to see that, for
any~$\sigma \geq \sigma_0$ and any~$u$, $v\in H^0(\mathcal{A},f, \Omega)$,
\begin{equation}\label{criceto}\begin{split}
|(\ml_{\sigma}(f) u, v)| &\leq |(\ml u, v)| + \sigma \int_{\Omega} f |u v| \,dx \\&\leq \big(3\sqrt{K_{\mathcal{A}}}+1+\sigma\big) \Vert u \Vert_{H^0(\mathcal{A},f, \Omega)} \Vert v \Vert_{H^0(\mathcal{A},f, \Omega)}.\end{split}
\end{equation}
In particular, if~$f$ is bounded in~$H^0(\am, \Omega)$, this gives that
$$ |(\ml_{\sigma}(f) u, v)| \leq  C\big(3\sqrt{K_{\mathcal{A}}}+1+\sigma\big) \Vert u \Vert_{H^0(\mathcal{A},\Omega)} \Vert v \Vert_{H^0(\mathcal{A},\Omega)},$$
for some~$C>0$ (recall Corollary~\ref{coroinut99}).

Additionally, by Proposition~\ref{ghi}, we find that
\begin{equation*}\begin{split}
(\ml_{\sigma}(f) u, u)& = (\ml u,u) + \sigma \int_{\Omega} f |u|^2 \, dx \\&\geq \frac12 \Vert u \Vert^2_{H^0(\mathcal{A},\Omega)} + (\sigma - \sigma_0) \int_{\Omega} f |u|^2 \, dx.\end{split}
\end{equation*}

Now, if~$\sigma>\sigma_0$, this implies that
$$ (\ml_{\sigma}(f) u, u)\ge
\min \left\{\frac12, \sigma-\sigma_0\right\} \Vert u \Vert^2_{H^0(\mathcal{A}, f, \Omega)},
$$
and, if~$f$ is bounded in~$H^0(\am, \Omega)$, that
$$ (\ml_{\sigma}(f) u, u)\ge
C\min \left\{\frac12, \sigma-\sigma_0\right\} \Vert u \Vert^2_{H^0(\mathcal{A},  \Omega)},
$$ fo some~$C>0$ (thanks to Corollary~\ref{coroinut99}).

If instead~$\sigma=\sigma_0$, we have that
$$  (\ml_{\sigma}(f) u, u)\ge\frac12 \Vert u \Vert^2_{H^0(\mathcal{A}, \Omega)}.$$
In both case, the Lax Milgram Theorem applies and the proof of Proposition~\ref{laximi} is complete.
\end{proof}
			
\begin{proof}[Proof of Theorem~\ref{freddi}] 
In lieu of Theorem~\ref{trente}, we can consider the following Hilbert Triplet
\begin{equation}\label{tripl}
H^0(\mathcal{A}, \Omega) \hookrightarrow \hookrightarrow L^2(f, \Omega) \hookrightarrow \left(H^0(\mathcal{A}, \Omega) \right)\rq{},
\end{equation}	
where~$\hookrightarrow \hookrightarrow$ denotes a compact embedding, while~$\hookrightarrow$ a continuous one.

As a consequence, thanks to Proposition~\ref{laximi}, we can rely on the Riesz-Schauder theory for compact operators (see e.g.~\cite[Theorem~1.8.10]{gazzola}) to conclude the proof.
\end{proof}

\begin{proof}[Proof of Remark~\ref{remi}]
{L}et us denote by~$\widetilde{H}^0(\mathcal{A},g,\Omega)$ the completion of~$\mathcal{D}(\Omega)$ with respect to the norm 
\begin{equation}\label{hido3w7543768tywg00} \|u\|_{\widetilde{H}^0(\mathcal{A},g,\Omega)}:=\left(\int_\Omega\left(
Du^T(x) \mathcal{A}_S (1,x)Du(x)+ g(x)| u(x)|^2
\right)\,dx\right)^{\frac12}\end{equation}
see e.g.~\cite[Equation~(1.9)]{trudinger}.
Then, we have that
\begin{equation}\label{hido3w7543768tywg}
\Vert u \Vert^2_{H^{0}(\mathcal{A}, \Omega)} \geq \mu(\{ 1\})\int_{\R^n} Du^T \mathcal{A}_S Du\, dx = \mu(\{ 1\})\Vert u \Vert^2_{\widetilde{H}^{0}(\mathcal{A}, \Omega)}
.\end{equation}

We recall that, since~$\lambda^{-1}\in L^1(\Omega)$, the set~$\widetilde{H}^0(\mathcal{A},\Omega)$
is compactly embedded into~$L^1(\Omega)$ (see the first statement in the proof of Lemma~1.6 in~\cite{trudinger}).
{F}rom this and~\eqref{hido3w7543768tywg}, we conclude that~$H^{0}(\mathcal{A}, \Omega)$ is compactly embedded into~$L^1(\Omega)$.

In light of this, we have that the proof of Theorem~\ref{trente} carries through if~$f$ is compactly bounded on~$H^{0}(\mathcal{A}, \Omega)$, this giving
that in this case the embedding of~$H^{0}(\mathcal{A}, \Omega)$ into~$L^2(f, \Omega)$ is compact.

Accordingly, we retrieve the Hilbert Triplet 
\begin{equation}
H^0(\mathcal{A}, \Omega) \hookrightarrow \hookrightarrow L^2(f, \Omega) \hookrightarrow \left(H^0(\mathcal{A}, \Omega) \right)\rq{},
\end{equation}	
where~$\hookrightarrow \hookrightarrow$ denotes a compact embedding, while~$\hookrightarrow$ a continuous one.

As a consequence, thanks to Proposition~\ref{laximi}, we can rely on the Riesz-Schauder theory for compact operators (see e.g.~\cite[Theorem~1.8.10]{gazzola}) to conclude the proof.
\end{proof}

\appendix

\section{Sufficient conditions for compact boundedness on $H^0(\mathcal{A}, g, \Omega)$}
\label{sappend}

Here we provide examples of functions that are compactly bounded on $H^0(\mathcal{A}, g, \Omega)$.

\begin{theorem}\label{mcza}
Let~$S_0\in (0,1]$ be the maximum of the support of~$\mu$. Let~$\delta$ be given by~\eqref{a_condi}.

If~$n\ge2$, assume that
\begin{equation}\label{jsd}
\delta> \frac{n-2S_0}{2S_0}.
\end{equation}
Then, for any
\begin{equation}\label{em1}
q \in \left( \frac{n(1+\delta)}{2S_0(1+\delta)-n}, + \infty \right],
\end{equation}
it holds that
\begin{equation}\label{lik}\begin{split}&
{\mbox{if~$f\in L^{q}(\Omega)$ with~$f^{-1}\in L^1(\Omega)$,}}\\& {\mbox{then~$f$ is compactly bounded on~$H^0(\mathcal{A},g,\Omega)$.}}\end{split}
\end{equation}

If~$n=1$, assume that
\begin{equation}\label{puj}
\delta (2S_0-1)> 2(1-S_0).
\end{equation}
Then, for any~$q \in \left( 1, + \infty \right]$, it holds that
\begin{equation}\label{hfm}\begin{split}&
{\mbox{if~$f\in L^{q}(\Omega)$ with~$f^{-1}\in L^1(\Omega)$,}}\\&{\mbox{then~$f$ is compactly bounded on~$H^0(\mathcal{A},g,\Omega)$.}}\end{split}
\end{equation}
\end{theorem}

\begin{proof}
We prove Theorem~\ref{mcza} only for~$q<+\infty$ since, being~$\Omega$ bounded, if~$f \in L^{\infty}(\Omega)$, then~$f\in L^q(\Omega)$ for any~$q \geq 1$.

We observe that if~$n\ge2$, then~$n>S_0 p(\delta)$,
where~$p(\delta)$ is as in~\eqref{piz}. Therefore, in this case, we can define the fractional critical exponent
$$ {p(\delta)}_{S_0}^* = 
\frac{np(\delta)}{n-S_0p(\delta)}
=\frac{2n(1+\delta)}{(2+\delta)n - 2S_0(1+\delta)}.$$
Notice that condition~\eqref{jsd} implies that~${p(\delta)}^*_{S_0}> 2$. As a consequence,
we can also define the quantity
\begin{equation*}
\overline{q}:=\frac{{p(\delta)}_{S_0}^*}{{p(\delta)}_{S_0}^*-2}= \frac{2n(1+\delta)}{4S_0(1+\delta)-2n}.
\end{equation*}

{N}ow, we point out that if~$q$ satisfies~\eqref{em1}, then~$q\in (\overline{q}, +\infty)$.
Also, we observe that, for any~$t \in (2, {p(\delta)}^*_{S_0})$, the function~$t \mapsto q(t):=t/(t-2)$ is strictly decreasing and takes value in~$(\overline{q}, +\infty)$.
Accordingly, for any~$q\in(\overline{q}, +\infty)$
there exists~$p\in (2,{p(\delta)}^*_{S_0})$ such that~$q=q(p)=p/(p-2)$.
With this choice, for any~$u \in L^p(\Omega)$ and~$f \in L^q(\Omega)$, thanks to the H\"{o}lder inequality (used here
with exponents~$p/2$ and~$q=p/(p-2)$),
\begin{equation}\label{u943075v6bgfueigwufrvbnmfdgiuer}\begin{split}&
\Vert u \Vert^2_{L^2(f,\Omega)}=\int_{\Omega} f|u|^2\,dx\le \left(\int_{\Omega}|u|^p\,dx\right)^{\frac2p}
\left(\int_\Omega |f|^q\,dx\right)^{\frac{1}{q}}\\&\qquad\qquad=\Vert u \Vert^2_{L^p(\Omega)}\Vert f \Vert_{L^q(\Omega)}.\end{split}
\end{equation}
This says, in particular, that 
\begin{equation}\label{du93075vcuhfdsgfuieryt983}
{\mbox{the embedding of~$L^p(\Omega)$ into~$L^2(f, \Omega)$ is continuous.}}\end{equation}

Now, thanks to Proposition~\ref{ytr}, we have that the space~$H^0(\mathcal{A}, g, \Omega)$ continuously embeds into $H_0^{s, p(\delta)}(\Omega)$ for any~$s \in(0,S_0)$ (and also for~$s=S_0$ if~$\mu(\{S_0\})>0$). As a consequence, Theorem~\ref{qtppd} gives that
\begin{equation}\label{proppi}
{\mbox{$H^0(\mathcal{A}, g, \Omega)$ compactly embeds into~$L^{p}(\Omega)$ for any~$p\in [1,{p(\delta)}^*_{S_0})$.}}
\end{equation}
{F}rom this and~\eqref{du93075vcuhfdsgfuieryt983}, we deduce that
\begin{equation}\label{zega}
{\mbox{$H^0(\mathcal{A}, g, \Omega)$ compactly embeds into~$L^{2}(f,\Omega)$.}}
\end{equation}

We now claim that
\begin{equation}\label{zega22}
{\mbox{$L^{2}(f,\Omega)$ continuously embeds into~$L^{1}(\Omega)$.}}
\end{equation}
For this, we employ Lemma~\ref{gt} (with~$t:=1$ and~$p:=2$) and we see that
\[ \Vert u \Vert_{L^{1}(\Omega)} \leq  \Vert f^{-1} \Vert^{\frac{1}{2}}_{L^1(\Omega)}  \Vert u \Vert_{L^2(h,\Omega)},\]
which establishes~\eqref{zega22}.

{F}rom~\eqref{zega}, \eqref{zega22} and the Ehrling Lemma, we obtain~\eqref{lik},
as desired.

If instead~$n=1$, we observe that condition~\eqref{puj} entails that we are in the case~$S_0p(\delta)<1$.
Moreover, Proposition~\ref{ytr} entails that the space~$H^0(\mathcal{A}, g, \Omega)$ continuously embeds into~$H_0^{s, p(\delta)}(\Omega)$, for any~$s \in(0,S_0)$ (and also for~$s=S_0$ if~$\mu(\{S_0\})>0$). Therefore, by Theorem~\ref{qtppd} we deduce that
\begin{equation}\label{kr}
{\mbox{$H^0(\mathcal{A}, g, \Omega) $ compactly embeds into~$L^{p}(\Omega)$  for any~$p \in [1,+\infty)$.}}
\end{equation}

{N}ow, let~$f \in L^q(\Omega)$ for some~$q\in (1, +\infty)$.
We notice that, for any~$t \in (2, +\infty)$, the function~$t \mapsto q(t):=t/(t-2)$ is strictly decreasing and takes value in~$\left(1, +\infty\right)$. With this choice, we have that~\eqref{u943075v6bgfueigwufrvbnmfdgiuer} holds true for any~$f\in L^q(\Omega)$ and~$u\in L^p(\Omega)$, and therefore~$L^p(\Omega)$ embeds continuously into~$L^2(f,\Omega)$.
This and~\eqref{kr} give that
\begin{equation}\label{dvw}
{\mbox{$H^0(\mathcal{A}, g, \Omega)$ compactly embeds into~$L^{2}(f,\Omega)$.}}
\end{equation}
Moreover, thanks to~\eqref{zega22}, we have that~$L^{2}(f,\Omega)$ continuously embeds into~$L^{1}(\Omega)$.
This, \eqref{dvw} and the Ehrling Lemma give~\eqref{hfm}.
\end{proof}

\begin{theorem}\label{zega2}
Assume that~$\mu(\{1\})>0$. 
Let~$t\in[1,+\infty]$ and suppose that~$\lambda^{-1} \in L^t(\Omega)$. 

When
\[ 1 + \frac{1}{t} < \frac{2}{n}, \] 
assume that~$f \in L^s(\Omega)$ with 
\[ \frac{1}{t} + \frac{1}{s} = \frac{2}{n}. \]
When~$t= +\infty$ and $n=2$, assume that~$f$ belongs to $L$log$L(\Omega)$.

Then, $f$ is compactly bounded on~$H^0(\mathcal{A}, \Omega)$.
\end{theorem}

\begin{proof}
We recall the definition of the norm~$\tilde{H}^0(\mathcal{A},\Omega)$ 
in~\eqref{hido3w7543768tywg00}. Moreover, from~\eqref{hido3w7543768tywg}, we know that
\begin{equation*}
\Vert u \Vert^2_{H^{0}(\mathcal{A}, \Omega)}\ge \mu(\{1\}) \Vert u \Vert^2_{\tilde{H}^{0}(\mathcal{A}, \Omega)}.
\end{equation*}
Also, by~\cite[Lemma~1.4]{trudinger}, we obtain that, for any~$\epsilon>0$ there exists~$K_{\epsilon}>0$ such that,
for any~$\phi \in \mathcal{D}(\Omega)$,
\begin{equation}\label{dmw}
\Vert \phi \Vert^2_{L^2(f, \Omega)} \leq \epsilon \|\phi\|_{\widetilde{H}^0(\mathcal{A},\Omega)} + K_{\epsilon} \Vert \phi \Vert^2_{L^1(\Omega)}.
\end{equation}
{F}rome these considerations we deduce the desired result.
\end{proof}

\section{Examples of functions that are bounded but not compactly bounded on $H^0(\mathcal{A},g,\Omega)$}\label{counterexa}

In this section we provide examples of functions that
are bounded but not compactly bounded on $H^0(\mathcal{A},g,\Omega)$, thus establishing that
the notions of boundedness and compact boundedness on~$H^0(\mathcal{A},g,\Omega)$ do not coincide.

\begin{proposition}
Let~$n\geq 2$, $\mu= \delta_{\bar{s}}$ for some~$\bar{s}\in (0,1)$ and 
\begin{equation}\label{dced}
\delta:= \frac{n-2\bar{s}}{2\bar{s}}.
\end{equation}

Then, any constant function~$f$ is bounded on~$H^0(\mathcal{A},\Omega)$ but not compactly bounded on~$H^0(\mathcal{A},\Omega)$.
\end{proposition}

\begin{proof}
We point out that, according to the specific choice of~$\mu$, the spaces $H^{\bar{s},p(\delta)}_0(\Omega)$ and~$H^0(\mathcal{A},\Omega)$ coincide.
Moreover, by~\eqref{piz} and~\eqref{dced}, we have that
\begin{equation}\label{nhbb000}
 p(\delta)=\frac{ 2+\frac{n-2\bar{s} }{\bar{s} }}{2+\frac{n-2\bar{s}}{2\bar{s}}}=\frac{2n}{n+2\bar s}
\end{equation} and thus
\begin{equation*}
p(\delta)_{\bar{s}}^* = \frac{n p(\delta)}{n-p(\delta)\bar s}
=2.
\end{equation*}
Then, by Theorem~\ref{qtppd} and Proposition~\ref{ytr}, we obtain that, for any~$\phi \in \mathcal{D}(\Omega)$,
\begin{equation}\label{dbfss}
\int_{\Omega} |\phi|^2 f \,dx\leq C \Vert \phi \Vert_{L^2(\Omega)}^2 \leq C \Vert \phi \Vert^2_{H^{s, p(\delta)}_0(\Omega)} \leq C \Vert \phi \Vert^2_{H^0(\mathcal{A},\Omega)},
\end{equation} up to changing~$C>0$,
namely, that~$f$ is bounded on~$H^0(\mathcal{A},\Omega)$.

{W}e now show that~$f$ is not compactly bounded on~$H^0(\mathcal{A},\Omega)$.
To prove this, we argue towards a contradiction and we suppose that~$f$ is compactly bounded on~$H^0(\mathcal{A},\Omega)$.

For any~$\phi \in \mathcal{D}(\Omega)$, $\lambda>0$ and~$\alpha \in \R$, we define~$\phi_{\lambda, \alpha}(x):= \lambda^{\alpha} \phi(\lambda x)$ for all~$x\in\R^n$.
Hence, changing variable~$\zeta:=\lambda y$, we see that
\begin{equation}\begin{split} \label{xop}
D^{\bar{s}} \phi_{\lambda, \alpha} \left(\frac{x}{\lambda}\right) &= \int_{\R^n} \frac{\left(\phi_{\lambda, \alpha}\left(\frac{x}{\lambda}\right)- \phi_{\lambda, \alpha}(y)\right)}{\left| \frac{x}{\lambda}-y \right|^{n+\bar{s}+1}} \left(\frac{x}{\lambda}-y\right) \, dy \\&= \lambda^{\alpha}\int_{\R^n} \frac{\left(\phi\left(x\right)- \phi(\lambda y)\right)}{\left| \frac{x}{\lambda}-y \right|^{n+\bar{s}+1}} \left(\frac{x}{\lambda}-y\right) \, dy \\&= \lambda^{\alpha+\bar{s}}\int_{\R^n} \frac{\left(\phi\left(x\right)- \phi(\zeta)\right)}{\left|x-\zeta \right|^{n+\bar{s}+1}} \left(x-\zeta\right) \, d\zeta \\& = \lambda^{\alpha+\bar{s}} D^{\bar{s}}  \phi(x).
\end{split}\end{equation}

Now, we set
\begin{equation}\label{21327496jfhffbllll540}
\bar{\alpha}:= \frac{n-\bar{s} p(\delta)}{p(\delta)}=\frac{n}2.
\end{equation}
{F}rom~\eqref{nhbb000} and~\eqref{xop} 
we infer that
\begin{equation}\begin{split}\label{mmmm}&
\Vert\phi_{\lambda, \bar{\alpha}} \Vert^{p(\delta)}_{H^0(\mathcal{A},\Omega)}= \int_{\R^n}|D^{\bar{s}} \phi_{\lambda, \bar{\alpha}}(x)|^{p(\delta)} \, dx = \lambda^{-n}\int_{\R^n} \left|D^{\bar{s}} \phi_{\lambda, \bar{\alpha}}\left(\frac{y}{\lambda}\right)\right|^{p(\delta)} \, dy  \\&\qquad\qquad
= \lambda^{p(\delta)(\bar{\alpha}+\bar{s})-n}\int_{\R^n} \left|D^{\bar{s}} \phi (y) \right|^{p(\delta)} \, dy \\&\qquad\qquad= \int_{\R^n} \left|D^{\bar{s}} \phi (y) \right|^{p(\delta)} \, dy = \Vert\phi\Vert_{H^0(\mathcal{A},\Omega)}^{p(\delta)}.
\end{split}\end{equation}
Also, by~\eqref{21327496jfhffbllll540}, we have that
\begin{equation}\label{opop}
\Vert \phi_{\lambda,\bar{\alpha}} \Vert_{L^2(f, \Omega)}^2 =  \lambda^{2 \bar{\alpha}} \int_{\Omega} f(x)|\phi(\lambda x)|^2 \, dx = \lambda^{2 \bar{\alpha} -n} \Vert \phi \Vert^2_{L^2(f,\Omega)} = \Vert \phi \Vert^2_{L^2(f,\Omega)}
\end{equation}
and that
\begin{equation}\label{jkkk}\begin{split}&
\Vert  \phi_{\lambda, \bar{\alpha}} \Vert_{L^{1}(\Omega)} = \int_{\Omega}\left|\phi_{\lambda, \bar{\alpha}}(x)\right| \, dx =\lambda^{\bar{\alpha}} \int_{\Omega} \left|\phi(\lambda x)\right| \, dx 
\\&\qquad= \lambda^{\bar{\alpha}-n} \Vert \phi \Vert_{L^1(\Omega)} = \lambda^{-\frac{n}2} \Vert \phi \Vert_{L^1(\Omega)}.\end{split}
\end{equation}

{N}ow, let~$\phi \in \mathcal{D}(\Omega)$ (not vanishing identically)
and set
\begin{equation*}
M:=\Vert \phi \Vert_{H^0(\mathcal{A},\Omega)} \qquad {\mbox{and}} \qquad L:=
\Vert \phi \Vert_{L^1(\Omega)} .
\end{equation*}
Furthermore, for all~$\lambda>0$, we set
\begin{equation}\label{mdss}
\widetilde{\phi}_{\lambda, \bar{\alpha}}: = \frac{\phi_{\lambda, \bar{\alpha}}}{\Vert \phi_{\lambda, \bar{\alpha}} \Vert_{L^2(f,\Omega)}}.
\end{equation}
We observe that, thanks to~\eqref{mmmm} and~\eqref{jkkk},
\begin{eqnarray*}&& \Vert \widetilde\phi_{\lambda, \bar{\alpha}} \Vert_{L^2(f,\Omega)}=1,\qquad
\Vert \widetilde{\phi}_{\lambda, \bar{\alpha}} \Vert_{H^0(\mathcal{A},\Omega)} =M\\&& {\mbox{and}}\qquad
\Vert \widetilde{\phi}_{\lambda, \bar{\alpha}} \Vert_{L^1(\Omega)}=
\lambda^{-\frac{n}2} \Vert \phi \Vert_{L^1(\Omega)}
.\end{eqnarray*}

Now, if~$f$ is compactly bounded on~$H^0(\mathcal{A},\Omega)$, given~$\epsilon_0 :=1/{2 M^2}$, there exists~$K_{\epsilon_0}>0$ such that
\begin{eqnarray*}
&&1=\Vert \widetilde{\phi}_{\lambda, \bar{\alpha}} \Vert_{L^2(f,\Omega)}^2  \leq \epsilon_0 \Vert \widetilde{\phi}_{\lambda, \bar{\alpha}} \Vert^2_{H^0(\mathcal{A},\Omega)} + K_{\epsilon_0} \Vert \widetilde{\phi}_{\lambda, \bar{\alpha}} \Vert^2_{L^1(\Omega)} 
\\&&\qquad \qquad= \frac1{2 M^2} M^2 + K_{\epsilon_0} \Vert \widetilde{\phi}_{\lambda, \bar{\alpha}} \Vert^2_{L^1(\Omega)} 
=\frac{1}{2} + K_{\epsilon_0}L\lambda^{-n}.\end{eqnarray*}
{F}rom this, sending~$\lambda \to +\infty$, we obtain that~$1 \leq 1/2$, which 
gives the desired contradiction.
\end{proof}

\section{Some integral computations towards the proof of Proposition~\ref{huy}}\label{tappend}

In this section we prove some integral results that are used in Proposition~\ref{huy} in order to compute the Fourier transform of the fractional gradient.

\begin{lemma}\label{vcc}
Let~$s \in (0,1)$ and~$\Gamma$ denote the Euler Gamma function. Then,
\begin{equation}
\int_{0}^{+\infty} \frac{\sin(t)}{t^{1+s}}\, dt = \frac{\Gamma(\frac{1+s}{2})\Gamma(\frac{1-s}{2})}{2\Gamma(1+s)}.
\end{equation}
\end{lemma}

\begin{proof}
For any~$t >0$, the change of variable~$z := t \tau$ gives that
\begin{equation*}
\Gamma(1+s) = \int_0^{+\infty} z^{s} e^{-z} \, dz = t^{s+1}\int_0^{+\infty}  \tau^s e^{-t \tau} \, d\tau.
\end{equation*}
Hence, denoting by~$B$ the Beta function, we have that
\begin{eqnarray*} &&
\int_0^{+\infty}  \frac{\sin t}{t^{1+s}}\, dt = \frac{1}{\Gamma(1+s)}\int_0^{+\infty}\int_0^{+\infty} \tau^s e^{-t \tau} \sin t \, d\tau \, dt\\&&\qquad = \frac{1}{\Gamma(1+s)} \int_0^{+\infty} \tau^s \left( \int_0^{+\infty} e^{-t\tau} \sin t \, dt \right)\, d\tau \\&&\qquad= \frac{1}{\Gamma(1+s)} \int_0^{+\infty}\frac{\tau^s}{\tau^2 +1} \, d\tau \\ &&\qquad= \frac{1}{2\Gamma(1+s)}\int_0^{+\infty} \frac{\zeta^{\frac{s-1}{2}}}{\zeta+1} \, d\zeta = \frac{B(\frac{1+s}{2}, \frac{1-s}{2})}{2\Gamma(1+s)},
\end{eqnarray*}
where we have used the Fubini Tonelli Theorem and applied the change of variable~$\tau := \sqrt{\zeta}$.

Now, recalling the property relating the Beta fuction to the Euler Gamma function
\begin{equation}\label{mnbvcxqwertyu1234567890}
B(z_1,z_2)=\frac{\Gamma(z_1)\Gamma(z_2)}{\Gamma(z_1+z_2)} \end{equation}
and that~$\Gamma(1)=1$, we obtain the desired result.
\end{proof}

\begin{lemma}\label{ccc}
Let~$n\geq 2$. Then,
\begin{equation}\label{xx}
\int_{\partial B_1} |\omega_1|^{1+s} \, d\mathcal{H}_{\omega}^{n-1} = 2 \pi^{\frac{n-1}{2}}  \frac{\Gamma\left(\frac{s+2}{2}\right)}{\Gamma\left(\frac{n+s+1}{2}\right)}.
\end{equation}
\end{lemma}

\begin{proof}
For any~$x \in \R^n$, we write~$x=(x_1,x') \in\R\times \R^{n-1}$.
For all~$x'\in\R^{n-1}$ with~$|x'|\le1$, we define the function~$
\psi(x')= \sqrt{1- |x'|^2}$.
Then,
\begin{equation*}
\partial B_1^+:= \left\{ \omega \in \partial B_1 : \omega_1>0 \right\}= \left\{ (\psi(\omega'), \omega') \in\R\times\R^{n-1}:|\omega'| \leq 1 \right\}.
\end{equation*}
In addition, since
\begin{equation*}
D\psi (\omega')= - \omega' \left(1-|\omega'|^2\right)^{-\frac{1}{2}},
\end{equation*}
we can write
$$
d\mathcal{H}_{\omega}^{n-1}= \sqrt{1+|D\psi(\omega')|^2} d\omega' = \sqrt{1+\frac{|\omega'|^2}{1-|\omega'|^2}} d\omega' = \left(1-|\omega'|^2\right)^{-\frac{1}{2}} d\omega'.
$$

Accordingly, using polar coordinates and the change of variable~$\rho:= \sqrt{\zeta}$,
\begin{equation}\begin{split}\label{oos}
\int_{\partial B_1} |\omega_1|^{1+s} \, d\mathcal{H}_{\omega}^{n-1} 
&= 2\int_{\partial B_1^+} \omega_1^{1+s} \, d\mathcal{H}_{\omega}^{n-1}
\\&= 2\int_{\{|\omega'| \leq 1 \}} \left(1-|\omega'|^2\right)^{\frac{1+s}{2}} \left(1-|\omega'|^2\right)^{-\frac{1}{2}}\, d\omega' \\&= 2\int_{\{|\omega'| \leq 1 \}}\left(1-|\omega'|^2\right)^{\frac{s}{2}} \, d\omega' 
\\&= 2(n-1)\omega_{n-1} \int_0^1\left(1-\rho^2\right)^{\frac{s}{2}}\rho^{n-2} \, d\rho
\\&=(n-1)\omega_{n-1} \int_0^1 (1-\zeta)^{\frac{s+2}{2}-1} \zeta^{\frac{n-1}{2}-1} \, d\zeta 
\\&= (n-1)\omega_{n-1} B\left(\frac{n-1}{2}, \frac{s+2}{2}\right),
\end{split}\end{equation} where~$B$ is the Beta function.

Now, exploiting the equality
\begin{equation*}
\omega_{n-1}= \displaystyle\frac{\pi^{\frac{n-1}{2}}}{\Gamma\left(\frac{n-1}{2}+1\right)},
\end{equation*}
the property of the Gamma function that~$\Gamma(z+1)= z \Gamma(z)$
and~\eqref{mnbvcxqwertyu1234567890},
we infer from~\eqref{oos} that
\begin{equation*}\begin{split}
\int_{\partial B_1} |\omega_1|^{1+s} \, d\mathcal{H}_{\omega}^{n-1} &= (n-1)\pi^{\frac{n-1}{2}}\frac{\Gamma\left(\frac{n-1}{2}\right)\Gamma\left(\frac{s+2}{2}\right)}{\Gamma\left(\frac{n+1}{2}\right)\Gamma\left(\frac{n+s+1}{2}\right)} 
\\&= 2 \pi^{\frac{n-1}{2}}  \frac{\Gamma\left(\frac{s+2}{2}\right)}{\Gamma\left(\frac{n+s+1}{2}\right)}. \qedhere\end{split}
\end{equation*}
\end{proof}

\begin{proposition}\label{fbp}
Let~$\xi\in \R^n$. Then, for all~$j\in\{1,\cdots,n\}$,
\begin{equation}\label{dbd}
\int_{\R^n} \frac{\sin(\xi \cdot t) t_j}{|t|^{n+s+1}} \, dt = 2^{-s}\pi^{\frac{n}{2}} |\xi|^{s-1}\xi_j\frac{\Gamma(\frac{1-s}{2})}{\Gamma(\frac{n+s+1}{2})}.
\end{equation}
\end{proposition}

\begin{proof}
Let~$\xi=(\xi_1,\cdots, \xi_n)$. Suppose that~$\xi_j=0$ for some~$j\in\{1,\dots,n\}$. In this case, we claim that
\begin{equation}\label{ffg}
\int_{\R^n} \frac{\sin(\xi \cdot t) t_j}{|t|^{n+s+1}} \, dt =0.
\end{equation}
For this, we perform the change of variable~$\tilde{t} := (t_1, \ldots, t_{j-1}, -t_j, t_{j+1}, \ldots, t_n)$ and we see that
\begin{equation}
\int_{\R^n} \frac{\sin(\xi \cdot t) t_j}{|t|^{n+s+1}} \, dt = - \int_{\R^n} \frac{\sin(\xi \cdot \tilde{t}) \tilde{t}_j}{|\tilde{t}|^{n+s+1}} \, d\tilde{t},
\end{equation}
which entails~\eqref{ffg}.

Suppose now that~$\xi_j \neq 0$ (in which case~$|\xi|\neq0$). In this case, we use the change of variable~$\tau:=|\xi|t $ to get that
\begin{equation}\label{exx}
\int_{\R^n} \frac{\sin(\xi \cdot t) t_j}{|t|^{n+s+1}} \, dt = |\xi|^{s} \int_{\R^n} \frac{\sin(\frac{\xi }{|\xi|}\cdot \tau) \tau_j}{|\tau|^{n+s+1}} \, d\tau
.\end{equation}
Now we consider a rotation matrix~$\mathcal{R}=[\mathcal{R}_{ik}]$ such that~$\mathcal{R} e_j= \frac{\xi}{|\xi|}$. Notice that~$\mathcal{R}_{ij}= \frac{\xi_i}{|\xi|}$, for all~$i\in\{1,\cdots,n\}$. Hence, 
changing variable~$\eta: = \mathcal{R}^{-1}\tau$,
\begin{equation*}\begin{split}
&\int_{\R^n} \frac{\sin(\frac{\xi }{|\xi|}\cdot \tau) \tau_j}{|\tau|^{n+s+1}} \, d\tau=
\int_{\R^n} \frac{\sin(\mathcal{R}e_j\cdot \tau) \tau_j}{|\tau|^{n+s+1}} \, d\tau \\&\qquad\quad
= \int_{\R^n} \frac{\sin(\mathcal{R}e_j\cdot \mathcal{R}\eta) 
\left(\mathcal{R}\eta \cdot e_j \right)}{|\eta|^{n+s+1}} \, d\eta \\ &\qquad\quad= \int_{\R^n} \frac{\sin(\eta_j) \left(\mathcal{R}\eta \cdot e_j \right)}{|\eta|^{n+s+1}} \, d\eta = \sum_{i=1}^n \mathcal{R}_{ji} \int_{\R^n} \frac{\sin(\eta_j) \eta_i}{|\eta|^{n+s+1}} \, d\eta  \\&\qquad\quad=\mathcal{R}_{jj} \int_{\R^n} \frac{\sin(\eta_j) \eta_j}{|\eta|^{n+s+1}} \, d\eta =  \frac{\xi_j}{|\xi|} \int_{\R^n} \frac{\sin(\eta_j) \eta_j}{|\eta|^{n+s+1}} \, d\eta.
\end{split}\end{equation*}
Plugging this information into~\eqref{exx}, we infer that
\begin{equation}\label{eqq}
\int_{\R^n} \frac{\sin(\xi \cdot t) t_j}{|t|^{n+s+1}} \, dt = |\xi|^{s-1}\xi_j \int_{\R^n} \frac{\sin(t_j) t_j}{|t|^{n+s+1}} \, dt.
\end{equation}

{F}rom now on, we consider separately the cases~$n=1$ and~$n>1$.

If~$n=1$, we exploit~\eqref{eqq} and Lemma~\ref{vcc} to find that
\begin{eqnarray*}&&
\int_{\R} \frac{\sin(\xi \cdot t) t}{|t|^{n+s+1}} \, dt = |\xi|^{s-1}\xi \int_{\R} \frac{\sin(t) t}{|t|^{n+s+1}} \, dt = 2|\xi|^{s-1}\xi \int_{0}^{+\infty} \frac{\sin(t)}{t^{n+s}} \, dt\\  &&\qquad\qquad= |\xi|^{s-1}\xi \frac{\Gamma(\frac{1+s}{2})\Gamma\left(\frac{1-s}{2}\right)}{\Gamma(1+s)} = |\xi|^{s-1}\xi \frac{2^{-s} \sqrt{\pi}\Gamma\left(\frac{1-s}{2}\right)}{\Gamma\left(\frac{2+s}{2}\right)},
\end{eqnarray*}
where the last equality exploits the Legendre duplication formula for the Gamma function
\begin{equation}\label{lege}
\Gamma(z) \Gamma\left(z+\frac{1}{2}\right)= 2^{1-2z} \sqrt{\pi} \Gamma(2z)
\end{equation}
used here above with~$z:= (s+1)/2$. This completes the proof of Proposition~\ref{fbp} for~$n=1$.

Hence, we now focus on the case~$n>1$. We set
\begin{equation*}
\partial B_1^+ = \left\{ \omega \in \partial B_1: \omega_j>0 \right\}
\end{equation*}
and, using polar coordinates and the change of variable~$r:= \rho/ \omega_j $, from~\eqref{eqq} we obtain that
\begin{equation*}\begin{split}
\int_{\R^n} \frac{\sin(\xi \cdot t) t_j}{|t|^{n+s+1}} \, dt &= |\xi|^{s-1}\xi_j \int_{\partial B_1}\omega_j \int_{0}^{+\infty} \frac{\sin(r \omega_j) }{r^{s+1}} \, dr \, d\mathcal{H}_{\omega}^{n-1}\\ &= 2|\xi|^{s-1}\xi_j \int_{\partial B_1^+}\omega_j \int_{0}^{+\infty} \frac{\sin(r \omega_j) }{r^{s+1}} \, dr \, d\mathcal{H}_{\omega}^{n-1}\\ &=2|\xi|^{s-1}\xi_j \left(\int_{\partial B_1^+}\omega_j^{1+s} \, d\mathcal{H}_{\omega} ^{n-1}\right) \left(\int_{0}^{+\infty} \frac{\sin\rho }{\rho^{s+1}} \, d\rho\right) \\ &= |\xi|^{s-1}\xi_j \left( \int_{\partial B_1}|\omega_1|^{1+s} d\mathcal{H}_{\omega}^{n-1}\right)\left( \int_{0}^{+\infty} \frac{\sin\rho }{\rho^{s+1}} \, d\rho\right),
\end{split}\end{equation*}
where we also exploited the rotation invariance of the integral in the last equality.

{F}rom this and Lemmata~\ref{vcc} and~\ref{ccc}, we obtain that
\begin{equation*}\begin{split}
\int_{\R^n} \frac{\sin(\xi \cdot t) t_j}{|t|^{n+s+1}} \, dt &= 
|\xi|^{s-1}\xi_j
\frac{2 \pi^{\frac{n-1}{2}}\Gamma\left(\frac{s+2}{2}\right)}{\Gamma\left(\frac{n+s+1}{2}\right)}  \frac{\Gamma(\frac{1+s}{2})\Gamma(\frac{1-s}{2})}{2\Gamma(1+s)}
\\ &=\frac{ \pi^{\frac{n-1}{2}}}{2} |\xi|^{s-1}\xi_j \frac{s \Gamma(\frac{1+s}{2})\Gamma\left(\frac{s}{2}\right)}{\Gamma(1+s)}  \frac{\Gamma(\frac{1-s}{2})}{\Gamma\left(\frac{n+s+1}{2}\right)}\\&=2^{-s}\pi^{\frac{n}{2}}|\xi|^{s-1}\xi_j \frac{ s\Gamma(s)}{\Gamma(1+s)}  \frac{\Gamma(\frac{1-s}{2})}{\Gamma\left(\frac{n+s+1}{2}\right)} \\&= 2^{-s}\pi^{\frac{n}{2}}|\xi|^{s-1}\xi_j \frac{\Gamma(\frac{1-s}{2})}{\Gamma\left(\frac{n+s+1}{2}\right)}.
\end{split}\end{equation*}
where we used that~$\Gamma(z+1)= z \Gamma(z)$ and the
Legendre duplication formula in~\eqref{lege} (used here with~$z:= s/2$).
This completes the proof of Proposition~\ref{fbp} for~$n>1$.
\end{proof}

\section{Properties of $\mathcal{A}$ and technical results about matrices}\label{fappend}

In this section we prove some algebraic results related to the matrix~$\mathcal{A}$, that are mostly used in Section~\ref{mainres} in the study of the operator~$\mathcal{L}$.

{W}e introduce the following norm in the vector space $\mathbb{R}^{n \times n}$: 
			\begin{equation*}
				\Vert \am \Vert := \sup_{{x\in\R^n}\atop{|x|=1}}|\am x |.
			\end{equation*} 
We will say that~$\mathcal{A}$ is bounded if~$\Vert \am \Vert <+\infty$.
			
	\begin{lemma}\label{lemon}
	Let~$\mathcal{A}$ be a positive definite matrix.
	
	If~$\mathcal{A}$ is symmetric, then~\eqref{vwm} holds true with~$\mathcal{K}_{\am} = 1$.
	
If~$\mathcal{A}$ is bounded and strictly elliptic in~$\R^n$ with constant~$c>0$, then~\eqref{vwm} holds true with
$$\mathcal{K}_{\am} = \left(\frac{\Vert \mathcal{A} \Vert}{ c}\right)^2.$$	
	\end{lemma}
	
	\begin{proof}
	If~$\mathcal{A}$ positive definite and symmetric, the form
	\begin{equation*}
	\langle \xi, \psi \rangle_{\mathcal{A}}: = \xi^T \mathcal{A} \psi \qquad\mbox{for any~$\xi$, $\psi \in \R^n$}
	\end{equation*}
	defines a scalar product in~$\R^n$. Thus, by the Cauchy-Schwarz inequality, we obtain that, for any~$ \xi$, $\psi \in \R^n$,
	\begin{equation*}
	\left|\xi^T \mathcal{A} \psi\right|^2 = \left|\langle \xi, \psi \rangle_{\mathcal{A}}\right|^2 \leq \langle \xi, \xi \rangle_{\mathcal{A}} \langle \psi, \psi \rangle_{\mathcal{A}} = \left(\xi^T \mathcal{A} \xi\right)\left(\psi^T \mathcal{A} \psi \right),
	\end{equation*}
	which is~\eqref{vwm} with~$\mathcal{K}_{\am} = 1$.
	
	If~$\mathcal{A}$ is bounded and strictly elliptic in~$\R^n$,
	we have that, for any~$ \xi$, $\psi \in \R^n$,
			\begin{equation*}\begin{split}&
				\left|\xi^T \mathcal{A} \psi\right|^2  = (\am \psi \cdot \xi) (\am^T \xi \cdot \psi) \leq \Vert \mathcal{A} \Vert^2 |\xi|^2 |\psi|^2 \\&\qquad= \frac{\Vert \mathcal{A} \Vert^2}{c^2} \left( c|\xi|^2 \right) \left( c |\psi|^2\right) \leq \left( \frac{\Vert \am\Vert}{c} \right)^2 \left(\xi^T \mathcal{A} \xi\right)\left(\psi^T \mathcal{A} \psi \right),\end{split}
			\end{equation*}
			which gives~\eqref{vwm} in this case.
\end{proof}
			
			\begin{lemma}\label{spexx}
			Let~$\mathcal{A}$ be positive definite matrix satisfying~\eqref{vwm} and let~$\mathcal{A}_S$ be its symmetric part.
Let~$\mathcal{B}=[b^{ij}]$ be the inverse matrix of~$\mathcal{A}$.

Then, for any~$\xi$, $\psi \in \R^n$, 
				\[|\xi\cdot \psi| \leq \sqrt{
				K_{\mathcal{A}} \big( \xi^T \mathcal{A}_S \xi \big) \big( 
				\psi^T \mathcal{B} \psi \big)} .\]
			\end{lemma}
			
			\begin{proof}
				For any~$\xi$, $\psi \in \R^n$, exploiting~\eqref{vwm}, we have that	\begin{eqnarray*}&&
|\xi\cdot \psi| = |  \mathcal{B}\mathcal{A}\xi\cdot\psi| = |\mathcal{A}\xi\cdot
\mathcal{B}^T\psi |
\leq \sqrt{K_{\mathcal{A}}\big( \mathcal{A}\xi\cdot\xi\big) \big(
\mathcal{A}\mathcal{B}^T\psi\cdot\mathcal{B}^T\psi\big)} 
\\&&\qquad=   \sqrt{
K_{\mathcal{A}} \big( \mathcal{A}\xi\cdot\xi\big)\big( \mathcal{B}^T\psi\cdot\psi\big)}.
			\end{eqnarray*}	Moreover,
				\[
\big( \mathcal{A}\xi\cdot\xi\big)\big( \mathcal{B}^T\psi\cdot\psi\big)
= \big( \mathcal{A}_S \xi\cdot\xi\big)\big( \mathcal{B}\psi\cdot\psi\big)
= \big( \mathcal{A}_S \xi\cdot\xi\big) \big( \mathcal{B}\psi\cdot\psi\big).\]
These observations give the desired result.
			\end{proof}

\section*{Declarations}

\begin{itemize}
\item Ethical Approval: NOT APPLICABLE.
\item Funding: this work has been supported by the Australian Research Council Laureate Fellowship FL190100081 and
by the Australian Future Fellowship FT230100333.
\item Availability of data and materials: NOT APPLICABLE. 
\end{itemize}

\end{document}